\documentclass[reqno]{amsart}
\usepackage{tkz-euclide}
\usepackage{hyperref}
\usepackage{amsmath}
\usepackage{amssymb} 
\usepackage{mathrsfs}
\usepackage{graphicx}
\usepackage{tikz}
\usepackage[font=footnotesize, labelfont=bf]{caption}
\usepackage{subcaption}
\usepackage{xcolor}  
\usepackage{amsthm}
\usepackage{float}
\usepackage{enumitem}
\usepackage[english]{babel}

\definecolor{LimeGreen}{cmyk}{0.50, 0.5, 1, 0}

\newcommand{\EEE}{\color{black}}

\newcommand{\dx}{\, {\rm d}x}
\newcommand{\dy}{\, {\rm d}y}

\newcommand{\dt}{\, {\rm d}t}
\newcommand{\dP}{\, {\rm d}P}
\newcommand{\e}{\varepsilon}

\DeclareMathOperator{\dist}{dist}
\newcommand{\Sph}{{\mathbb S}}

\newcommand{\ie}{{\it; i.e.}, }

\theoremstyle{plain}
\begingroup
\newtheorem{theorem}{Theorem}[section]
\newtheorem{lemma}[theorem]{Lemma}
\newtheorem{proposition}[theorem]{Proposition}
\newtheorem{corollary}[theorem]{Corollary}
\endgroup

\numberwithin{equation}{section}
\let \O=\Omega
\newcommand{\N}{\mathbb{N}}
\newcommand{\Z}{\mathbb{Z}}
\newcommand{\Q}{\mathbb{Q}}
\newcommand{\R}{\mathbb{R}}

\newcommand{\I}{\mathcal{I}}
\newcommand{\F}{\mathscr{F}}

\newcommand{\Adm}{\mathscr{A}}
\newcommand{\A}{\mathcal{A}}
\renewcommand{\S}{\mathbb{S}}
\renewcommand{\L}{\mathcal{L}}
\renewcommand{\H}{\mathcal{H}}

\newcommand{\dHn}{\, {\rm d}\H^{n-1}}

\newcommand{\loc}{\mathrm{loc}}
\newcommand{\T}{\mathcal{T}}
\newcommand{\x}{\times }
\newcommand{\m}{\mathbf{m}}

\renewcommand{\hom}{\mathrm{hom}}
\newcommand{\defas}{:=}
\newcommand{\wto}{\rightharpoonup}

\newcommand{\wcont}{\subset\subset}
\newcommand{\LtL}{L^0(\R^n;\R^m)\times L^0(\R^n)}

\newcommand{\om}{\omega}
\newcommand{\uu}{{\rm u}}
\newcommand{\vv}{{\rm v}}

\newenvironment{step}[1]{\medskip\textit{Step #1:}}{}

\theoremstyle{definition}
\begingroup
\newtheorem{definition}[theorem]{Definition}

\endgroup
\theoremstyle{remark}
\newtheorem{remark}[theorem]{Remark}
\renewcommand{\tilde}{\widetilde}
\newcommand{\sm}{\setminus}

\renewcommand{\d}{\, \mathrm{d}}
\usepackage[normalem]{ulem}

\title[Singularly-perturbed elliptic functionals]
{$\Gamma$-convergence and stochastic homogenisation of singularly-perturbed elliptic functionals}

\author[A. Bach]{Annika Bach}
\address[]{Zentrum Mathematik (M7) TU M\"unchen, Germany}
\email[Annika Bach]{annika.bach@ma.tum.de}
\author[R. Marziani]{Roberta Marziani}
\address[]{Angewandte Mathematik, WWU M\"unster, Germany}
\email[Roberta Marziani]{rmarzian@uni-muenster.de}
\author[C. I. Zeppieri]{Caterina Ida Zeppieri}
\address[]{Angewandte Mathematik, WWU M\"unster, Germany}
\email[Caterina Zeppieri]{caterina.zeppieri@uni-muenster.de}

\begin{document}


\begin{abstract}
We study the limit behaviour of singularly-perturbed elliptic functionals of the form
\[
\F_k(u,v)=\int_A v^2\,f_k(x,\nabla u)\dx+\frac{1}{\e_k}\int_A g_k(x,v,\e_k\nabla v)\dx\,,  
\]
where $u$ is a vector-valued Sobolev function, $v \in [0,1]$ a phase-field variable, and $\e_k>0$ a singular-perturbation parameter\ie $\e_k \to 0$, as $k\to +\infty$.

Under mild assumptions on the integrands $f_k$ and $g_k$, we show that if $f_k$ grows superlinearly in the gradient-variable, then the functionals $\F_k$ $\Gamma$-converge (up to subsequences) to a \emph{brittle} energy-functional\ie to a free-discontinuity functional
whose surface integrand does \emph{not} depend on the jump-amplitude of $u$. This result is achieved by providing explicit asymptotic formulas for the bulk and surface integrands which show, in particular, that volume and surface term in $\F_k$ \emph{decouple} in the limit.

The abstract $\Gamma$-convergence analysis is complemented by a stochastic homogenisation result for \emph{stationary random} integrands.     
\end{abstract}

\maketitle

{\small
\noindent 

\medskip

\noindent 
}

\section{Introduction}
\noindent Since the seminal work of Modica and Mortola \cite{Mo, MoMo} and of Ambrosio and Tortorelli \cite{AT90, AT92}, singularly-perturbed elliptic functionals have proven to be an effective tool to approximate free-discontinuity problems in manifold situations.   
Elliptic-functionals based approximations have been successfully used, \emph{e.g.}, for numerical simulations in imaging or brittle fracture (see, \emph{e.g.}, \cite{Bour99, BouFraMa, BouFraMa-1}), in cavitation problems \cite{DMCX, DMCX-16}, and to define a notion of regular evolution in fracture mechanics \cite{BaMi, Giac}, just to mention few examples. Besides being an approximation tool,   
these kinds of functionals are also commonly employed to model a range of phenomena where ``diffuse'' interfaces appear (see, \textit{e.g.}, \cite{Owen, OwSte, Ste, BaFo, Bou, FonTar, FonMan, CDMFL, CSZ2011}), or as instances of gradient damage models (see, \textit{e.g.}, \cite{PMM, DMI13, Iur}).

In this paper we study the $\Gamma$-convergence, as $k \to +\infty$, of general elliptic functionals of the form
\begin{equation}\label{intro:Fe} 
\F_k(u,v)=\int_A \psi(v)\,f_k(x,\nabla u)\dx+\frac{1}{\e_k}\int_A g_k(x,v,\e_k\nabla v)\dx\,,  
\end{equation}
where $\e_k \searrow 0$ is a singular-perturbation parameter. The set $A\subset \R^n$ is open bounded and with Lipschitz boundary, $u \colon A \to \R^m$ is a vectorial function, $v\colon A \to [0,1]$ is a phase-field variable, and $\psi \colon [0,1] \to [0,1]$ is an increasing and continuous function satisfying $\psi (0)=0$, $\psi(1)=1$, and $\psi(s)>0$ for $s>0$. For every $k\in \N$, the integrands $f_k \colon \R^n \times \R^{m\times n} \to [0,+\infty)$ and $g_k \colon \R^n \times [0,1] \times \R^n \to [0,+\infty)$ belong to suitable classes of functions denoted, respectively, by $\mathcal F$ and $\mathcal G$ (see Section \ref{subs:setting} for their definition). The requirement $(f_k) \subset \mathcal F$ and $(g_k) \subset \mathcal G$ in particular ensures the existence of an exponent $p>1$ such that for every $k\in \N$ and every $x\in\R^n$: 
\begin{equation}\label{intro-fk}
c_1 |\xi|^p \leq f_k(x,\xi)\leq c_2 |\xi|^p,              
\end{equation}
for every $\xi \in \R^{m\times n}$ and for some $0<c_1 \leq c_2<+\infty$, and  
\begin{equation}\label{intro-gk}
c_3 \big( |1-v|^p+|w|^p\big) \leq g_k(x,v,w)\leq c_4 \big(|1-v|^p +|w|^p\big),              
\end{equation}
for every $v\in [0,1]$ and $w\in \R^n$, and for some $0<c_3 \leq c_4<+\infty$. 
As a consequence, the functionals $\F_k$ are finite in $W^{1,p}(A;\R^m)\times W^{1,p}(A)$ and are bounded both from below and from above by Ambrosio-Tortorelli functionals of the form 
\begin{equation}\label{intro:AT}
AT_{k}(u,v)=\int_A \psi(v)\,|\nabla u|^p\dx+\int_A\bigg(\frac{(1-v)^p}{\e_k}+\e_k^{p-1}|\nabla v|^p\bigg)\dx\,. 
\end{equation}
Therefore if $(u_k,v_k) \subset W^{1,p}(A;\R^m)\times W^{1,p}(A)$ is a pair satisfying $\sup_{k}\F_k(u_k,v_k)<+\infty$, the lower bound on $\F_k$ immediately yields $v_k \to 1$ in $L^p(A)$, as $k \to +\infty$. On the other hand, $|\nabla u_k|$ can blow up in the regions where $v_k$ is asymptotically small, so that one expects a limit $u$ which may develop discontinuities. In \cite{AT90, AT92} Ambrosio and Tortorelli showed that functionals of type \eqref{intro:AT} provide a variational approximation, in the sense of $\Gamma$-convergence, of the free-discontinuity functional of Mumford-Shah type given by
\begin{equation}\label{intro:MSp}
\int_A |\nabla u|^p\dx+c_p\mathcal H^{n-1}(S_u)\,,
\end{equation}
where now the variable $u$ belongs to the space of \emph{generalised special functions of bounded variation} $GSBV^p(A;\R^m)$. As in the case of the Modica-Mortola approximation of the perimeter-functional \cite{Mo, MoMo}, the effect of the singular perturbation, $\e_k^{p-1}|\nabla v|^p$, in \eqref{intro:AT} is that of producing a transition layer around the discontinuity set of $u$, denoted by $S_u$. Similarly, the pre-factor $c_p>0$ is related to the cost of an optimal transition of the phase-field variable, now taking place between the value zero, where $\psi$ vanishes, and the value one.\ Another similarity shared by the Modica-Mortola and the Ambrosio-Tortorelli approximation is that they are substantially one-dimensional. Indeed, in both cases the $n$-dimensional analysis can be carried out by resorting to an integral-geometric argument, the so-called slicing procedure, which allows to reduce the general situation to the one-dimensional case. 

A relevant feature of the Ambrosio-Tortorelli approximation is that the ``regularised'' bulk and surface terms in \eqref{intro:AT} separately converge to their sharp-interface counterparts in \eqref{intro:MSp}. This kind of volume-surface decoupling can be also observed in a number of variants of \eqref{intro:AT}. Indeed, this is the case, \emph{e.g.}, of the anisotropic functionals analysed in \cite{Foc01}, of the phase-field approximation of brittle fracture in linearised elasticity in \cite{Cham}, of the second-order variants proposed in \cite{BEZ, Bach}, and of the finite-difference discretisation of the Ambrosio-Tortorelli functional on periodic \cite{BBZ} and on stochastic grids \cite{BCR19}. More specifically, in \cite{Foc01, Cham, BEZ, Bach} the volume-surface decoupling is obtained by means of general integral-geometric arguments which can be employed thanks to the specific form of the approximating functionals. On the other hand, in \cite{BBZ, BCR19} the interplay between the singular perturbation and the discretisation parameters makes for a subtle problem for which an \emph{ad hoc} proof is needed. In \cite{BBZ} this proof relies on an explicit geometric construction which, however, is feasible only in dimension $n=2$. In fact, more refined arguments are necessary to deal with the case $n\geq 3$, as shown in \cite{BCR19}. Namely, in \cite{BCR19} the limit volume-surface decoupling is achieved by resorting to a weighted co-area formula, which is reminiscent of a technique introduced by Ambrosio \cite{A94} (see also \cite{BDV96, Gia-Pon, CDMSZ19, Ruf}). This procedure allows to identify an asymptotically small region where the phase-field variable $v_k$ can be modified and set equal to zero while the corresponding  $u_k$ makes a steep transition between two constant values. In this way, a pair $(u_k,v_k)$ is obtained whose bulk energy vanishes while the surface energy does not essentially increase.

In the present paper we show that the volume-surface decoupling illustrated above takes place also for the general functionals $\F_k$, whose integrands $f_k$ and $g_k$ combine both a $k$ and an $x$ dependence, and satisfy \eqref{intro-fk} and \eqref{intro-gk}. Moreover, being the dependence on $x$ only measurable, the case of homogenisation is covered by our analysis as well, as shown in this paper. We remark that the general character of the problem does not allow us to use either the slicing or the blow-up method to establish a $\Gamma$-convergence result for $\F_k$, as it is instead customary for phase-field functionals of Ambrosio-Tortorelli type. Our approach is close in spirit to that of \cite{CDMSZ19} and combines the general localisation method of $\Gamma$-convergence \cite{DM, BDf} with a careful local analysis which eventually allows us to completely characterise the integrands of the $\Gamma$-limit thus proving, in particular, that volume and surface term do not interact in the limit. 

The volume-surface decoupling has been extensively analysed in the case of free-discontinuity functionals, starting with the seminal work \cite{A94}. It has then been proven that a decoupling takes place in the case of free-discontinuity functionals with periodically oscillating integrands \cite{BDV96}, for scale-dependent scalar brittle energies both in the continuous \cite{Gia-Pon} and in the discrete case \cite{Ruf}, for general vectorial scale-dependent free-discontinuity functionals \cite{CDMSZ19, CDMSZ19a}, and, more recently, also in the setting of linearised elasticity \cite{FPS}. The clear advantage of having such a decoupling is that limit volume and surface integrands can be determined independently from one another, by means of asymptotic formulas which are then easier to handle, \textit{e.g.}, computationally. Moreover in \cite{Gia-Pon} it is shown that the noninteraction between volume and surface is crucial to prove the stability of unilateral minimality properties in the study of crack-propagation in composite materials. The same applies to the case of the evolution considered in \cite{Giac}, where this feature plays a central role in proving that the regular quasistatic evolution for the Ambrosio-Tortorelli functional converges to a quasistatic evolution for brittle fracture in the sense of \cite{FraLa}. These considerations also motivate the analysis carried out in the present paper.

The main result of this paper is contained in Theorem \ref{thm:main-result} and is a $\Gamma$-convergence and an integral representation result for the $\Gamma$-limit. Namely, in Theorem \ref{thm:main-result} we show that if $f_k$ and $g_k$ satisfy rather mild assumptions, (see Subsection \ref{subs:setting} for the complete list of hypotheses) together with \eqref{intro-fk} and \eqref{intro-gk}, then (up to subsequences) the functionals $\F_k$ $\Gamma$-converge to a free-discontinuity functional of the form  
\begin{equation}\label{intro:F}
\F(u)=\int_A f_\infty(x,\nabla u)\dx +\int_{S_u\cap A} g_\infty(x,\nu_u)\d\H^{n-1},
\end{equation}
where now $u\in GSBV^p(A;\R^m)$ and $\nu_u$ denotes the generalised normal to $S_u$. 
We observe that the surface term in $\F$ is both inhomogeneous and anisotropic, however it does {\em not} depend on the jump-opening $[u]=u^+-u^-$; in other words, $\F$ is a so-called {\em brittle energy}. The form of the surface term in \eqref{intro:F} is one of the effects of the volume-surface limit decoupling mentioned above, which is apparent from the asymptotic formulas defining $f_\infty$ and $g_\infty$. In fact, in Theorem \ref{thm:main-result} we also provide formulas for $f_\infty$ and $g_\infty$. Namely, we prove that 
\begin{equation}\label{intro:f-infty} 
f_{\infty}(x,\xi)=\limsup_{\rho \to 0} \lim_{k \to +\infty} \frac{1}{\rho^n} \inf \int_{Q_\rho(x)} f_k(x,\nabla u)\dx       
\end{equation}
where the infimum in \eqref{intro:f-infty} is taken over all functions $u\in W^{1,p}(Q_\rho(x);\R^m)$ with $u(x)=\xi x$ near $\partial Q_\rho(x)$. The surface energy density is given instead by 
\begin{equation}\label{intro:g-infty}
g_\infty(x,\nu)= \limsup_{\rho \to 0} \lim_{k \to +\infty} \frac{1}{\rho^{n-1}} \inf \frac{1}{\e_k}\int_{Q^\nu_\rho(x)} g_k(x,v,\e_k \nabla v)\dx       
\end{equation}
where the cube $Q^\nu_\rho(x)$ is a suitable rotation of $Q_\rho(x)$ and the infimum in $v$ is taken in a $u$-dependent class of functions. More precisely, the infimum in \eqref{intro:g-infty} is taken among all $v \in W^{1,p}(Q^\nu_\rho(x))$, with $0\leq v \leq 1$, for which there exists $u \in W^{1,p}(Q^\nu_\rho(x);\R^m)$ such that $v \nabla u$ = 0 a.e.\ in $Q^\nu_\rho(x)$ and $(u,v)=(u^\nu_x,1)$ in $U \cap \{|(y-x)\cdot \nu|>\e_k\}$ where $U$ is a neighbourhood of $\partial Q^\nu_\rho(x)$, and $u^\nu_x$ is the jump function
given by
\begin{equation*}
u^{\nu}_{x}(y)=\begin{cases} e_1 & \text{if }\; (y-x) \cdot \nu \geq 0\,, 
\cr
0 & \text{if }\; (y-x) \cdot \nu < 0\,.
\end{cases}
\end{equation*}
In~\eqref{intro:g-infty} the boundary datum $(u^\nu_x,1)$ cannot be prescribed in the vicinity of $\{y \in \R^n \colon (y-x) \cdot \nu =0\}$ due to the discontinuity of $u^{\nu}_{x}$ and to the fact that $v$ must be equal to zero (and not to one) where $u$ jumps. However, this mixed Dirichlet-Neumann boundary condition can be replaced by a Dirichlet boundary condition prescribed on the whole boundary of $Q^\nu_{\rho}(x)$, up to replacing $u^{\nu}_{x}$ with a regularised counterpart defined, \emph{e.g.}, as in \ref{uv-bar-e}, Subsection \ref{subs:notation}.  
  
In view of the growth conditions \eqref{intro-fk} satisfied by $f_k$ and the properties of $\psi$, the constraint $v \nabla u=0$ satisfied a.e.\ in $Q^\nu_\rho(x)$ is equivalent to
\[
\int_{Q^\nu_\rho(x)} \psi(v)f_k(x,\nabla u)\dx=0,
\] 
which makes apparent why the bulk term in $\F_k$ does not contribute to $g_\infty$. We notice, however, that due to the nature of the problem, the variable $u$ must enter in the definition of $g_\infty$, so that in this case a decoupling is not to be intended as in the case of free-discontinuity functionals \cite{CDMSZ19}.   
 
To derive the formula for $f_\infty$ we follow a similar strategy as in \cite{BCS} and use the co-area formula in the Modica-Mortola term in \eqref{intro:Fe} to show that, in the set where $v$ is bounded away from zero, $\F_k$ behaves like a sequence of free-discontinuity functionals whose volume integrand is $f_k$. Then, we conclude by invoking the decoupling result for free-discontinuity functionals proven in \cite{CDMSZ19}. In fact, we notice that \eqref{intro:f-infty} coincides with the asymptotic formula for the limit volume integrand provided in \cite{CDMSZ19}.  The proof of \eqref{intro:g-infty} is more subtle and is substantially different, \emph{e.g.}, from that in \cite{BCR19}. Namely, to prove \eqref{intro:g-infty} we need to modify a sequence $(u_k, v_k)$ with bounded energy in the cube $Q^\nu_\rho(x)$ to get a new sequence with zero volume energy which can be used as a test in \eqref{intro:g-infty}, hence, in particular, the modification to $(u_k, v_k)$ shall not increase the surface energy. In the case of the discretised Ambrosio-Tortorelli functional considered in \cite{BCR19}, the discrete nature of the problem allows for a construction which is not feasible in a continuous setting. In our case, instead, we follow an argument which is close in spirit to a construction in \cite{Giac}. This argument amounts to partition the set where $\nabla u_k \neq 0$ and to use the bound on the energy to single out a set of the partition with small measure and small volume energy. Then in this set the function $v_k$ is modified by suitably interpolating between the value zero and two functions explicitly depending on $u_k$. The advantage of this interpolation is that it allows to easily estimate the increment in surface energy in terms of the volume energy and at the same time to define a test pair for \eqref{intro:g-infty}. Eventually, to prove that the increment in surface energy is asymptotically negligible we need to use that $p>1$.
We notice that the assumption $p>1$ is optimal in the sense that if $f_k$ is linear in the gradient variable\ie \eqref{intro-fk} holds with $p=1$, then it is well known \cite{ABS, Al-Fo} that the corresponding Ambrosio-Tortorelli functional $\Gamma$-converges to a free-discontinuity functional whose surface energy explicitly depends on $[u]$, this dependence being the result of a nontrivial limit volume-surface interaction.     

Our general analysis is then applied to study the homogenisation of damage models\ie to deal with the case of integrands $f_k$ and $g_k$ of type 
\begin{equation}\label{intro:hom}
f_k(x,\xi)= f\Big(\frac{x}{\e_k},\xi\Big)\; \text{ and } \; g_{k}(x,v,w)=g\Big(\frac{x}{\e_k},v, w\Big), 
\end{equation}
for some $f \in \mathcal F$ and $g\in \mathcal G$. More specifically, in Theorem \ref{thm:hom} we prove a homogenisation result for $\F_k$, with $f_k$ and $g_k$ as in \eqref{intro:hom}, \emph{without} requiring any spatial periodicity of the integrands, but rather assuming the existence and spatial homogeneity of the limit of certain scaled minimisation problems (cf.\ \eqref{eq:f-hom} and \eqref{eq:g-hom}). Eventually, we show that the assumptions of Theorem \ref{thm:hom} are satisfied, almost surely, in the case where the integrands $f$ and $g$ are \emph{stationary random variables} and derive the corresponding stochastic homogenisation result, Theorem \ref{thm:stoch_hom_2}. Thanks to the decoupling result, Theorem \ref{thm:main-result}, the stochastic homogenisation of the bulk term readily follows from \cite{DMM}. On the other hand, the treatment of the regularised surface term requires a new \emph{ad hoc} analysis which shares some similarities with that developed for random surface functionals \cite{BP, ACR11, CDMSZ19a}. We also mention here the recent paper \cite{Morfe} where the stochastic homogenisation of Modica-Mortola functionals with a stationary and ergodic gradient-perturbation is studied. 

To conclude we notice that our analysis also allows to deduce a $\Gamma$-convergence result for functionals with oscillating integrands of type \eqref{intro:hom} when the heterogeneity scale does not necessarily coincide with the Ambrosio-Tortorelli parameter $\e_k$, but is rather given by a different infinitesimal scale $\delta_k>0$. In this case, though, the asymptotic formulas provided by Theorem \ref{thm:main-result} would fully characterise the homogenised volume energy but not the surface energy. In fact, in this case a full characterisation of the homogenised surface integrand requires a further investigation which, in particular, shall distinguish between the regimes $\e_k \ll \delta_k$ and $\e_k \gg \delta_k$. A complete analysis of this type, in the spirit of \cite{AnBraCP}, goes beyond the purpose of the present paper and is instead the object of the ongoing work \cite{BEMZ21}.     

\medskip

\noindent {\bf Outline of the paper.} This paper is organised as follows. In Section \ref{sect:setting} we collect some notation used througout, introduce the mathematical setting and the functionals we are going to analyse, moreover we prove some preliminary results. In Section \ref{sect:main} we state the main results of the paper, namely, the $\Gamma$-convergence and integral representation result (Theorem \ref{thm:main-result}), a convergence result for some associated minimisation problems (Theorem \ref{t:con-min-pb}), and a homogenisation result without periodicity assumptions (Theorem \ref{thm:hom}). In Section \ref{sect:prop} we prove some properties satisfied by the limit volume and surface integrands (Proposition \ref{prop:f'-f''} and Proposition \ref{prop:g'-g''}). In Section \ref{s:G-convergence} we implement the localisation method of $\Gamma$-convergence proving, in particular, a fundamental estimate for the functionals $\F_k$ (Proposition \ref{prop:fund-est}) and a compactness and integral representation result for the $\Gamma$-limit of $\F_k$ (Theorem \ref{thm:int-rep}). In Section \ref{sect:bulk} we characterise the volume integrand of the $\Gamma$-limit (Proposition \ref{p:volume-term}) and in Section \ref{sect:surface} the surface integrand (Proposition \ref{p:surface-term}), thus fully achieving the proof of the main result, Theorem \ref{thm:main-result}. In Section \ref{sect:stochastic-homogenisation} we prove a stochastic homogenisation result for stationary random integrands (Theorem \ref{thm:stoch_hom_2}). Eventually in the Appendix we prove two technical lemmas which are used in Section \ref{sect:stochastic-homogenisation}.

%
%
\section{Setting of the problem and preliminaries}\label{sect:setting}

\noindent In this section we collect some notation, introduce the functionals we are going to study, and prove some preliminary results.     

\subsection{Notation}\label{subs:notation} The present subsection is devoted to the notation we employ throughout.  

\begin{enumerate}[label=(\alph*)]
\item $m\geq 1$ and $n\geq 2$ are fixed positive integers; we set $\R^m_0:=\R^m\setminus \{0\}$;
\item $\S^{n-1}\defas\{\nu=(\nu_1,\ldots,\nu_n)\in \R^n \colon \nu_1^2+\cdots+\nu_n^2=1\}$ and $\widehat{\S}^{n-1}_\pm\defas\{\nu \in\S^{n-1}\colon \pm\nu_{i(\nu)}>0\}$, where $i(\nu)\defas\max\{i\in\{1,\ldots,n\}\colon \nu_i\neq 0\}$;
\item $\mathcal L^n$ and and $\mathcal H^{n-1}$ denote, respectively, the Lebesgue measure and the $(n-1)$-dimensional Hausdorff measure on $\R^n$;
\item $\A$ denotes the collection of all open and bounded subsets of $\R^n$ with Lipschitz boundary. If $A,B \in \A$ by $A \subset \subset B$ we mean that $A$ is relatively compact in $B$;
\item $Q$ denotes the open unit cube in $\R^n$ with sides parallel to the coordinate axis, centred at the the origin; for $x\in \R^n$ and $r>0$ we set $Q_r(x):= rQ+x$. Moreover, $Q'$ denotes the open unit cube in $\R^{n-1}$ with sides parallel to the coordinate axis, centred at the origin, for every $r>0$ we set $Q_r'\defas r Q'$;
\item\label{Rn} for every $\nu\in \Sph^{n-1}$ let $R_\nu$ denote an orthogonal $(n\x n)$-matrix such that $R_\nu e_n=\nu$; we also assume that $R_{-\nu}Q=R_\nu Q$ for every $\nu \in \S^{n-1}$, $R_\nu\in\Q^{n\x n}$ if $\nu\in\S^{n-1}\cap\Q^n$, and that the restrictions of the map $\nu\mapsto R_\nu$ to $\widehat{\Sph}_{\pm}^{n-1}$ are continuous. For an explicit example of a map $\nu \mapsto R_\nu$ satisfying all these properties we refer the reader, \textit{e.g.}, to~\cite[Example A.1]{CDMSZ19};
\item for $x\in\R^n$, $r>0$, and $\nu\in\S^{n-1}$, we define $Q^\nu_r(x):=R_\nu Q_r(0)+x$. 
\item for $\xi\in \R^{m \x n}$ we let $u_\xi$ be the linear function whose gradient is equal to $\xi$\ie $u_\xi(x):=\xi x$, for every $x\in \R^n$;
\item\label{jump-fun} for $x\in \R^n$, $\zeta\in \R^m_0$, and $\nu \in \Sph^{n-1}$ we denote with $u_{x,\zeta}^{\nu}$ the piecewise constant function taking values $0,\zeta$ and jumping across the hyperplane $\Pi^\nu(x):=\{y\in \R^n \colon (y-x) \cdot \nu=0\}$\ie
\begin{equation*}
u^{\nu}_{x,\zeta}(y):=\begin{cases} \zeta & \text{if }\; (y-x) \cdot \nu \geq 0\,, 
\cr
0 & \text{if }\; (y-x) \cdot \nu < 0\,,
\end{cases}
\end{equation*}
when $\zeta=e_1$ we simply write $u_{x}^\nu$ in place of $u^{\nu}_{x,e_1}$;
\item\label{1dim-couple} let ${\rm u} \in C^1(\R)$, ${\rm v}\in C^1(\R)$, with $0\leq \vv \leq 1$, be one-dimensional functions satisfying the following two properties:
\smallskip

\begin{enumerate}[label= \roman*.]
\item $\vv\uu'\equiv 0$ in $\R$; 

\smallskip

\item $(\uu(t),\vv(t))=(\chi_{(0,+\infty)}(t),1)$ for $|t|>1$.
\end{enumerate}

\item\label{uv-bar} for $x\in \R^n$ and $\nu \in \Sph^{n-1}$ we set
\begin{equation*}
\bar u^\nu_{x} (y):= \uu ((y-x) \cdot \nu) e_1\,, \quad \bar v^\nu_x(y):= \vv ((y-x) \cdot \nu)\,;
\end{equation*}
\item\label{uv-bar-e} for $x\in\R^n$, $\nu\in\S^{n-1}$, $\zeta\in\R^m_0$ and $\e>0$ we set
\begin{equation*}
\bar u^\nu_{x,\zeta,\e}(y)\defas \uu \big(\tfrac{1}{\e}(y-x)\cdot\nu)\zeta\,,\quad\bar v^\nu_{x,\e}(y)\defas\vv\big(\tfrac{1}{\e}(y-x)\cdot\nu).
\end{equation*}
When $\zeta=e_1$ we simply write $\bar{u}_{x,\e}^\nu$ in place of $\bar{u}_{x,e_1,\e}^\nu$. We notice that in particular, $\bar{u}_{x,1}^\nu=\bar{u}_x^\nu$, $\bar{v}_{x,1}^\nu=\bar{v}_x^\nu$;
\item for a given topological space $X$, $\mathcal{B}(X)$ denotes the Borel $\sigma$- algebra on $X$. If $X=\R^d$, with $d\in \N$, $d\ge1$ we simply write $\mathcal{B}^d$ in place of $\mathcal B(\R^d)$. For $d=1$ we write $\mathcal B$. 
\end{enumerate}

\medskip

\noindent For every $\mathcal L^n$-measurable set $A\subset \R^n$ we define $L^0(A;\R^m)$ as the space of all $\R^m$-valued Lebesgue measurable functions. We endow $L^0(A;\R^m)$ with the topology of convergence in measure on bounded subsets of $A$ and  recall that this topology is both metrisable and separable.

Let $A\in \A$; in this paper we deal with the functional space $SBV(A;\R^m)$ (resp.\ $GSBV(A;\R^m)$) of special functions of bounded variation (resp.\ of generalised special functions of bounded variation) on $A$, for which we refer the reader to the monograph \cite{AFP}. 
Here we only recall that if $u\in SBV(A;\R^m)$ then its distributional derivative can be represented as 
\begin{equation}\label{c:SBV}
Du(B)=\int_B \nabla u \dx+\int_{B\cap S_u}[u]\otimes \nu_u \d\mathcal H^{n-1},
\end{equation}
for every $B \in \mathcal B^n$. In \eqref{c:SBV} $\nabla u$ denotes the approximate gradient of $u$ (which makes sense also for $u\in GSBV$), $S_u$ the set of approximate discontinuity points of $u$, $[u]:=u^+-u^-$ where $u^\pm$ are the one-sided approximate limit points of $u$ at $S_u$, and $\nu_u$ is the measure theoretic normal to $S_u$. 

Let $p>1$; we also consider	
\begin{equation*}
SBV^{p}(A;\R^m):= \{u \in SBV(A;\R^m)\colon \nabla u\in L^{p}(A;\R^{m\x n}) \text{ and } \mathcal{H}^{n-1}(S_{u})<+\infty\}\,,
\end{equation*}
and
\begin{equation*}
GSBV^{p}(A;\R^m):= \{u \in GSBV(A;\R^m)\colon \nabla u\in L^{p}(A;\R^{m\x n}) \text{ and } \mathcal{H}^{n-1}(S_{u})<+\infty\}\,.
\end{equation*}
We recall that $GSBV^p(A;\R^m)$ is a vector space; moreover, if $u\in GSBV^p(A;\R^m)$ then we have that $\phi(u) \in SBV^p(A;\R^m)\cap L^\infty(A;\R^m)$, for every $\phi \in C^1_c(\R^m;\R^m)$ (see \cite{DMFT}). 

\medskip

Throughout the paper $C$ denotes a strictly positive constant which may vary from line to line and within the same expression.

\subsection{Setting of the problem}\label{subs:setting}
Let $p \in (1,+\infty)$; let $c_1,c_2, c_3, c_4, L_1, L_2$ be given constants such that $0<c_1\leq c_2<+\infty$,  $0<c_3\leq c_4<+\infty$, $0<L_1,L_2<+\infty$.
Let $\mathcal{F}:=\mathcal{F}(p, c_1,c_2,L_1)$ denote the collection of all functions $f\colon \R^n\x\R^{m\x n}\to [0,+\infty)$ satisfying the following conditions:
\begin{enumerate}[label=($f\arabic*$)]
\item\label{hyp:meas-f} (measurability) $f$ is Borel measurable on $\R^n\x \R^{m\x n}$;
\item\label{hyp:lb-f} (lower bound) for every $x \in \R^n$ and every $\xi \in \R^{m\x n}$
\begin{equation*}
c_1 |\xi|^p \leq f(x,\xi)\, ;
\end{equation*}
\item\label{hyp:ub-f} (upper bound) for every $x \in \R^n$ and every $\xi \in \R^{m\x n}$
\begin{equation*}
f(x,\xi) \leq c_2|\xi|^p\, ;
\end{equation*}
\item\label{hyp:cont-xi-f} (continuity in $\xi$) for every $x \in \R^{n}$ we have
\begin{equation*}
|f(x,\xi_1)-f(x,\xi_2)| \leq L_1\big(1+|\xi_1|^{p-1}+|\xi_2|^{p-1}\big)|\xi_1-\xi_2|,
\end{equation*}
for every $\xi_1$, $\xi_2 \in \R^{m\x n}$;
\end{enumerate}
Moreover, $\mathcal{G}:=\mathcal{G}(p, c_3,c_4,L_2)$ denotes the collection of all functions $g\colon \R^n\x \R\x \R^n\to [0,+\infty)$ satisfying the following conditions:
\begin{enumerate}[label=($g\arabic*$)]
\item\label{hyp:meas-g} (measurability) $g$ is Borel measurable on $\R^n\x  \R \EEE \x \R^n$;
\item\label{hyp:lb-g} (lower bound) for every $x \in \R^n$, every $v\in  \R $, and every $w \in \R^{n}$
\begin{equation*}
c_3\big( |1-v|^p \EEE+|w|^p\big) \leq g(x,v,w)\, ;
\end{equation*}
\item\label{hyp:up-g} (upper bound) for every $x \in \R^n$, every $v\in  \R $, and every $w \in \R^{n}$
\begin{equation*}
g(x,v,w) \leq c_4( |1-v|^p \EEE+|w|^p)\, .
\end{equation*}
\item\label{hyp:cont-g} (continuity in $v$ and $w$) for every $x \in \R^n$ we have
\begin{equation*}
|g(x,v_1,w_1)-g(x,v_2,w_2)| \leq L_2\Big( \big(1+|v_1|^{p-1}+|v_2|^{p-1}\big)|v_1-v_2|+\big(1+|w_1|^{p-1}+|w_2|^{p-1}\big)|w_1-w_2|\Big)
\end{equation*}
for every $v_1$, $v_2\in  \R$ and every $w_1$, $w_2 \in \R^{n}$;
\item\label{hyp:mono-g} (monotonicity in $v$) for every $x\in\R^n$ and every $w\in\R^n$, $g(x,\cdot\, ,w)$ is decreasing on $ (-\infty, 1)$ and increasing on $[1,+\infty)$;
\item\label{hyp:min-g} (minimum in $w$) for every $x\in\R^n$ and every $v\in  \R$ it holds 
\begin{equation*}
g(x,v,0)\leq g(x,v,w)
\end{equation*}
for every $w\in\R^n$.
\end{enumerate}
\begin{remark}\label{rem:trivial}
Let  $x \in \R^n$; we notice that gathering \ref{hyp:lb-f} and \ref{hyp:ub-f} readily implies that
\begin{equation}\label{f-value-at-0}
 f(x,\xi) = 0\; \text{ if and only if }\; \xi=0\,. \EEE
\end{equation} 
Moreover, from~\ref{hyp:lb-g} and~\ref{hyp:up-g} we deduce that 
\begin{equation}\label{g-value-at-0}
 g(x,v,w)=0\; \text{ if and only if }\; (v,w)=(1,0)\,. \EEE
\end{equation}
  \end{remark}
Let $\psi \colon  \R \to [0,+\infty)$ be continuous, decreasing on $(-\infty,0]$, increasing on  $(0,+\infty)$, such that $\psi(1)=1$, and $\psi(v)=0$ iff $v=0$. 

For $k\in \N$ let $(f_k) \subset \mathcal F$ and $(g_k)\subset \mathcal G$ and let $(\e_k)$ be a decreasing sequence of strictly positive real numbers converging to zero, as $k \to +\infty$.

We consider the sequence of elliptic functionals $\F_k \colon L^0(\R^n;\R^m) \times L^0(\R^n) \times \A \longrightarrow [0,+\infty]$ defined by 
\begin{align}\label{F_e}
\F_k(u,v, A)\defas
\begin{cases}
\displaystyle\int_A \psi(v)\,f_k(x,\nabla u)\dx+\frac{1}{\e_k}\int_A g_k(x,v,\e_k\nabla v)\dx &\text{if}\ (u,v)\in W^{1,p}(A;\R^m)\x W^{1,p}(A)\, ,\\ &0\leq v\leq 1 \, ,\\
+\infty &\text{otherwise}\,. 
\end{cases}
\end{align}
\begin{remark}\label{rem:bounds-Fe}
Assumptions~\ref{hyp:lb-f}-\ref{hyp:ub-f} and \ref{hyp:lb-g}-\ref{hyp:up-g} imply that for every $A\in\A$ and every $(u,v)\in W^{1,p}(A;\R^m)\times W^{1,p}(A)$, $0\leq v\leq 1$ it holds
\begin{equation}\label{est:bounds-Fe}
\begin{split}
c_1\int_A \psi(v)|\nabla u|^p\dx &+c_2\int_A\bigg(\frac{(1-v)^p}{\e_k}+\e_k^{p-1}|\nabla v|^p\bigg)\dx \leq\F_k(u,v,A)\\
&\leq c_3\int_A \psi(v)|\nabla u|^p\dx+c_4\int_A \bigg(\frac{(1-v)^p}{\e_k}+\e_k^{p-1}|\nabla v|^p\bigg)\dx\,;
\end{split}
\end{equation}
that is, up to a multiplicative constant, the functionals $\F_k$ are bounded from below and from above by the Ambrosio-Tortorelli functionals 
\begin{equation*}
AT_{k}(u,v)\defas\int_A \psi(v)\,|\nabla u|^p\dx+\int_A\bigg(\frac{(1-v)^p}{\e_k}+\e_k^{p-1}|\nabla v|^p\bigg)\dx\,. 
\end{equation*} 
\end{remark}
\begin{remark}\label{rem:continuity}
For later use, we notice that the assumptions on $\psi$, $f_k$ and $g_k$ ensure that for every $A\in\A$ the functionals $\F_k(\,\cdot\,,\,\cdot\,,A)$ are continuous in the strong $W^{1,p}(A;\R^m)\times W^{1,p}(A)$ topology.
\end{remark}
 For every $A\in \A$, $u\in L^0(\R^n;\R^m)$ and $v \in L^0(\R^n)$ it is also convenient to write 
\[\F_k(u,v,A)=\F_k^b(u,v,A)+\F_k^s(v,A)\,,\]
where $\F_k^b \colon L^0(\R^n;\R^m) \times L^0(\R^n) \times \A \longrightarrow [0,+\infty]$ and  $\F_k^s \colon L^0(\R^n) \times \A \longrightarrow [0,+\infty]$ denote the bulk and the surface part of $\F_k$, respectively\ie
\begin{equation}\label{bulk}
\F_k^b(u,v,A)\defas\begin{cases}
\displaystyle\int_A \psi(v)f_k(x,\nabla u)\dx &\text{if}\ (u,v)\in W^{1,p}(A;\R^m)\x W^{1,p}(A)\, ,\ 0\leq v\leq 1\, ,\\
+\infty &\text{otherwise}\,
\end{cases}
\end{equation}
and
\begin{equation}\label{surface}
\F_k^s(v,A)\defas\begin{cases}
\displaystyle\frac{1}{\e_k}\int_A g_k(x,v,\e_k\nabla v)\dx &\text{if}\ v\in W^{1,p}(A)\, ,\ 0\leq v\leq 1\, ,\\
+\infty &\text{otherwise}.
\end{cases} 
\end{equation}
For $\rho> 2\e_k$, $x\in \R^n$, $\xi \in \R^{m \times n}$, and $\nu\in \Sph^{n-1}$ we consider the two following minimisation problems 
\begin{align}\label{eq:mb}
\m_{k}^b(u_\xi,Q_\rho(x))\defas\inf\{\F_{k}^b(u,1,Q_\rho(x))\colon u\in W^{1,p}(Q_\rho(x);\R^m)\, ,\ u=u_\xi\ \text{near}\ \partial Q_\rho(x)\}
\end{align}
and 
\begin{equation}\label{eq:ms}
\m_{k, {\rm N}}^s(u_{x}^\nu,Q_\rho^\nu(x))\defas\inf\{\F_k^s(v,Q_\rho^\nu(x))\colon v\in \Adm_{\e_k,\rho}(x,\nu)\}\,,
\end{equation}
where 
\begin{multline}\label{c:adm-e}
\Adm_{\e_k,\rho}(x,\nu)\defas\big\{v\in W^{1,p}(Q_\rho^\nu(x)),\; 0\leq v\leq 1 \colon \, \exists\, u\in W^{1,p}(Q_\rho^\nu(x);\R^m)\ \text{ with } v \,\nabla u=0\ \text{a.e. in}\  Q_\rho^\nu(x)\\[4pt]
\text{and}\ (u,v)=(u_{x}^\nu,1)\ \text{near}\ \partial Q_\rho^\nu(x)\ \text{in}\ \{|(y-x)\cdot\nu|>\e_k\} \big\}\,.
\end{multline}
We observe that in \eqref{eq:mb} by ``$u=u_\xi\ \text{near}\ \partial Q_\rho(x)$'' we mean that the boundary datum is attained in a neighbourhood of $\partial Q_\rho(x)$. Whereas in \eqref{c:adm-e} the boundary datum is prescribed only in 
\[
U\cap \{y\in \R^n \colon |(y-x)\cdot\nu|>\e_k\}, 
\]
for some neighbourhood $U$ of $\partial Q_\rho^\nu(x)$. 
\begin{remark}\label{rem:admissible-competitor}
Clearly, the class of competitors $\Adm_{\e_k,\rho}(x,\nu)$ is nonempty. Indeed, the pair $(\bar{u}_{x,\e_k}^\nu,\bar{v}_{x,\e_k}^\nu)$ defined as in~\ref{uv-bar-e}, with $\e=\e_k$, satisfies both
\begin{equation*}
(\bar{u}_{x,\e_k}^\nu,\bar{v}_{x,\e_k}^\nu)=(u_x^\nu,1)\; \text{ in }\; \{|(y-x)\cdot\nu|>\e_k\}\quad\text{and}\quad \bar v_{x,\e_k}^\nu \nabla \bar{u}_{x,\e_k}^\nu\equiv 0\,.
\end{equation*}
Thus the restriction of $\bar v_{x,\e_k}^\nu$ to $Q_\rho^\nu(x)$ belongs to $\Adm_{\e_k,\rho}(x,\nu)$.

\medskip

In view of \eqref{f-value-at-0} and of the properties satisfied by $\psi$, we also observe that in \eqref{c:adm-e} the constraint 
\[
v\, \nabla u=0 \; \text{ a.e.\ in }\; Q_\rho^\nu(x)
\]
can be equivalently replaced by
\[
\F_k^b(u,v,Q_\rho^\nu(x))=0\,. 
\]
Hence, in particular, for every $x\in\R^n$, $\zeta\in\R^m_0$, $\nu\in\S^{n-1}$, $\rho>2\e_k$, and $k\in\N$ we have
\begin{equation}\label{prop:barue-barve}
\F_k(\bar{u}_{x,\zeta,\e_k}^\nu,\bar{v}_{x,\e_k}^\nu,Q_\rho^\nu(x))=\F_k^s(\bar{v}_{x,\e_k}^\nu,Q_\rho^\nu(x))\,.
\end{equation}
\end{remark}

\medskip

Finally, for every $x \in \R^{n}$ and every $\xi\in \R^{m\x n}$ we define
\begin{equation}\label{f'}
f'(x,\xi):=\limsup_{\rho\to 0}\frac{1}{\rho^n}\liminf_{k\to +\infty}\m_{k}^b(u_\xi,Q_\rho(x))\, ,
\end{equation}
\begin{equation}\label{f''}
f''(x,\xi):=\limsup_{\rho\to 0}\frac{1}{\rho^n}\limsup_{k\to +\infty}\m_{k}^b(u_\xi,Q_\rho(x))\,,
\end{equation}
while for every $x \in \R^n$ and every $\nu\in \mathbb S^{n-1}$ we set
\begin{equation}\label{g'}
g'(x,\nu):=\limsup_{\rho\to 0}\frac{1}{\rho^{n-1}}\liminf_{k\to +\infty}\m_{k, {\rm N}}^s(u_{x}^\nu,Q^\nu_\rho(x))\, ,
\end{equation}
\begin{equation}\label{g''}
g''(x,\nu):=\limsup_{\rho\to 0}\frac{1}{\rho^{n-1}}\limsup_{k\to +\infty}\m_{k, {\rm N}}^s(u_{x}^\nu,Q^\nu_\rho(x))\,.
\end{equation}

\subsection{Equivalent formulas for $g'$ and $g''$}

For later use, in Proposition \ref{p:equivalence-N-D} below, we prove that $g'$ and $g''$ can be equivalently defined by replacing the boundary conditions in~\eqref{g'}--\eqref{g''} with suitable Dirichlet boundary conditions on the whole boundary of $Q^\nu_\rho(x)$. More precisely, we consider the minimum values defined as follows: For every $x\in\R^n$, $\nu\in\Sph^{n-1}$, and $A\in\A$ we set
\begin{equation}\label{eq:ms-D}
\m_k^s(\bar u_{x,\e_k}^\nu,A)\defas\inf\{\F_k^s(v,A)\colon v\in \Adm(\bar{u}_{x,\e_k}^\nu,A)\}\,,
\end{equation}
where
\begin{multline}\label{c:adm-e-bar}
\Adm(\bar{u}_{x,\e_k}^\nu,A)\defas\{v\in W^{1,p}(A),\, 0\leq v\leq 1\colon\exists\, u\in W^{1,p}(A;\R^m)\ \text{with}\ v\nabla u=0\ \text{a.e.\ in}\ A\\[4pt]
\text{and}\ (u,v)=(\bar{u}_{x,\e_k}^\nu,\bar{v}_{x,\e_k}^\nu)\ \text{near}\ \partial A\}\,,
\end{multline}
with $(\bar{u}_{x,\e_k}^\nu,\bar{v}_{x,\e_k}^\nu)$ as in~\ref{uv-bar-e}.
\begin{remark}
Let $A\in \A$ be such that $A=A'\x I$ with $A'\subset\R^{n-1}$ open and bounded and $I\subset\R$ open interval.
Let $\nu\in \S^{n-1}$ and set $ A_{\nu}:=R_\nu A$, with $R_\nu$ as in \ref{Rn}. For every $k \in \N$ we have 
\begin{equation}\label{c:needs-name}
\int_{A_\nu}\Bigg(\frac{(1-\bar{v}_{x,\e_k}^\nu(y))^p}{\e_k}+\e_k^{p-1}|\nabla\bar{v}_{x,\e_k}^\nu(y)|^p\Bigg)\dy\leq \int_{A'}\int_
\R\big((1- \vv(t))^p+|\vv'(t)|^p\big)\dt\dy'\le C_{\vv}\L^{n-1}(A')\,,
\end{equation} 
where 
\[
C_{\vv}:=\int_\R\big((1- \vv(t))^p+|\vv'(t)|^p\big)\dt=\int_{-1}^1\big((1- \vv(t))^p+|\vv'(t)|^p\big)\dt <+\infty\,.
\]
In particular from \ref{hyp:up-g}, Remark~\ref{rem:admissible-competitor}, and \eqref{c:needs-name} we infer
\begin{equation}\label{1dim-energy-bis}
	\m_k^s(\bar u_{x,\e_k}^\nu,A_\nu)\le \F_k^s(\bar{v}_{x,\e_k}^\nu,A_\nu)\leq c_4C_{\vv}\L^{n-1}(A')\,.
\end{equation}
\end{remark}
We are now in a position to prove the following equivalent formulation for $g'$ and $g''$. We observe that the most delicate part in the proof of this result is to show that a suitable Dirichlet boundary datum can be prescribed on the whole $\partial Q_\rho^\nu(x)$ while keeping the nonconvex constraint $v\,\nabla u=0$ a.e.\ in $Q_\rho^\nu(x)$.  
\begin{proposition}\label{p:equivalence-N-D}
For every $\rho>2\e_k$, $x\in\R^n$, and $\nu\in\S^{n-1}$ let $\m_k^s(\bar{u}_{x,\e_k}^\nu,Q_\rho^\nu(x))$ be as in~\eqref{eq:ms-D} with $A=Q_\rho^\nu(x)$. Then we have
\begin{align*}
g'(x,\nu) &=\limsup_{\rho\to 0}\frac{1}{\rho^{n-1}}\liminf_{k\to+\infty}\m_k^s(\bar{u}_{x,\e_k}^\nu,Q_\rho^\nu(x))\,,\\
g''(x,\nu) &=\limsup_{\rho\to 0}\frac{1}{\rho^{n-1}}\limsup_{k\to+\infty}\m_k^s(\bar{u}_{x,\e_k}^\nu,Q_\rho^\nu(x))\,,
\end{align*}
where $g'$, $g''$ are as in~\eqref{g'} and~\eqref{g''}, respectively.
\end{proposition}
\begin{proof}
We only prove the equality for $g'$, the proof of the equality for $g''$ being analogous. Let $x\in\R^n$ and $\nu\in \S^{n-1}$ be fixed and set 
\[\underline{g}'(x,\nu)\defas\limsup_{\rho\to 0}\frac{1}{\rho^{n-1}}\liminf_{k\to+\infty}\m_k^s(\bar{u}_{x,\e_k}^\nu,Q_\rho^\nu(x))\,.\]
In view of Remark~\ref{rem:admissible-competitor} we readily have 
	\begin{equation}\label{c:prima}
	g'(x,\nu)\leq\underline{g}'(x,\nu)\,,
	\end{equation}	
for every $x\in\R^n$ and every $\nu\in \S^{n-1}$.
To prove the opposite inequality, let $\rho>0$ and $\alpha\in(0,1)$ be fixed and let $\bar k\in \N$ be such that $\e_k<\frac{\alpha\rho}{2}$, for every $k\geq \bar k$. Let moreover $v_k\in \Adm_{\e_k,\rho}(x,\nu)$ be such that
	\begin{equation}\label{est:energy-ms-per}
	\F_k^s(v_k,Q_{\rho}^\nu(x))\leq\m^s_{k,\rm N}(u_{x}^\nu,Q_\rho^\nu(x))+\rho^n\,.
	\end{equation}
	Then, there exists $u_k\in W^{1,p}(Q_\rho^\nu(x);\R^m)$ satisfying $v_k \nabla u_k=0$ a.e.\ in $Q_{\rho}^\nu(x)$ and
	\begin{equation}\label{bc}
	(u_k,v_k)=(u_{x}^\nu,1)=(\bar u_{x,\e_k}^\nu,\bar v_{x,\e_k}^\nu)\quad\text{in}\quad U_k\cap\{|(y-x)\cdot\nu|>\e_k\},
	\end{equation}
	where $U_k$ is a neighbourhood of $\partial Q_\rho^\nu(x)$.
	
Starting from $(u_k,v_k) \in  W^{1,p}(Q_\rho^\nu(x);\R^m) \times  W^{1,p}(Q_\rho^\nu(x))$ we are now going to define a new pair $(\tilde u_k, \tilde v_k) \in W^{1,p}(Q_{(1+\alpha)\rho}^\nu(x);\R^m) \times  W^{1,p}(Q_{(1+\alpha)\rho}^\nu(x))$ with $\tilde v_k\in\Adm(\bar{u}_{x,\e_k}^\nu,Q_{(1+\alpha)\rho}^\nu(x))$. Moreover, we will do this in a way such that $\F_k^s(\tilde v_k,Q_{(1+\alpha)\rho}^\nu(x))$ will be bounded from above by $\F_k^s(v_k,Q_\rho^\nu(x))$ and hence, thanks to \eqref{est:energy-ms-per}, by $\m^s_{k,\rm N}(u_{x}^\nu,Q_\rho^\nu(x))$. 

To this end, let $\beta_k\in(0,2\e_k)\subset(0,\rho)$ be such that $Q_\rho^\nu(x)\sm\overline{Q}_{\rho-\beta_k}^\nu(x)\subset U_k$ and set
	\begin{equation*}
	R_k\defas R_\nu\Big(\big(Q_{\rho}' \sm\overline{Q'}_{\!\!\rho-\beta_k}\big)\x (-\e_k,\e_k)\Big)+x\,, 
	\end{equation*}
	where $R_\nu$ is as in~\ref{Rn}. By construction we have that
	\begin{equation}\label{property:R_k}
	\big(Q_{\rho}^\nu(x)\sm\overline{Q}_{\rho-\beta_k}^\nu(x)\big)\sm\overline{R}_k\subset U_k\cap\{|(y-x)\cdot\nu|>\e_k\}\,
	\end{equation}
	(see Figure~\ref{fig:ModificationBC}). 		
	Now let $\varphi_k \in C_c^1(Q_{\rho}')$ be a cut-off function between $Q'_{\rho-\beta_k}$ and $Q'_{\rho}$\ie $0\leq \varphi_k \leq 1$, and 
  $\varphi_k\equiv 1$ on $Q_{\rho-\beta_k}'$. Eventually, for $y=(y',y_n)\in Q_{(1+\alpha)\rho}^\nu(x)$ we define the pair
  $(\tilde u_k,\tilde v_k)$ by setting
	\begin{equation*}
		\tilde u_k(y):=\varphi_k\big((R_\nu^T(y-x))'\big)u_k(y)+\Big(1-\varphi_k\big((R_\nu^T(y-x))'\big)\Big)\bar u_{x,\e_k}^\nu(y)
		\end{equation*}
		and 
		\begin{equation*}
		\tilde v_k(y):=\begin{cases}	\min\{v_k(y),d_k(y)\}&\text{in }\; Q_{\rho}^\nu(x)\,,\cr 
		\min\{\bar v_{x,\e_k}^\nu(y),d_k(y)\}&\text{in }\; Q_{(1+\alpha)\rho}^\nu(x)\sm\overline{Q}_{\rho}^\nu(x)\,,
	\end{cases}
	\end{equation*}
	where $d_k(y)\defas\frac{1}{\e_k}\dist(y,R_k)$.
	\begin{figure}[t]
	\centering
	\def\svgwidth{.4\textwidth}
	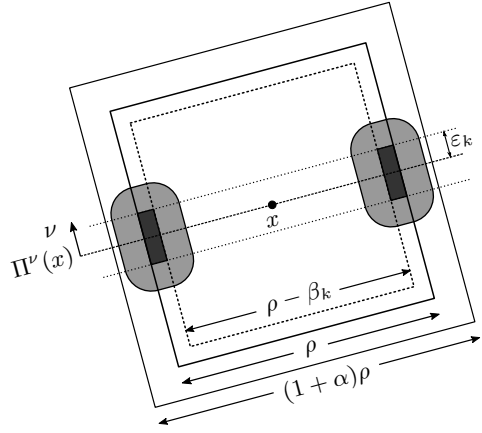
	\caption{The cubes $Q_{\rho-\beta_k}^\nu(x)$, $Q_\rho^\nu(x)$, $Q_{(1+\alpha)\rho}^\nu(x)$ and in gray the sets $R_k$ (dark gray) and $\{d_k<1\}$ (light gray).}
	\label{fig:ModificationBC}
	\end{figure}
	Clearly $\tilde u_k \in W^{1,p}(Q_{(1+\alpha)\rho}^\nu(x);\R^m)$, moreover, thanks to~\eqref{bc} the function $\tilde{v}_k$ belongs to $W^{1,p}(Q_{(1+\alpha)\rho}^\nu(x))$. Furthermore, it holds $\tilde{v}_k\,\nabla\tilde{u}_k=0$ a.e.\ in $Q_{(1+\alpha)\rho}^\nu(x)$. Indeed, $\tilde{v}_k|\nabla\tilde{u}_k|\leq v_k|\nabla u_k|=0$ a.e.\ in $Q_{\rho-\beta_k}^\nu(x)$; similarly in $Q_{(1+\alpha)\rho}^\nu(x)\sm\overline{Q}_{\rho}^\nu(x)$ there holds $\tilde{v}_k|\nabla\tilde{u}_k|\leq \bar v_{x,\e_k}^\nu|\nabla\bar u_{x,\e_k}^\nu|=0$. Finally, thanks to~\eqref{bc} and~\eqref{property:R_k}, in $Q_{\rho}^\nu(x)\sm\overline{Q}_{\rho-\beta_k}^\nu(x)\sm\overline{R}_k$ we have $\nabla\tilde{u}_k=\nabla\bar{u}_{x,\e_k}^\nu=0$, while, by definition, $\tilde{v}_k=0$ in $R_k$.
	
Since moreover $d_k>1$ in $Q_{(1+\alpha)\rho}^\nu(x)\sm\overline{Q}_{\rho+2\e_k}^\nu(x)$, we also have $(\tilde u_k,\tilde v_k)=(\bar u_{x,\e_k}^\nu,\bar v_{x,\e_k}^\nu)$ in $Q_{(1+\alpha)\rho}^\nu(x)\sm\overline{Q}_{\rho+2\e_k}^\nu(x)$, and hence
$\tilde{v}_k\in\Adm(\bar{u}_{x,\e_k}^\nu,Q_{(1+\alpha)\rho}^\nu(x))$ for $k\geq\bar{k}$.

Then, to conclude it only remains to estimate $\F_k^s(\tilde v_k,Q_{(1+\alpha)\rho}^\nu(x))$. To this end, we start noticing that $\tilde{v}_k=v_k$ in $Q_\rho^\nu(x)\cap\{d_k>1\}$, $\tilde{v}_k=\bar{v}_{x,\e_k}^\nu$ in $(Q_{(1+\alpha)\rho}^\nu(x)\sm\overline{Q}_\rho^\nu(x))\cap\{d_k>1\}$, while in $Q_\rho^\nu(x)\cap\{d_k<1\}$ we have 
	\begin{align*}
	\frac{1}{\e_k}g_k(y,\tilde{v}_k,\e_k\nabla\tilde{v}_k)&\leq\frac{1}{\e_k}g_k(y,v_k,\e_k\nabla v_k)+\frac{1}{\e_k}g_k(x,d_k,\e_k\nabla d_k)\leq\frac{1}{\e_k}g_k(y,v_k,\e_k\nabla v_k)+\frac{2c_4}{\e_k}\,,
	\end{align*}
	where the last estimate follows from~\ref{hyp:up-g}. Similarly, in $(Q_{(1+\alpha)\rho}^\nu(x)\sm\overline{Q}_\rho^\nu(x))\cap\{d_k<1\}$ we have
	\begin{equation*}
	\frac{1}{\e_k}g_k(y,\tilde{v}_k,\e_k\nabla\tilde{v}_k)\leq\frac{1}{\e_k}g_k(y,\bar{v}_{x,\e_k}^\nu,\e_k\nabla\bar{v}_{x,\e_k}^\nu)+\frac{2c_4}{\e_k}.
	\end{equation*}
	Thus, from~\eqref{g-value-at-0} and~\eqref{1dim-energy-bis} we infer
	\begin{equation}\label{c:natale-final}
	\begin{split}
		\F_k^s(\tilde v_k,Q_{(1+\alpha)\rho}^\nu(x))&\le \F_k^s(v_k,Q_\rho^\nu(x))+ \F_k^s(\bar v_{x,\e_k}^\nu,Q_{(1+\alpha)\rho}^\nu(x)\sm\overline{Q}_{\rho}^\nu(x))+ \frac{2c_4}{\e_k}\L^n(\{d_k<1\})\\[4pt]
		&\le \F_k^s(v_k,Q_\rho^\nu(x))+c_4C_\vv \mathcal L^{n-1}(Q'_{(1+\alpha)\rho}\sm\overline{Q'_{\rho}})+C(\beta_k+\e_k)\rho^{n-2}. 
		\end{split}
	\end{equation}
Then, thanks to~\eqref{est:energy-ms-per}, 
dividing both sides of \eqref{c:natale-final} by $((1+\alpha)\rho)^{n-1}$, since $\beta_k<2\e_k$ we obtain
	\begin{equation}\label{c:seconda}
		\frac{\m_k^s(\bar u_{x,\e_k}^\nu,Q_{(1+\alpha)\rho}^\nu(x))}{((1+\alpha)\rho)^{n-1}}\le \frac{1}{(1+\alpha)^{n-1}}\bigg(\frac{	\m_{k,\rm N}^s(u_{x}^\nu,Q_{\rho}^\nu(x))}{\rho^{n-1}}+\rho+ c_4C_{\vv}((1+\alpha)^{n-1}-1)+\frac{C\e_k}{\rho}\bigg)\,.
	\end{equation}
	Eventually, from~\eqref{c:seconda} passing first to the liminf as $k \to +\infty$ and then to the limsup as $\rho \to 0$ we deduce
	\begin{equation*}
	(1+\alpha)^{n-1}\underline{g}'(x,\nu)\leq g'(x,\nu)+c_4C_\vv((1+\alpha)^{n-1}-1)\,,
	\end{equation*}
	which, together with~\eqref{c:prima}, thanks to the arbitrariness of $\alpha>0$ yields $\underline{g}'(x,\nu)= g'(x,\nu)$.
\end{proof}
%
%
\section{Statements of the main results}\label{sect:main}
\noindent In this section we state the main results of this paper, namely, a $\Gamma$-convergence and integral representation result (Theorem \ref{thm:main-result}), a converge result for some associated minimisation problems (Theorem \ref{t:con-min-pb}), and a homogenisation result without periodicity assumptions (Theorem \ref{thm:hom}).     

\subsection{$\Gamma$-convergence}
The following result asserts that, up to subsequences, the functionals $\F_k$ $\Gamma$-converge to an integral functional of free-discontinuity type. Furthermore, it provides asymptotic cell formulas for the volume and surface limit integrands. These asymptotic cell formulas show, in particular, that volume and surface term decouple in the limit.   

\begin{theorem}[$\Gamma$-convergence]\label{thm:main-result}
Let $(f_k) \subset \mathcal F$, $(g_k) \subset \mathcal G$ and let $\F_k$ be the functionals as in \eqref{F_e}. Then there exists a subsequence, not relabelled, such that for every $A \in \A$ the functionals $\F_{k}(\cdot\,, \cdot\,, A)$ $\Gamma$-converge in  $L^0(\R^n;\R^m)\times L^0(\R^n)$ to $\F(\cdot\,,\cdot\,,A)$ with $\F \colon L^0(\R^n;\R^m) \times L^0(\R^n) \times \A \longrightarrow [0,+\infty]$ given by
\begin{equation*}
\F(u, v, A)\defas
\begin{cases}
\displaystyle\int_A f_\infty(x,\nabla u)\dx+\int_{S_u \cap A} g_\infty(x,\nu_u)\dHn &\text{if}\ u \in GSBV^p(A;\R^m)\, ,\\
& v= 1\, \text{a.e.\ in } A\, ,\\[4pt]
+\infty &\text{otherwise}\,,
\end{cases}
\end{equation*}
where $f_\infty \colon \R^n \times \R^{m \times n} \to [0,+\infty)$ and $g_\infty \colon \R^n \times \Sph^{n-1} \to [0,+\infty)$ are Borel functions. Moreover, it holds:
\begin{enumerate}[label=\roman*.]
\item for a.e.\ $x \in \R^n$ and for every $\xi \in \R^{m\x n}$  
\begin{equation*}
f_\infty(x,\xi)= f'(x,\xi) = f''(x,\xi)\, ,
\end{equation*}
with $f'$, $f''$ as in \eqref{f'} and \eqref{f''}, respectively;
\item for every $x\in\R^n$ and every $\nu\in\Sph^{n-1}$
\begin{equation*}
g_\infty(x,\nu)= g'(x,\nu) = g''(x,\nu)\, ,  
\end{equation*}
with $g'$, $g''$ as in \eqref{g'} and \eqref{g''}, respectively. 
\end{enumerate}
\end{theorem}

\begin{remark}
The choice of considering functionals $\F_{k}$ which are finite when the variable $v$ satisfies the bounds $0\leq v \leq 1$ is a choice of convenience and it is not restrictive. Indeed, in view of \ref{hyp:mono-g} and \ref{hyp:min-g} and of the properties of $\psi$, the functionals $\F_{k}$ decrease under the transformation $v \to \min\{\max\{v,0\},1\}$. Hence a $\Gamma$-convergence result for functionals $\F_k$ defined on functions $v$ with values in $\R$ can be easily deduced from Theorem \ref{thm:main-result}.
\end{remark}
\EEE
The proof of Theorem \ref{thm:main-result} will be achieved in four main steps which are addressed in Sections \ref{sect:prop}, \ref{s:G-convergence}, \ref{sect:bulk}, and \ref{sect:surface}, respectively. Firstly, we show that the functions $f', f'', g'$, and $g''$ satisfy a number of properties and, in particular, they are Borel measurable (see Proposition \ref{prop:f'-f''} and Proposition \ref{prop:g'-g''}). In the second step, we prove the existence of a sequence $(k_j)$, with $k_j \to +\infty$ as $j\to +\infty$, such that for every $A\in \mathcal A$ the corresponding functionals  $\F_{k_j}(\,\cdot\,,\, \cdot\,, A)$ $\Gamma$-converge to a free-discontinuity functional which is finite in $GSBV^p(A;\R^m)\times \{1\}$ and is of the form 
\begin{equation}\label{F-hat}
\int_A \hat f(x,\nabla u)\dx+\int_{S_u}\hat g(x,[u],\nu_u)\dHn\,, 
\end{equation}
for some Borel functions $\hat f$ and $\hat g$ (see Theorem \ref{thm:int-rep}). 

In the third step we identify $\hat f$ by showing that it is equal both to $f'$ and $f''$ (see Proposition \ref{p:volume-term}). Eventually, in the final step we identify $\hat g$ by proving that it coincides with both $g'$ and $g''$ (see Proposition \ref{p:surface-term}). The representation result for $\hat g$ implies,  in particular, that the surface term in \eqref{F-hat} does not depend on $[u]$.  

\medskip

The following result is an immediate consequence of Theorem \ref{thm:main-result} and of the Urysohn property of $\Gamma$-convergence \cite[Proposition 8.3]{DM}. 

\begin{corollary}\label{c:cor-main-thm}
Let $(f_k) \subset  \mathcal F$, $(g_k) \subset \mathcal G$ and let $\F_k$ be the functionals as in \eqref{F_e}. Let $f'$, $f''$ be as in \eqref{f'} and \eqref{f''}, respectively, and $g'$, $g''$ be as in \eqref{g'} and \eqref{g''}, respectively.
Assume that 
\[
f'(x,\xi) = f''(x,\xi)=:f_\infty(x,\xi),\; \text{ for a.e.\ $x \in \R^n$ and for every $\xi \in \R^{m\x n}$} 
\]
and 
\[
g'(x,\nu) = g''(x,\nu)=:g_\infty(x,\nu),\; \text{ for every $x\in\R^n$ and every $\nu\in\Sph^{n-1}$},  
\]
for some Borel functions $f_\infty \colon \R^n \times \R^{m\times n} \to [0,+\infty)$ and $g_\infty \colon \R^n \times \Sph^{n-1} \to [0,+\infty)$.
Then, for every $A \in \A$ the functionals $\F_{k}(\cdot\,, \cdot\,, A)$ $\Gamma$-converge in $L^0(\R^n;\R^m)\times L^0(\R^n)$ to $\F(\cdot\,,\cdot\,,A)$ with $\F \colon L^0(\R^n;\R^m) \times L^0(\R^n) \times \A \longrightarrow [0,+\infty]$ given by
\[
\F(u, v, A)\defas
\begin{cases}
\displaystyle\int_A f_\infty(x,\nabla u)\dx+\int_{S_u \cap A} g_\infty(x,\nu_u)\dHn &\text{if}\ u \in GSBV^p(A;\R^m)\, ,\\
& v= 1\, \text{ a.e.\ in } A\, ,\\[4pt]
+\infty &\text{otherwise}\,.
\end{cases}
\]
\end{corollary}

\subsection{Convergence of minimisation problems} In view of Theorem \ref{thm:main-result} and Corollary \ref{c:cor-main-thm} we are in a position to prove the following result on the convergence of certain minimisation problems associated with $\F_k$. Other minimisation problems can be treated similarly.  

\begin{theorem}[Convergence of minimisation problems]\label{t:con-min-pb} 
Assume that the hypotheses of Corollary \ref{c:cor-main-thm} are satisfied.
Let $q\geq 1$, let $h\in L^q(A;\R^m)$, and set
\[
M_k :=\inf\bigg\{\F_{k}(u,v,A)+\int_A |u-h|^q\dx \colon (u,v)\in L^0(\R^n;\R^m)\times L^0(\R^n)\bigg\}\,.
\]
Then $M_k \to M$ as $k \to +\infty$ where 
\begin{equation}\label{c:lim-mp}
M :=\min\bigg\{\F(u,1,A)+\int_A |u-h|^q\dx \colon u\in GSBV^p(A;\R^m) \cap L^q(A;\R^m)\bigg\}\,.
\end{equation}
Moreover if $(u_k,v_k) \subset L^0(\R^n;\R^m)\times L^0(\R^n)$ is such that
\begin{equation}\label{c:min-seq}
\lim_{k \to +\infty} \bigg(\F_{k}(u_k,v_k,A)+\int_A |u_k-h|^q\dx-M_k\bigg)=0\,,
\end{equation}
then $v_k \to 1$ in $L^p(A)$ and there exists a subsequence of $(u_k)$ which converges in $L^q(A;\R^m)$ to a solution of \eqref{c:lim-mp}.
\end{theorem}
\begin{proof}
Let $(u_k,v_k) \subset L^0(\R^n;\R^m)\times L^0(\R^n)$ be as in \eqref{c:min-seq}. Then the convergence $v_k \to 1$ in $L^p(A)$ readily follows by \ref{hyp:lb-g}, whereas the lower bound in \eqref{est:bounds-Fe} on the functionals $\F_k$ together with \cite[Lemma 4.1]{Foc01} give the existence of a subsequence $(u_{k_j})\subset (u_k)$ with $u_{k_j} \to u$ in $L^0(A;\R^m)$, for some $u\in GSBV^p(A;\R^m)$. Eventually, the convergence $M_k \to M$, the improved convergence $u_{k_j} \to u$ in $L^q(A;\R^m)$, and the fact that $u$ is a solution to \eqref{c:lim-mp} follow arguing as in \cite[Theorem 7.1]{DMI13}, now invoking Corollary \ref{c:cor-main-thm}. 
\end{proof}

\subsection{Homogenisation} 
In this subsection we prove a general homogenisation theorem without assuming any spatial periodicity of the integrands $f_k$ and $g_k$. This theorem will be crucial to prove the stochastic homogenisation result Theorem \ref{thm:stoch_hom_2}.
 
\medskip 
 
In order to state the homogenisation result, we need to introduce some further notation. 

For $f\in \mathcal F$, $g\in \mathcal G$, $A\in \A$, and $u\in W^{1,p}(A;\R^m)$ set 
\begin{equation}\label{eq:Fb}
\F^b(u,A):=\int_A f (x,\nabla u)\dx,
\end{equation}
while for $v\in W^{1,p}(A)$ with $0\leq v \leq 1$ set
\begin{equation}\label{eq:Fs}
\F^s(v,A):=\int_A g (x,v,\nabla v)\dx. 
\end{equation}
Let $A\in \A$; for $x\in \R^n$ and $\xi\in \R^{m \times n}$ we define 
\begin{equation}\label{eq:m-b}
\m^b(u_\xi,A)\defas\inf\{\F^b(u,A)\colon u\in W^{1,p}(A;\R^m)\, ,\ u=u_\xi\ \text{near}\ \partial A\},
\end{equation}
while for $z\in \R^n$, $\nu \in \Sph^{n-1}$, and $\bar u^\nu_z$ as in \ref{uv-bar} (with $x=z$) we set
\begin{equation}\label{eq:m-s-bis}
\m^s(\bar u^\nu_z,A):=\inf\{\F^s(v,A)\colon v\in \Adm(\bar u^\nu_z,A)\}\,,
\end{equation}
where $\Adm(\bar u^\nu_z,A)$ is as in \eqref{c:adm-e-bar}, with $\bar u^\nu_{x,\e_k}$ and $\bar v^\nu_{x,\e_k}$ replaced, respectively, by $\bar u^\nu_z$ and $\bar v^\nu_z$ (that is, $x=z$ and $\e_k=1$).

We are now ready to state the homogenisation result; the latter corresponds to the choice 
\begin{equation}\label{eq:f-g-hom}
f_k(x,\xi):= f\Big(\frac{x}{\e_k},\xi\Big)\quad \text{and} \quad g_{k}(x,v,w):=g\Big(\frac{x}{\e_k},v, w\Big), 
\end{equation}
with $f\in \mathcal F$ and $g\in \mathcal G$. We stress again that we will not assume any spatial periodicity of $f$ and $g$. 
\begin{theorem}[Deterministic homogenisation]\label{thm:hom} Let $f\in \mathcal F$ and $g\in \mathcal G$.
Let  also $\m^b(u_\xi,Q_r(rx))$ be as in \eqref{eq:m-b} with $A=Q_r(rx)$ and $\m^s(\bar u_{rx}^\nu,Q^\nu_r(rx))$ be as in \eqref{eq:m-s-bis} with $z=rx$ and $A=Q_r^\nu(rx)$.  Assume that for every $x\in \R^n$, $\xi\in\R^{m\x n}$, $\nu\in \S^{n-1}$ the two following limits
\begin{equation}\label{eq:f-hom}
\lim_{r\to +\infty} \frac{\m^b(u_\xi,Q_r(rx))}{r^n}=:f_{\rm hom}(\xi)\,,
\end{equation}
\begin{equation}\label{eq:g-hom}
\lim_{r\to +\infty} \frac{\m^s(\bar{u}^\nu_{rx},Q^\nu_r(rx))}{r^{n-1}}=:g_{\rm hom}(\nu)\,,
\end{equation}
exist and are independent of $x$. Then, for every $A\in \A$ the functionals $\F_k(\cdot,\cdot, A)$ defined in \eqref{F_e} with $f_k$ and $g_k$ as in \eqref{eq:f-g-hom} $\Gamma$-converge in $L^0(\R^n;\R^m) \times L^0(\R^n)$ to the functional $\F_{\rm hom}(\cdot,\cdot,A)$, with $\F_{\rm hom}\colon \LtL \times \A \longrightarrow [0,+\infty]$ given by 
\begin{equation*}
\F_{\rm hom}(u,v,A)\defas
\begin{cases}
\displaystyle\int_A f_{\rm hom}(\nabla u)\dx+\int_{S_u \cap A} g_{\rm hom}(\nu_u)\dHn &\text{if}\ u \in GSBV^p(A;\R^m)\, ,\\
& v= 1\, \text{a.e.\ in } A\, ,\\[4pt]
+\infty &\text{otherwise}\,.
\end{cases}
\end{equation*}
\end{theorem}
\begin{proof}
Let $f'$, $f''$ be as in \eqref{f'}, \eqref {f''}, respectively, and $g'$, $g''$ be as in \eqref{g'}, \eqref{g''}, respectively. By virtue of Corollary \ref{c:cor-main-thm} it suffices to show that 
	\begin{equation}\label{c:id-hom}
		f_{\hom}(\xi)=f'(x,\xi)=f''(x,\xi)\,,\quad g_{\hom}(\nu)=g'(x,\nu)=g''(x,\nu)\,,
	\end{equation}
for every $x\in\R^n$, $\xi\in\R^{m\x n}$, and $\nu\in \Sph^{n-1}$.
	
	We start by proving the first two equalities in \eqref{c:id-hom}. To this end, fix $x\in\R^n$, $\xi\in\R^{m\x n}$, $\rho>0$ and $k\in\N$. Take $u\in W^{1,p}(Q_\rho(x);\R^m)$ and let $u_k\in W^{1,p}(Q_{\frac{\rho}{\e_k}}(\frac{x}{\e_k});\R^m)$ be defined as $u_k(y):=\frac{1}{\e_k}u(\e_ky)$. We notice that $u=u_\xi$ near $\partial Q_\rho(x)$ if and only if $u_k=u_\xi$ near $\partial Q_{\frac{\rho}{\e_k}}(\frac{x}{\e_k})$. 
	Moreover, a change of variables gives
	$
	\F^b_{k}(u,1,Q_\rho(x))=\e_k^n\F^b\big(u_k,Q_{\frac{\rho}{\e_k}}\big(\tfrac{x}{\e_k}\big)\big)\,,
	$
	from which, setting $r_k:=\frac{\rho}{\e_k}$ we immediately deduce 
	\begin{equation*}
		\m^b_{k}(u_\xi,Q_\rho(x))=\e_k^n\m^b\big(u_\xi,Q_{\frac{\rho}{\e_k}}\big(\tfrac{x}{\e_k}\big)\big)=
		\frac{\rho^n}{r_k^n}\m^b\big(u_\xi,Q_{r_k}\big(r_k\tfrac{x}{\rho}\big)\big)\,,
	\end{equation*}
	and thus \eqref{eq:f-hom} applied with $x$ replaced by $x/\rho$ yields
	\begin{equation*}
		\lim_{k\to+\infty}\frac{\m^b_{\e_k}(u_\xi,Q_\rho(x))}{\rho^n}=\lim_{k\to+\infty}\frac{\m^b\big(u_\xi,Q_{r_k}\big(r_k\tfrac{x}{\rho}\big)\big)}{r_k^n}=f_{\hom}(\xi)\,.
	\end{equation*}
	Eventually, by \eqref{f'} and \eqref{f''} we get $f'(x,\xi)=f''(x,\xi)=f_{\hom}(\xi)$, for every $x\in \R^n$, $\xi\in \R^{m\times n}$.
	
	We now prove the second two equalities in \eqref{c:id-hom}. To this end, for fixed $x\in\R^n$, $\nu\in \S^{n-1}$, $\rho>0$ and $k\in\N$, let $v\in \Adm(\bar{u}_{x,\e_k}^\nu,Q_\rho^\nu(x))$; then there exists $u\in W^{1,p}(Q^\nu_\rho(x);\R^m)$ such that 
	\begin{equation}\label{cond:v-homogenisation}
	v\,\nabla u=0\;\text{ a.e.\ in }\; Q_\rho^\nu(x)\; \text{ and }\;
		(u,v)=(\bar{u}_{x,\e_k}^\nu,\bar{v}_{x,\e_k}^\nu)\; \text{ near }\; \partial Q^\nu_\rho(x)\,.
	\end{equation}
	 Define $(u_k,v_k)\in W^{1,p}(Q^\nu_{\frac{\rho}{\e_k}}(\frac{x}{\e_k});\R^m)\times W^{1,p}(Q^\nu_{\frac{\rho}{\e_k}}(\frac{x}{\e_k}))$  as $u_k(y):=u_k(\e_ky)$, $v_k(y):=v(\e_ky)$. Then~\eqref{cond:v-homogenisation} is equivalent to
	\begin{equation*}
v_k\,\nabla u_k=0\;\text{ a.e.\ in }\; Q^\nu_{\frac{\rho}{\e_k}}(\tfrac{x}{\e_k})\; \text{ and }\;
(u_k,v_k)=(\bar{u}_{\frac{x}{\e_k}}^\nu,\bar{v}_{\frac{x}{\e_k}}^\nu)\;\text{ near }\; \partial Q^\nu_{\frac{\rho}{\e_k}}(\tfrac{x}{\e_k})\,,
\end{equation*}
that is, $v_k \in \Adm(\bar{u}_{\frac{x}{\e_k}}^\nu,Q^\nu_{\frac{\rho}{\e_k}}(\tfrac{x}{\e_k}))$. 
Further, by a change of variables we immediately obtain that
$
\F_{k}^s(v,Q^\nu_\rho(x))=\e_k^{n-1}\F^s\big(v_k,Q^\nu_{\frac{\rho}{\e_k}}\big(\tfrac{x}{\e_k}\big)\big)\,.
$
Hence, again setting $r_k:=\frac{\rho}{\e_k}$, we infer
	\begin{equation*}
\m^s_{k}(\bar{u}_{x,\e_k}^\nu,Q_\rho^\nu(x))=\e_k^{n-1}\m^s\big(\bar{u}_{\frac{x}{\e_k}}^\nu,Q_{\frac{\rho}{\e_k}}^\nu\big(\tfrac{x}{\e_k}\big)\big)=
\frac{\rho^{n-1}}{r_k^{n-1}}\m^s\big(\bar u^\nu_{r_k\frac{x}{\rho}},Q_{r_k}\big(r_k\tfrac{x}{\rho}\big)\big)\,.
\end{equation*}
Hence invoking \eqref{eq:g-hom} applied with $x$ replaced by $x/\rho$ we get
\begin{equation*}
\lim_{k\to+\infty}\frac{\m^s_{k}(\bar u_{x,\e_k}^\nu,Q_\rho^\nu(x))}{\rho^{n-1}}=\lim_{k\to+\infty}\frac{\m^s\big(\bar u^\nu_{r_k\frac{x}{\rho}},Q_{r_k}^\nu\big(r_k\tfrac{x}{\rho}\big)\big)}{r_k^{n-1}}=g_{\hom}(\nu)\,.
\end{equation*}
Eventually, Proposition~\ref{p:equivalence-N-D} gives $g'(x,\nu)=g''(x,\nu)=g_{\hom}(\nu)$, for every $x\in \R^n$, $\nu\in \Sph^{n-1}$ and thus the claim.
\end{proof}
%
%
\section{Properties of $f',f'',g',g''$}\label{sect:prop}
\noindent This section is devoted to prove some properties satisfied by the functions $f', f'', g'$, and $g''$ defined in \eqref{f'}, \eqref{f''}, \eqref{g'}, and \eqref{g''}, respectively. 

\begin{proposition}\label{prop:f'-f''}
Let $(f_k) \subset \mathcal F$; then the functions $f'$ and $f''$ defined, respectively, as in \eqref{f'} and \eqref{f''} satisfy properties \ref{hyp:meas-f}-\ref{hyp:ub-f}. Moreover they also satisfy \ref{hyp:cont-xi-f}, albeit with a different constant $\tilde L_1>0$.
\end{proposition}
\begin{proof}
The proof readily follows from \cite[Lemma A.6]{CDMSZ19}. 
\end{proof}

\begin{remark}
As shown in \cite[Lemma A.6]{CDMSZ19}, hypothesis \ref{hyp:cont-xi-f} in the definition of the class $\mathcal F$ can be weaken to obtain a larger class of volume integrands $\widetilde{\mathcal F}$ which is closed, in the sense that if $(f_k) \subset \widetilde{\mathcal F}$ then $f', f'' \in \widetilde{\mathcal F}$.  
\end{remark}

Before proving the corresponding result for $g',g''$ we need to prove the two following technical lemmas. 
\begin{lemma}\label{lemma:g'-g''}
Let $\rho,\delta,\e_k>0$ with $\rho> \delta>2\e_k$. Set 
\begin{equation*}
\m_{k}^{s,\delta}(\bar u_{x,\e_k}^\nu,Q_\rho^\nu(x))\defas\inf\{\F_k^s(v,Q_\rho^\nu(x))\colon v\in \Adm^\delta(\bar u_{x,\e_k}^\nu,Q_\rho^\nu(x))\},
\end{equation*}
where 
\begin{multline*}
\Adm^\delta(\bar u_{x,\e_k}^\nu,Q_\rho^\nu(x))\defas\{v\in W^{1,p}(Q_\rho^\nu(x)),\; 0\leq v\leq 1 \colon \, \exists\, u\in W^{1,p}(Q_\rho^\nu(x);\R^m)\ \text{ with } v \,\nabla u=0\\[4pt] \text{a.e.\ in}\  Q_\rho^\nu(x)\
\text{and}\ (u,v)=(\bar u_{x,\e_k}^\nu,\bar{v}_{x,\e_k}^\nu)\ \text{in}\ Q_\rho^\nu(x) \setminus \overline Q_{\rho-\delta}^\nu(x)\}\,,
\end{multline*}
and $(\bar{u}_{x,\e_k}^\nu,\bar{v}_{x,\e_k}^\nu)$ are as in~\ref{uv-bar-e}.
Moreover, let $g'_\rho,g''_\rho:\R^n\x\S^{n-1}\to[0,+\infty]$ be the functions defined as
\begin{align}\label{c:g'-rho}
g'_\rho(x,\nu) &\defas \inf_{\delta>0}\liminf_{k\to+\infty}\m_{k}^{s,\delta}(\bar u_{x,\e_k}^\nu,Q_\rho^\nu(x))= \lim_{\delta\to 0}\liminf_{k\to+\infty}\m_{k}^{s,\delta}(\bar u_{x,\e_k}^\nu,Q_\rho^\nu(x))\,,\\[4pt]\nonumber
g''_\rho(x,\nu) &\defas \inf_{\delta>0}\limsup_{k\to+\infty}\m_{k}^{s,\delta}(\bar u_{x,\e_k}^\nu,Q_\rho^\nu(x))=\lim_{\delta\to 0}\limsup_{k\to+\infty}\m_{k}^{s,\delta}(\bar u_{x,\e_k}^\nu,Q_\rho^\nu(x))\,.
\end{align}
Then the restrictions of $g'_\rho$, $g''_\rho$ to the sets $\R^n\x\widehat{\S}^{n-1}_+$ and $\R^n\x\widehat{\S}^{n-1}_-$ are upper semicontinuous.
\end{lemma}
\begin{proof}
We only show that the restriction of the function $g'_\rho$ to the set $\R^n\x\widehat{\S}^{n-1}_+$ is upper semicontinuous, the other proofs being analogous. 

Let $\rho>0$, $x\in\R^n$, and $\nu\in\widehat{\S}^{n-1}_+$ be fixed. Let $\delta'>\delta$, since $\Adm^{\delta'}(\bar{u}_{x,\e_k}^\nu,Q_\rho^\nu(x))\subset\Adm^\delta(\bar{u}_{x,\e_k}^\nu,Q_\rho^\nu(x))$, we immediately obtain that in~\eqref{c:g'-rho} the infimum in $\delta>0$ coincides with the limit as $\delta\to 0$, which in particular exists. Thus, for fixed $\eta>0$ there exists $\delta_\eta>0$ such that  
\begin{equation}\label{c:delta0-almost-optimal}
\liminf_{k \to +\infty }\m_k^{s,\delta}(\bar{u}_{x,\e_k}^\nu,Q_\rho^\nu(x))\leq g_\rho'(x,\nu)+\eta,
\end{equation}
for every $\delta\in(0,\delta_\eta)$. 
Fix $\delta_0\in(0,\frac{\delta_\eta}{3})$, hence by \eqref{c:delta0-almost-optimal} we get
\begin{equation}\label{delta0-almost-optimal}
\liminf_{k\to+\infty}\m_{k}^{s,3\delta_0}(\bar u_{x,\e_k}^\nu,Q_\rho^\nu(x))\leq g_\rho'(x,\nu)+\eta\,.
\end{equation}
Now let $v_k \in \Adm^{3\delta_0}(\bar{u}_{x,\e_k}^\nu,Q_\rho^\nu(x))$ be such that
\begin{equation}\label{almostmin:uppersemicont}
\F_{k}^s(v_k,Q_\rho^\nu(x))\leq\m_{k}^{s,3\delta_0}(u_x^\nu,Q_\rho^\nu(x))+\eta\,, 
\end{equation}
hence, in particular, $v_k \nabla u_k =0$ a.e.\ in $Q_\rho^\nu(x)$, for some $u_k \in W^{1,p}(Q_\rho^\nu(x);\R^m)$; further 
\begin{equation}\label{bc:uppersemicont}
(u_k,v_k)=(\bar u_{x,\e_k}^\nu,\bar{v}_{x,\e_k}^\nu)\quad\text{in }\big(Q_\rho^\nu(x)\sm\overline{Q}^\nu_{\rho-3\delta_0}(x)\big)\,.
 \end{equation}
Now let $(x_j, \nu_j)\subset\R^n \times \widehat{\S}^{n-1}_+$ be  a sequence converging to $(x,\nu)$. 
Thanks to the continuity of the map $\nu\mapsto R_\nu$ there exists $\hat\jmath= \hat\jmath(\delta_0)\in\N$ such that for every $j\geq \hat\jmath$ we have
\begin{equation}\label{inclusion:uppersemicont}
Q_{\rho-2\delta_0}^\nu(x)\subset Q_{\rho-\delta_0}^{\nu_j}(x_j)\quad\text{and}\quad Q_{\rho-5\delta_0}^{\nu_j}(x_j)\subset Q_{\rho-4\delta_0}^\nu(x)\,.
\end{equation}
We now modify $(u_k,v_k)$ to obtain a new pair $(\tilde{u}_k,\tilde{v}_k)\in W^{1,p}(Q_\rho^{\nu_j}(x_j);\R^m)\x W^{1,p}(Q_\rho^{\nu_j}(x_j))$ with $\tilde{v}_k \in \Adm^\delta(\bar{u}_{x_j,\e_k}^{\nu_j},Q_\rho^{\nu_j}(x_j))$ for any $j\geq \hat \jmath$ and any $\delta\in (0,\delta_0)$. 
To this end, set 
\[
h_{k,j}\defas\e_k+|x-x_j|+\frac{\sqrt{n}}{2}\rho|\nu-\nu_j|
\]
and
\begin{equation*}
R_{k,j}\defas R_\nu\Big(\big(Q'_{\rho-3\delta_0+2\e_k}\sm\overline{Q}'_{\rho-3\delta_0}\big)\x\big(-h_{k,j},h_{k,j}\big)\Big)+x\,,
\end{equation*} 
(see Figure~\ref{fig:Uppersemicontinuity}).
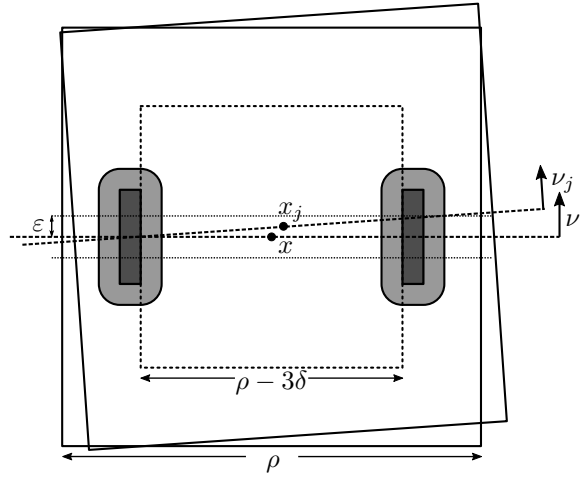
\begin{figure}[t]
	\centering
	\def\svgwidth{.5\textwidth}
\begingroup%
  \makeatletter%
  \providecommand\color[2][]{%
    \errmessage{(Inkscape) Color is used for the text in Inkscape, but the package 'color.sty' is not loaded}%
    \renewcommand\color[2][]{}%
  }%
  \providecommand\transparent[1]{%
    \errmessage{(Inkscape) Transparency is used (non-zero) for the text in Inkscape, but the package 'transparent.sty' is not loaded}%
    \renewcommand\transparent[1]{}%
  }%
  \providecommand\rotatebox[2]{#2}%
  \newcommand*\fsize{\dimexpr\f@size pt\relax}%
  \newcommand*\lineheight[1]{\fontsize{\fsize}{#1\fsize}\selectfont}%
  \ifx\svgwidth\undefined%
    \setlength{\unitlength}{80.83150368bp}%
    \ifx\svgscale\undefined%
      \relax%
    \else%
      \setlength{\unitlength}{\unitlength * \real{\svgscale}}%
    \fi%
  \else%
    \setlength{\unitlength}{\svgwidth}%
  \fi%
  \global\let\svgwidth\undefined%
  \global\let\svgscale\undefined%
  \makeatother%
  \begin{picture}(1,0.81444199)%
    \lineheight{1}%
    \setlength\tabcolsep{0pt}%
    \put(0,0){\includegraphics[width=\unitlength,page=1]{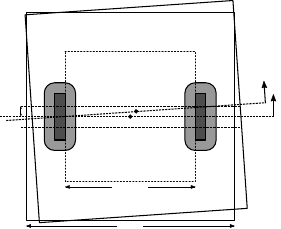}}%
    \put(0.45430296,-0.01){\color[rgb]{0,0,0}\makebox(0,0)[lt]{\lineheight{1.25}\smash{\begin{tabular}[t]{l}$\rho$\end{tabular}}}}%
    \put(0.4,0.129){\color[rgb]{0,0,0}\makebox(0,0)[lt]{\lineheight{1.25}\smash{\begin{tabular}[t]{l}$\rho-3\delta$\end{tabular}}}}%
    \put(0.03997675,0.41){\color[rgb]{0,0,0}\makebox(0,0)[lt]{\lineheight{1.25}\smash{\begin{tabular}[t]{l}$\e$\end{tabular}}}}%
    \put(0.985,0.41753507){\color[rgb]{0,0,0}\makebox(0,0)[lt]{\lineheight{1.25}\smash{\begin{tabular}[t]{l}$\nu$\end{tabular}}}}%
    \put(0.95959952,0.49176255){\color[rgb]{0,0,0}\makebox(0,0)[lt]{\lineheight{1.25}\smash{\begin{tabular}[t]{l}$\nu_j$\end{tabular}}}}%
    \put(0.47547782,0.37){\color[rgb]{0,0,0}\makebox(0,0)[lt]{\lineheight{1.25}\smash{\begin{tabular}[t]{l}$x$\end{tabular}}}}%
    \put(0.47547782,0.44){\color[rgb]{0,0,0}\makebox(0,0)[lt]{\lineheight{1.25}\smash{\begin{tabular}[t]{l}$x_j$\end{tabular}}}}%
  \end{picture}%
\endgroup%

	\caption{The cubes $Q_{\rho}^\nu(x)$, $Q_{\rho-3\delta_0}^\nu(x)$, $Q_\rho^{\nu_j}(x_j)$, and in gray the sets $R_{k,j}$ (dark gray) and $\{d_k<1\}$ (light gray).}
	\label{fig:Uppersemicontinuity}
	\end{figure}
Upon possibly increasing $\hat \jmath$, we can assume that $R_{k,j}\subset Q_{\rho-3\delta_0+2\e_k}^\nu(x)\sm\overline{Q}^\nu_{\rho-3\delta_0}(x)$, for every $j\geq \hat \jmath$. Thus, setting $d_k(y)\defas\frac{1}{\e_k}\dist(y,R_{k,j})$ and using~\eqref{inclusion:uppersemicont}, for $\e_k<\frac{\delta_0}{4}$ we also have 
\begin{equation}\label{inclusion:dk}
\{d_k<1\}\subset Q_{\rho-2\delta_0}^\nu(x)\sm\overline{Q}_{\rho-4\delta_0}^\nu(x)\subset Q_{\rho-\delta_0}^{\nu_j}(x_j)\sm\overline{Q}_{\rho-5\delta_0}^{\nu_j}(x_j)\,.
\end{equation}
Now let $\varphi_k\in C_c^\infty(Q'_{\rho-3\delta_0+2\e_k})$ be a cut-off between $Q'_{\rho-3\delta_0}$ and $Q'_{\rho-3\delta_0+2\e_k}$\ie $0\leq \varphi_k \leq 1$ and $\varphi_k =1$ in $Q'_{\rho-3\delta_0}$. For $y\in Q_\rho^{\nu_j}(x_j)$ we set 
\begin{equation*}
\tilde{u}_k(y) \defas\varphi_k\big((R_\nu^T(y-x))'\big)u_k(y)+\Big(1-\varphi_k\big(R_\nu^T(y-x))'\big)\Big)\bar u_{x_j,\e_k}^{\nu_j}(y)
\end{equation*}
and 
\begin{equation*}
\tilde{v}_k(y)\defas
\begin{cases}
\min\big\{v_k(y),d_k(y)\big\} &\text{in}\ Q_{\rho-3\delta_0}^\nu(x)\,,\\[4pt]
\min\big\{\bar v_{x_j,\e_k}^{\nu_j}(y),d_k(y)\big\} &\text{in}\ Q_{\rho}^{\nu_j}(x_j)\sm Q_{\rho-3\delta_0}^\nu(x)\,.
\end{cases}
\end{equation*}
By construction, we have $\tilde{u}_k=u_k$, $\tilde{v}_k\leq v_k$ in $Q_{\rho-3\delta_0}^\nu(x)$ and $\tilde{u}_k=\bar u_{x_j,\e_k}^{\nu_j}$, $\tilde{v}_k\leq \bar v_{x_j,\e_k}^{\nu_j}$ in $Q_\rho^{\nu_j}(x_j)\sm Q_{\rho-3\delta_0+2\e_k}^\nu(x)$; therefore, in particular, it holds
\begin{equation}\label{cond:vgradu}
\begin{split}
0\leq\tilde{v}_k|\nabla\tilde{u}_k| &\leq v_k|\nabla u_k|=0\; \text{ a.e.\ in }\; Q_{\rho-3\delta_0}^\nu(x)\\[4pt]
0\leq\tilde{v}_k|\nabla\tilde{u}_k| &\leq \bar v_{x_j,\e_k}^{\nu_j}|\nabla \bar u_{x_j,\e_k}^{\nu_j}|=0\; \text{ a.e.\ in }\; Q_{\rho}^{\nu_j}(x_j)\sm Q^\nu_{\rho-3\delta_0+2\e_k}(x)\,.
\end{split}
\end{equation}
Moreover, for $y\in \Big(Q_{\rho-3\delta_0+2\e_k}^\nu(x)\sm \overline{Q}_{\rho-3\delta_0}^\nu(x)\Big)\sm \overline{R}_{k,j}$ we have
\begin{equation}\label{scalarproduct1}
|(y-x)\cdot\nu|=|R_\nu^T(y-x)\cdot e_n|>h_{k,j}>\e_k\,.
\end{equation}
Then, by applying the triangular inequality twice, also noticing that $Q_\rho^\nu(x)\subset B_{\frac{\sqrt{n}}{2}\rho}(x)$, we obtain
\begin{equation}\label{scalarproduct2}
\begin{split}
|(y-x_j)\cdot\nu_j| &> 
h_{k,j}-|x-x_j|-|y-x||\nu-\nu_j|\geq\e_k\,.
\end{split}
\end{equation}
In view of~\eqref{bc:uppersemicont}, gathering~\eqref{scalarproduct1}~and~\eqref{scalarproduct2} we infer that 
\[
u_k=\bar u_{x,\e_k}^\nu=\bar u_{x_j,\e_k}^{\nu_j},\quad v_k=\bar v_{x,\e_k}^\nu=v_{x_j,\e_k}^{\nu_j} \; \text{ on }\; \Big(Q_{\rho-3\delta_0+2\e_k}^\nu(x)\sm \overline{Q}_{\rho-3\delta_0}^\nu(x)\Big)\sm \overline{R}_{k,j}\,. 
\]
Hence $\tilde{v}_k\in W^{1,p}(Q_\rho^{\nu_j}(x_j))$ and $\nabla\tilde{u}_k=0$ in $\big(Q_{\rho-3\delta_0+2\e_k}^\nu(x)\sm \overline{Q}_{\rho-3\delta_0}^\nu(x)\big)\sm \overline{R}_{k,j}$. Since moreover $\tilde{v}_k=0$ in $\overline{R}_{k,j}$, we immediately obtain that 
$\tilde{v}_k \nabla\tilde{u}_k=0$ in $Q_{\rho-3\delta_0+2\e_k}^\nu(x)\sm \overline{Q}_{\rho-3\delta_0}^\nu(x)$, 
which in view of~\eqref{cond:vgradu} yields  
\[
\tilde{v}_k\,\nabla\tilde{u}_k=0 \; \text{ a.e.\ in }\; Q_\rho^{\nu_j}(x_j)\,.
\]
Eventually, since in $Q_\rho^{\nu_j}(x_j)\sm Q_{\rho-3\delta_0}^\nu(x)$ we have $\tilde{v}_k=\bar v_{x_j,\e_k}^{\nu_j}$ if $d_k\geq 1$, from~\eqref{inclusion:dk} we deduce that $(\tilde{u}_k,\tilde{v}_k)=(\bar u_{x_j,\e_k}^{\nu_j},\bar v_{x_j,\e_k}^{\nu_j})$ in $Q_\rho^{\nu_j}(x_j)\sm\overline{Q}_{\rho-\delta_0}^{\nu_j}(x_j)$, hence also in $Q_\rho^{\nu_j}(x_j)\sm\overline{Q}_{\rho-\delta}^{\nu_j}(x_j)$ for every $\delta\in(0,\delta_0)$. In particular, $\tilde{v}_k \in \Adm^\delta(\bar u_{x_j,\e_k}^{\nu_j},Q_\rho^{\nu_j}(x_j))$ for any $\delta\in(0,\delta_0)$. 
Now it only remains to estimate $\F_{k}^s(\tilde{v}_k,Q_\rho^{\nu_j}(x_j))$. Arguing as in Proposition~\ref{p:equivalence-N-D}, we get
\begin{equation}\label{est:energy-uppersemicont}
\F_{k}^s(\tilde{v}_k,Q_\rho^{\nu_j}(x_j))\leq\F_{k}^s(v_k,Q_\rho^\nu(x))+ \F_{k}^s(\bar v_{x_j,\e_k}^{\nu_j},Q_{\rho}^{\nu_j}(x_j)\sm\overline{Q}_{\rho-3\delta_0}^\nu(x))+\frac{2c_4}{\e_k}\L^n(\{d_k<1\})\,.
\end{equation}
Moreover, the second inclusion in~\eqref{inclusion:uppersemicont} together with~\eqref{g-value-at-0} and~\eqref{1dim-energy-bis} implies that
\begin{align}\nonumber
\F_{k}^s(\bar v_{x_j,\e_k}^{\nu_j},Q_{\rho}^{\nu_j}(x_j)\sm\overline{Q}_{\rho-3\delta_0}^\nu(x)) & \leq \F_{k}^s(\bar v_{x_j,\e_k}^{\nu_j},Q_{\rho}^{\nu_j}(x_j)\sm\overline{Q}_{\rho-5\delta_0}^{\nu_j}(x_j))\\\label{est1:energy-uppersemicont}
&\leq c_4C_\vv\L^{n-1}\Big(Q'_{\rho}\sm\overline{Q'}_{\!\!\rho-5\delta_0}\big)\Big)\leq C\delta_0\rho^{n-2}\,,
\end{align}
also
\begin{align}\nonumber
\frac{2c_4}{\e_k}\L^n\big(\{d_k<1\}\big)&\leq\frac{4c_4}{\e_k}\L^{n-1}\big(Q'_{\rho-3\delta_0+4\e_k}\sm \overline{Q'}_{\!\!\rho-3\delta_0-2\e_k}\big) (h_{k,j}+\e_k)
\\ \label{est2:energy-uppersemicont}
& \leq Ch_{k,j}(\rho-3\delta_0)^{n-2}+\mathcal{O}(\e_k)\,,
\end{align}
as $k\to +\infty$.

Since $\tilde{v}_k \in \Adm^\delta(\bar u_{x_j,\e_k}^{\nu_j},Q_\rho^{\nu_j}(x_j))$ for any $\delta\in(0,\delta_0)$, by combining~\eqref{almostmin:uppersemicont} and~\eqref{est:energy-uppersemicont}-\eqref{est2:energy-uppersemicont}, we get
\begin{equation}\label{est:uppersemicont-final}
\m_{k}^{s,\delta}(\bar u_{x_j,\e_k}^{\nu_j},Q_\rho^{\nu_j}(x_j))\leq \m_{k}^{s,3\delta_0}(\bar u_{x,\e_k}^\nu,Q_\rho^\nu(x))+\eta+C(\delta_0+h_{k,j})\rho^{n-2}+\mathcal{O}(\e_k)\,,
\end{equation}
as $k\to+\infty$. Using~\eqref{delta0-almost-optimal} and taking in~\eqref{est:uppersemicont-final} first the liminf as $k \to +\infty$, then the limit as $\delta \to 0$, and finally the limsup as $j \to +\infty$ gives
\begin{align*}
\limsup_{j\to+\infty}g_\rho'(x_j,\nu_j)\leq g_\rho'(x,\nu)+2\eta+C\delta_0\rho^{n-2},
\end{align*}
since $\lim_j\lim_k h_{k,j}=0$.
Therefore, letting $\delta_0\to 0$ the upper semicontinuity of $g_\rho'$ restricted to $\widehat{\S}_+^{n-1}\x\R^n$ follows by the arbitrariness of $\eta>0$.
\end{proof}
\begin{lemma}\label{lemma-delta-n}
Let $g'$ and $g''$ be as in \eqref{g'} and \eqref{g''}, respectively, and let $g_\rho'$, $g_\rho''$ be as in~\eqref{c:g'-rho}. Then for every $x\in \R^n$ and every $\nu \in \S^{n-1}$ it holds
\begin{equation*}
g'(x,\nu)=\limsup_{\rho\to 0}\frac{1}{\rho^{n-1}}g_\rho'(x,\nu)\;\text{ and }\; g''(x,\nu)=\limsup_{\rho\to 0}\frac{1}{\rho^{n-1}}g_\rho''(x,\nu). 
\end{equation*}
\end{lemma}
\begin{proof}
We prove the statement only for $g'$, the proof for $g''$ being analogous. 
We notice that since $\Adm^\delta(\bar{u}_{x,\e_k}^\nu,Q_\rho^\nu(x))\subset\Adm(\bar{u}_{x,\e_k}^\nu,Q_\rho^\nu(x))$ for every $\delta \in (0,\rho)$, we clearly have 
\[
g'(x,\nu)\leq\limsup_{\rho\to 0}\frac{1}{\rho^{n-1}}g_\rho'(x,\nu),
\] 
therefore to conclude we just need to prove the opposite inequality. This can be done by means of an easy extension argument as follows. 
For fixed $\rho> 0$, $x\in\R^n$, and $\nu\in\S^{n-1}$ and for every $k\in \N$ such that $\e_k\in(0,\frac{\rho}{2})$ let $v_k\in\Adm(\bar{u}_{x,\e_k}^\nu,Q_\rho^\nu(x))$ satisfy 
\begin{equation}\label{cond:v-extension}
\F_k^s(v_k,Q_\rho^\nu(x))\leq\m_k^s(\bar{u}_{x,\e_k}^\nu,Q_\rho^\nu(x))+\rho^n\,,
\end{equation}
and let $u_k\in W^{1,p}(Q_\rho^\nu(x);\R^m)$ be the corresponding $u$-variable. Let $\alpha>0$ be arbitrary; thanks to the boundary conditions satisfied by $(u_k,v_k)$ we can extend the pair $(u_k,v_k)$ to $Q_{(1+\alpha)\rho}^\nu(x)$ by setting $(u_k,v_k)\defas(\bar{u}_{x,\e_k}^\nu,\bar{v}_{x,\e_k}^\nu)$ in $Q_{(1+\alpha)\rho}^\nu(x)\sm\overline{Q}_\rho^\nu(x)$. Then \eqref{g-value-at-0} and \eqref{1dim-energy-bis} yield
\begin{equation}\label{est:extension}
\begin{split}
\F_k^s(v_k,Q_{(1+\alpha)\rho}^\nu(x)) &\leq\F_k^s(v_k,Q_\rho^\nu(x))+\F_k^s\big(\bar{v}_{x,\e_k}^\nu,(Q_{(1+\alpha)\rho}^\nu(x)\sm\overline{Q}_\rho^\nu(x)\big)\\
&\leq \F_k^s(v_k,Q_\rho^\nu(x))+c_4C_\vv((1+\alpha)^{n-1}-1))\rho^{n-1}\,.
\end{split}
\end{equation}
Moreover, for $\delta\in(0,\alpha\rho)$ we have $v_k\in\Adm^\delta(\bar{u}_{x,\e_k}^\nu,Q_{(1+\alpha)\rho}^\nu(x))$, thus~\eqref{cond:v-extension} and~\eqref{est:extension} give
\begin{equation*}
\inf_{\delta>0}\liminf_{k\to+\infty}\m_k^{s,\delta}(\bar{u}_{x,\e_k}^\nu,Q_{(1+\alpha)\rho}^\nu(x))\leq\m_k^s(\bar{u}_{x,\e_k}^\nu,Q_\rho^\nu(x))+\rho^n+c_4C_\vv((1+\alpha)^{n-1}-1))\rho^{n-1}\,.
\end{equation*}
Hence, dividing the above inequality by $((1+\alpha)\rho)^{n-1}$ and taking the limsup as $\rho\to 0$, thanks to Proposition~\ref{p:equivalence-N-D} we obtain 
\[
(1+\alpha)^{n-1}\limsup_{\rho \to 0}\frac{1}{\rho^{n-1}}g_\rho'(x,\nu)\leq g'(x,\nu)+c_4C_\vv((1+\alpha)^{n-1}-1)),
\] 
thus we conclude by the arbitrariness of $\alpha>0$. 
\end{proof}

We are now ready to state and prove the following proposition which establishes the properties satisfied by $g'$ and $g''$.

\begin{proposition}\label{prop:g'-g''}
Let $(g_k) \subset \mathcal G$; then the functions $g'$ and $g''$ defined, respectively, as in \eqref{g'} and \eqref{g''} are Borel measurable and satisfy the following two properties:
\begin{enumerate}

\item (symmetry) for every $x\in \R^n$ and every $\nu \in \S^{n-1}$ it holds
\begin{equation}\label{pr:symmetry}
g'(x, \nu)=g'(x,-\nu), \quad g''(x, \nu)=g''(x,-\nu);
\end{equation}
\item (boundedness) for every $x\in \R^n$ and every $\nu \in \Sph^{n-1}$ it holds
\begin{equation}\label{c:bdds-g'-g''}
c_3 c_p \leq g'(x, \nu) \leq c_4 c_p, \quad c_3 c_p \leq g''(x,\nu) \leq c_4 c_p, 
\end{equation}
where $c_p:=2(p-1)^{\frac{1-p}{p}}$. 
\end{enumerate}
\end{proposition}

\begin{proof}
We prove the statement only for $g'$, the proof for $g''$ being analogous.  

We divide the proof into three steps.

\step 1 $g'$ is Borel measurable. Let $\rho>0$ and let $g'_\rho$ be the function defined in \eqref{c:g'-rho}.  
Arguing as in the proof of Lemma \ref{lemma-delta-n} we deduce that the function $\rho \to g'_{\rho}(x,\nu)-c_4C_\vv\rho^{n-1}$ is nonincreasing on $(0,+\infty)$. \EEE From this it follows that
\[
\lim_{\rho'\to \rho^-}g'_{\rho'}(x,\nu)\geq g'_{\rho}(x,\nu)\geq \lim_{\rho'\to \rho^+} g'_{\rho'}(x,\nu),
\]
for every $x\in \R^n$, $\nu \in \Sph^{n-1}$, and every $\rho>0$. Thus, if $D$ is a countable dense subset of $(0,+\infty)$ we have 
\[
\limsup_{\rho\to 0}\frac{1}{\rho^{n-1}}g'_{\rho}(x,\nu)= \limsup_{\substack{\rho\to 0\\\rho \in D}}\frac{1}{\rho^{n-1}}g'_{\rho}(x,\nu)
\]
and hence by Lemma \ref{lemma-delta-n} we get
\begin{align*}
g'(x,\nu)&=\limsup_{\substack{\rho\to 0\\ \rho \in D}}\frac{1}{\rho^{n-1}} g'_{\rho}(x,\nu).
\end{align*}
Therefore the Borel measurablility of $g'$ follows by Lemma \ref{lemma:g'-g''} which guarantees, in particular, that the function
$(x,\nu) \mapsto g'_\rho(x,\nu)$ is Borel measurable for every $\rho>0$. 

\step 2 $g'$ is symmetric in $\nu$.  Property \eqref{pr:symmetry} immediately follows from the definition of $g'$ and from the fact that $u^\nu_x=-u^{-\nu}_x+e_1$ a.e.\ and $Q^\nu_\rho(x)=Q^{-\nu}_\rho(x)$ (see \ref{Rn}), which implies in particular that $v\in\Adm_{\e_k,\rho}(x,\nu)$ with corresponding $u\in W^{1,p}(Q_\rho^{\nu}(x);\R^m)$ if and only if $v\in\Adm_{\e_k,\rho}(x,-\nu)$ with corresponding $w\defas -u+e_1\in W^{1,p}(Q_\rho^{-\nu}(x);\R^m)$.

\step 3 $g'$ is bounded. 
To prove that $g'$ satisfies the bounds in \eqref{c:bdds-g'-g''} we start by observing that thanks to \ref{hyp:lb-g} and \ref{hyp:up-g} we have 
\[
c_3 \int_{Q^\nu_\rho(x)} \bigg(\frac{(1-v)^p}{\e_k}+\e_k^{p-1}|\nabla v|^p\bigg)\dx \leq \F^s_{k}(v,Q^\nu_\rho(x)) \leq c_4 \int_{Q^\nu_\rho(x)} \bigg(\frac{(1-v)^p}{\e_k}+\e_k^{p-1}|\nabla v|^p\bigg)\dx,
\]
for every $v\in W^{1,p}(Q^\nu_\rho(x))$ with $0\leq v \leq 1$ a.e.\ in $Q^\nu_\rho(x)$. 
Therefore to establish \eqref{c:bdds-g'-g''} it is enough to show that
\begin{equation}\label{c:claim}
\lim_{k \to +\infty} \m_{k,\rho}(x,\nu)=c_p\rho^{n-1},
\end{equation}
for every $x\in\R^n$ and $\rho>0$, where
\begin{equation*}
\m_{k,\rho}(x,\nu) \defas \min\bigg\{\int_{Q^\nu_\rho(x)} \bigg(\frac{(1-v)^p}{\e_k}+\e_k^{p-1}|\nabla v|^p\bigg)\dy \colon v\in \Adm_{\e_k,\rho}(x,\nu)\bigg\}.
\end{equation*}
Let $x\in\R^n$, $\nu\in\S^{n-1}$, $\rho>0$, and let $k\in\N$ be such that $2\e_k<\rho$. By the homogeneity and rotation invariance of the Ambrosio-Tortorelli functional we have 
\[
\m_{k,\rho}(x,\nu) = \m_{\e_k,\rho}(0,e_n),
\] 
Let $\eta>0$ be arbitrary; reasoning as in the construction of a recovery sequence for the Ambrosio-Tortorelli functional we find a sequence $(u_k,v_k)\subset W^{1,p}(Q_\rho(0);\R^m)\times W^{1,p}(Q_\rho(0))$ satisfying $v_k\nabla u_k=0$ a.e.\ in $Q_\rho(0)$, $(u_k,v_k)=(u_0^{e_n},1)$ in $\{|y_n|>\e_k T_\eta\}$ for some $T_\eta>0$ and
\begin{equation}\label{c:cp-recovery}
\limsup_{k \to +\infty}\int_{Q_\rho(0)} \bigg(\frac{(1-v_k)^p}{\e_k}+\e_k^{p-1}|\nabla v_k|^p\bigg)\dx = (c_p+\eta) \rho^{n-1}\,.
\end{equation}
Then, using a similar argument as in the proof of Proposition~\ref{p:equivalence-N-D}, we can modify $v_k$ to obtain a function $\tilde{v}_k\in W^{1,p}(Q_\rho(0))$ satisfying $\tilde{v}_k=1$ in $Q_{\rho}(0)\sm \overline Q_{\rho-\e_k}(0)\cap \{|y_n|>\e_k\}$ and
\begin{equation}\label{c:cp-modification}
\int_{Q_\rho(0)} \bigg(\frac{(1-\tilde v_k)^p}{\e_k}+\e_k^{p-1}|\nabla \tilde v_k|^p\bigg)\dx\leq \int_{Q_\rho(0)} \bigg(\frac{(1-v_k)^p}{\e_k}+\e_k^{p-1}|\nabla v_k|^p\bigg)\dx+C\e_k\rho^{n-2}\,.
\end{equation}
In particular, since $\tilde{v}_k\in\Adm_{\e_k,\rho}(0,e_n)$, gathering~\eqref{c:cp-recovery}--\eqref{c:cp-modification}, by the arbitrariness of $\eta>0$ we conclude that
\begin{equation}\label{c:cp-one}
\limsup_{k\to +\infty}\m_{k,\rho}(0,e_n) \leq c_p \rho^{n-1}\,.
\end{equation}
We now turn to the proof of the lower bound. By the Fubini Theorem we have   
\begin{equation}\label{c:jg}
\begin{split}
\int_{Q_\rho(0)} \bigg(\frac{(1-v)^p}{\e_k}+\e_k^{p-1}|\nabla v|^p\bigg)\dx&= \int_{Q'_{\rho}}\int_{-\frac{\rho}{2}}^{\frac{\rho}{2}} \bigg(\frac{(1-v(x',x_n))^p}{\e_k}+\e_k^{p-1}|\nabla v (x',x_n)|^p\bigg)\dx_n \dx'\\
&\geq \int_{Q'_{\rho}}\int_{-\frac{\rho}{2}}^{\frac{\rho}{2}} \bigg(\frac{(1-v(x',x_n))^p}{\e_k}+\e_k^{p-1}\Big|\frac {\partial v (x',x_n)}{\partial x_n}\Big|^p\bigg)\dx_n \dx'.
\end{split}
\end{equation}
Then, if $v\in \Adm_{\e_k, \rho}(0,e_n)$ the corresponding $u$ coincides with $u_0^{e_n}$ in a neighbourhood of 
\[
\partial^\pm Q_\rho(0)\defas \Big\{(x',x_n)\in \R^{n-1}\times \R \colon x'\in \overline{Q'_\rho}, \; x_n=\pm\frac{\rho}{2}\Big\}.
\]
Since it must hold that $v\nabla u=0$ a.e.\ in $Q_\rho(0)$, then almost every straight line intersecting $\partial^\pm Q_\rho(0)$ and parallel to $e_n$ also intersects the level set $\{v=0\}$. Indeed, for $\L^{n-1}$-a.e.\ $x'\in Q'_{\rho}$ the pair $(u_{x'}(t),v_{x'}(t))\defas\big(u(x',t)\cdot e_1,v(x',t)\big)$ belongs to $W^{1,p}(-\frac{\rho}{2},\frac{\rho}{2})\times W^{1,p}(-\frac{\rho}{2},\frac{\rho}{2})$ and satisfies 
\begin{equation}\label{cond:1-dim-restriction}
v_{x'}(t)u_{x'}'(t)=v(x',t)\frac{\partial u(x',t)}{\partial x_n}\cdot e_1=0\;\text{ for $\L^1$-a.e.}\; t\in\Big(-\frac{\rho}{2},\frac{\rho}{2}\Big)\,,
\end{equation}
as well as $v_{x'}(\pm\frac{\rho}{2})=1$, $u_{x'}(-\frac{\rho}{2})=0$, and $u_{x'}(\frac{\rho}{2})=1$. Since $u_{x'}\in W^{1,p}(-\frac{\rho}{2},\frac{\rho}{2})$, the boundary conditions satisified by $u_{x'}$ imply the existence of a subset of $(-\frac{\rho}{2},\frac{\rho}{2})$ with positive $\L^1$-measure on which $u_{x'}'\neq 0$, hence $v_{x'}=0$ in view of~\eqref{cond:1-dim-restriction}. In particular, for $\L^{n-1}$-a.e.\ $x'\in Q_\rho'$ there exists $s\in(-\frac{\rho}{2},\frac{\rho}{2})$ such that $v(x',s)=0$.

Therefore, the Young Inequality
\[
\frac{(1-v(x',x_n))^p}{\e_k}+\e_k^{p-1}\bigg|\frac {\partial v (x',x_n)}{\partial x_n}\bigg|^p \geq \Big(\frac{p}{p-1}\Big)^{\frac{p-1}{p}}p^{\frac{1}{p}}(1-v(x',x_n))^{p-1}\bigg|\frac {\partial v (x',x_n)}{\partial x_n}\bigg|,
\]
applied for $\L^{n-1}$-a.e.\ $x'\in Q'_\rho$ together with the integration on $(-\frac{\rho}{2}, \frac{\rho}{2})$ give
\begin{align}\label{c:j-g}
\int_{-\frac{\rho}{2}}^{\frac{\rho}{2}}\bigg(\frac{(1-v(x',x_n))^p}{\e_k}+\e_k^{p-1}\bigg|\frac {\partial v (x',x_n)}{\partial x_n}\bigg|^p\bigg)\dx_n&= \int_{-\frac{\rho}{2}}^{s}\bigg(\frac{(1-v(x',x_n))^p}{\e_k}+\e_k^{p-1}\bigg|\frac {\partial v (x',x_n)}{\partial x_n}\bigg|^p\bigg)\dx_n
\nonumber\\
&+ \int_{s}^{\frac{\rho}{2}}\bigg(\frac{(1-v(x',x_n))^p}{\e_k}+\e_k^{p-1}\bigg|\frac {\partial v (x',x_n)}{\partial x_n}\bigg|^p\bigg)\dx_n\nonumber\\
&\geq 2\Big(\frac{p}{p-1}\Big)^{\frac{p-1}{p}}p^{\frac{1}{p}}\int_0^1 (1-t)^{p-1}\dt=c_p\,,
\end{align}
for $\L^{n-1}$-a.e.\ $x'\in Q'_\rho$.  Thus, gathering \eqref{c:jg} and \eqref{c:j-g} we get 
\[
\int_{Q_\rho(0)} \bigg(\frac{(1-v)^p}{\e_k}+\e_k^{p-1}|\nabla v|^p\bigg)\dx \geq c_p \rho^{n-1},
\]
for every $v\in \Adm_{\e_k,\rho}(0,e_n)$. Passing to the infimum on $v$ and to the liminf as $k\to +\infty$ we get
\begin{equation}\label{c:cp-two}
\liminf_{k\to +\infty}\m_{k,\rho}(0,e_n) \geq c_p \rho^{n-1},
\end{equation}
for every $\rho>0$. Eventually, by combining \eqref{c:cp-one} and \eqref{c:cp-two} we get \eqref{c:claim}, and hence \eqref{c:bdds-g'-g''}. 
\end{proof}
\EEE

%
%

\section{$\Gamma$-convergence and integral representation}\label{s:G-convergence}

\noindent In this section we show that, up to subsequences, the functionals $\F_k$ $\Gamma$-converge in $L^0(\R^n;\R^m)\times L^0(\R^n)$ to an integral functional of free-discontinuity type. This result is achieved by following a standard procedure which combines the localisation method of $\Gamma$-convergence (see \textit{e.g.}, \cite[Chapters 14-18]{DM} or \cite[Chapters 10, 11]{BDf}) together with an integral-representation result in $SBV$ \cite[Theorem 1]{BFLM02}. Though rather technical, this procedure is by now classical. For this reason here we only detail the adaptations of the theory to our specific setting, while we refer the reader to the literature for the more standard aspects.    

We start by showing that the functionals $\F_k$ satisfy the so-called fundamental estimate, uniformly in $k$.
\begin{proposition}[Fundamental estimate]\label{prop:fund-est}
Let $\F_k$ be as in~\eqref{F_e}. Then, for every $\eta>0$ and for every $A,\,A',\,B\in \A$ with $A \subset\subset A'$ there exists a constant $M_\eta>0$ (also depending on $A,A',B$) satisfying the following property: For every $k\in\N$ and for every $(u,v)\in W^{1,p}(A';\R^m)\times W^{1,p}(A')$, $(\tilde u,\tilde v)\in W^{1,p}(B;\R^m)\times W^{1,p}(B)$, $0\leq v,\tilde v\leq 1$, there exists a pair $(\hat u,\hat v)\in W^{1,p}(A\cup B;\R^m)\times W^{1,p}(A\cup B)$ with $0\leq\hat v\leq 1$ such that $(\hat{u},\hat{v})=(u,v)$ a.e.\ in $A$, $(\hat{u},\hat{v})=(\tilde{u},\tilde{v})$ a.e.\ in $B\sm\overline{A}'$ and
\begin{align}\label{fund-est}
\F_k(\hat u, \hat v, A\cup B)\leq (1+\eta)\big(\F_k(u,v,A')+ \F_k(\tilde u, \tilde v,B)\big)+M_\eta \big(\|u-\tilde{u}\|_{L^p(S;\R^m)}^p+\e_k^{p-1}\big)\,,
\end{align}
where $S\defas (A'\sm A)\cap B$.
\end{proposition}
\begin{proof}
Fix $k\in \N$, $\eta>0$ and $A,\,A',\,B\in \A$ with $A\wcont A'$. Let $N \in \N$ and $A_1,\ldots,A_{N+1}\in \A$ with
\begin{equation*}
A\wcont A_1\wcont \dots \wcont A_{N+1} \wcont A'\,.
\end{equation*}
For each $i=2,\ldots,N+1$ let $\varphi_i$ be a smooth cut-off function between $A_{i-1}$ and $A_i$ and let
\begin{equation*}
M:=\max_{2\leq i\leq N+1} \|\nabla \varphi_i\|_\infty\,.
\end{equation*}
Let $(u,v)$ and $(\tilde u,\tilde v)$ be as in the statement and consider the function $w\in L^0(\R^n)$ defined by
\begin{equation*}
w:=\min\{v,\tilde{v}\}\,,
\end{equation*}
clearly $0\leq w\leq 1$.
For $i=3,\ldots,N$ we define $(\hat u^i, \hat v^i) \in W^{1,p}(A\cup B;\R^m)\times W^{1,p}(A\cup B)$ as follows
\begin{equation*}
\hat u^i := \varphi_{i} u + (1-\varphi_{i}) \tilde u
\quad\text{and}\quad
\hat v^i :=
\begin{cases}
\varphi_{i-1} v + (1-\varphi_{i-1}) w & \text{in}\ A_{i-1}\,, \\
w & \text{in}\ A_i \setminus A_{i-1}\,, \\
 \varphi_{i+1}\,w + (1-\varphi_{i+1}) \tilde v & \text{in}\ \R^n\setminus A_{i}\,.
\end{cases}
\end{equation*}
Then, setting $S_i\defas A_i\sm\overline{A}_{i-1}$ and taking into account the definition of $(\hat{u}^i,\hat{v}^i)$ we have
\begin{equation}\label{4est2}
\begin{split}
\F_k (\hat u^i,\hat v^i, A\cup B)\leq \F_k (u,v, A_{i-2}) &+ \F_k (u,\hat v^i, S_{i-1}\cap B) +\F_k (\hat u^i,w, S_i\cap B)\\
&+ \F_k (\tilde{u},\hat v^i, S_{i+1}\cap B)+ \F_k (\tilde{u},\tilde{v}, B\setminus \overline A_{i+1})\,.
\end{split}
\end{equation}
We now come to estimate the three terms in~\eqref{4est2} involving the sets  $S_{i-1}, S_i$, and $S_{i+1}$. \EEE We start observing that thanks to~\ref{hyp:ub-f} and \ref{hyp:lb-f}, exploiting the definition of $w$ and the fact that $\psi$ is increasing, we have
\begin{equation}\label{est:fund-est-Fb1}
\F_k^b(u,\hat v^i, S_{i-1}\cap B)\leq c_2\int_{S_{i-1}\cap B}\hspace*{-1em}\psi(v)|\nabla u|^p\dx\leq\frac{c_2}{c_1}\F_k^b(u,v,S_{i-1}\cap B)\,.
\end{equation}
Analogously, there holds
\begin{equation}\label{est:fund-est-Fb2}
\F_k^b(\tilde{u},\hat v^i, S_{i+1}\cap B)\leq\frac{c_2}{c_1}\F_k^b(\tilde{u},\tilde{v},S_{i+1}\cap B)\,.
\end{equation}
We complete the estimate of the bulk part of the energy by noticing that on $S_i\cap B$ we have 
\begin{align*}
|\nabla\hat{u}^i|^p\leq 3^{p-1}\big(|\nabla\varphi_i|^p|u-\tilde{u}|^p+|\nabla u|^p+|\nabla\tilde{u}|^p\big)\leq 3^{p-1}\big(M^p|u-\tilde{u}|^p+|\nabla u|^p+|\nabla\tilde{u}|^p\big)\,.
\end{align*}
Integrating over $S_i\cap B$, using~\ref{hyp:ub-f} and~\ref{hyp:lb-f}, the definition of $w$, and the monotonicity of $\psi$, we infer 
\begin{equation}\label{est:fund-est-Fb3}
\begin{split}
&\F_k^b(\hat{u}^i,w,S_i\cap B)\leq c_2\int_{S_i\cap B}\hspace*{-1em}\psi(w)|\nabla\hat{u}^i|^p\dx\\
&\hspace{1em}\leq c_2 3^{p-1}\int_{S_i\cap B}\hspace*{-1em}\big(\psi(v)|\nabla u|^p+\psi(\tilde{v})|\nabla\tilde{u}|^p\big)\dx+3^{p-1}M^pc_2\int_{S_i\cap B}\hspace*{-1em}|u-\tilde{u}|^p\dx\\
&\hspace{1em}\leq \frac{3^{p-1}c_2}{c_1}\big(\F_k^b(u,v,S_i\cap B)+\F_k^b(\tilde{u},\tilde{v},S_i\cap B)\big)+3^{p-1}M^pc_2\int_{S_i\cap B}\hspace*{-1em}|u-\tilde{u}|^p\dx\,.
\end{split}
\end{equation}
It remains to estimate the surface term in $\F_k$. Thanks to~\ref{hyp:up-g} it holds
\begin{equation}\label{est:fund-est-Fs1}
\F_k^s(w,S_i\cap B)\leq c_4\int_{S_i\cap B}\hspace*{-.5em}\bigg(\frac{(1-w)^p}{\e_k}+\e_k^{p-1}|\nabla w|^p\bigg)\dx\,.
\end{equation}
We now want to bound the right-hand side of~\eqref{est:fund-est-Fs1} in terms of $\F_k^s(v,S_i\cap B)+\F_k^s(\tilde{v},S_i\cap B)$. To this end we first observe that by definition  of $w$ we have 
\begin{equation}\label{prop:w}
|\nabla w|^p\leq |\nabla v|^p+|\nabla\tilde{v}|^p\quad \text{and}\quad (1-w)^p\leq (1-v)^p+(1-\tilde{v})^p\,.
\end{equation}
Thus, thanks to~\ref{hyp:lb-g}, \eqref{est:fund-est-Fs1} becomes
\begin{equation}\label{est:fund-est-Fs2}
\F_k^s(w,S_i\cap B)\leq \frac{c_4}{c_3}\big(\F_k^s(v,S_i\cap B)+\F_k^s(\tilde{v},S_i\cap B)\big)\,.
\end{equation}
Moreover, it holds
\begin{equation}\label{est:fund-est-Fs3}
\F_k^s(\hat{v}^i,(S_{i-1}\cup S_{i+1})\cap B)\leq c_4\int_{(S_{i-1}\cup S_{i+1})\cap B}\hspace*{-.5em}\bigg(\frac{(1-\hat{v}^i)^p}{\e_k}+\e_k^{p-1}|\nabla \hat{v}^i|^p\bigg)\dx\,.
\end{equation}
By the definition of $\hat{v}^i$ and by the convexity of $z\mapsto(1-z)^p$ for $z\in[0,1]$, on $S_{i-1}\cap B$ we have
\begin{align*}
(1-\hat{v}^i)^p=(\varphi_{i-1}(1-v)+(1-\varphi_{i-1})(1-w))^p\leq (1-v)^p+(1-w)^p\leq 2\big((1-v)^p+(1-\tilde{v})^p\big)\,,
\end{align*}
where in the last step we used again~\eqref{prop:w}. Similarly, there holds
\begin{align*}
|\nabla\hat{v}^i|^p\leq 3^{p-1}\big(|\nabla \varphi_i|^p|v-w|^p+|\nabla v|^p+|\nabla w|^p\big)\leq 3^{p-1}\big(M^p|v-\tilde{v}|^p+2(|\nabla v|^p+|\nabla \tilde{v}|^p)\big)\,.
\end{align*}
Since analogous arguments hold on $S_{i+1}\cap B$, from~\eqref{est:fund-est-Fs3} and~\ref{hyp:lb-g} we deduce
\begin{equation}\label{est:fund-est-Fs4}
\begin{split}
\F_k^s(\hat{v}^i,(S_{i-1}\cup S_{i+1})\cap B) &\leq 3^{p-1} \frac{2c_4}{c_3}\big(\F_k^s(v,(S_{i-1}\cup S_{i+1})\cap B)+\F_k^s(\tilde{v},(S_{i-1}\cup S_{i+1})\cap B)\big)\\
&+3^{p-1}M^p c_4\int_{(S_{i-1}\cup S_{i+1})\cap B}\hspace*{-3em}\e_k^{p-1}|v-\tilde{v}|^p\dx\,.
\end{split}
\end{equation}
Now set $\widehat{M}\defas\big(\frac{(1+3^{p-1}) c_2}{c_1}+\frac{(1+3^{p-1}2)c_4}{c_3}\big)$; then, 
summing up in~\eqref{4est2} over all $i$, gathering~\eqref{est:fund-est-Fb1}--\eqref{est:fund-est-Fb3}, \eqref{est:fund-est-Fs2}, and~\eqref{est:fund-est-Fs4}, by averaging we find an index $i^*\in\{3,\ldots,N\}$ such that
\begin{align*}
\F_k(\hat{u}^{i^*},\hat{v}^{i^*},A\cup B) &\leq\frac{1}{N-2} \sum_{i=3}^{N}\F_k (\hat u^i, \hat v^i, A\cup B)\\
&\leq\Big(1+\frac{\widehat{M}}{N-2}\Big)\big(\F_k(u,v,A')+\F_k(\tilde{u},\tilde{v},B)\big)\\
&+\frac{3^{p-1}M^pc_2}{N-2}\int_{S}|u-\tilde{u}|^p\dx+\frac{3^{p-1}M^p c_4}{N-2}\int_{S}\e_k^{p-1}|v-\tilde{v}|^p\dx\,.
\end{align*}
Thus, upon choosing $N$ large enough so that $\frac{\widehat{M}}{N-2}<\eta$, since $0\leq v,\tilde{v}\leq 1$ we obtain~\eqref{fund-est} by setting $(\hat{v},\hat{v})\defas(\hat{u}^{i^*},\hat{v}^{i^*})$ and $M_\eta\defas \frac{3^{p-1}M^p}{N-2}\big(c_2+2c_4\L^n(S)\big)$.
\end{proof}
On account of the fundamental estimate, Proposition \ref{prop:fund-est}, we are now in a position to prove the following $\Gamma$-convergence result.  
\begin{theorem}\label{thm:int-rep}
Let $\F_k$ be as in~\eqref{F_e}. Then there exist a subsequence $(\F_{k_j})$ of $(\F_k)$ and a functional $\F\colon L^0(\R^n;\R^m)\times L^0(\R^n)\times\A\longrightarrow [0,+\infty]$ such that for every $A\in\A$ the functionals $\F_{k_j}(\cdot\, ,\cdot\, ,A)$ $\Gamma$-converge in $L^0(\R^n;\R^m)\times L^0(\R^n)$ to $\F(\cdot\,,\cdot\,,A)$. Moreover, $\F$ is given by
\begin{equation*}
\F(u,v,A)\defas
\begin{cases}
\displaystyle\int_A \hat f(x,\nabla u)\dx+\int_{S_u\cap A}\hat g(x,[u],\nu_u)\dHn\, &\text{if}\ u \in GSBV^p(A;\R^m)\, ,\ v= 1\, \text{a.e.\ in } A\, ,\\
+\infty &\text{otherwise}\,,
\end{cases}
\end{equation*}
with $\hat{f}\colon\R^n\x\R^{m\x n}\to[0,+\infty)$, $\hat{g}\colon \R^n\x\R_0^m\x\S^{n-1}\to[0,+\infty)$ given by
\begin{align}
\hat f(x,\xi) &:=\limsup_{\rho\to 0}\frac{1}{\rho^{n}}\m(u_\xi, Q_\rho(x))\,,\label{eq:f-hat}\\
\hat g(x,\nu,\zeta) &:=\limsup_{\rho\to 0}\frac{1}{\rho^{n-1}}\m(u_{x,\zeta}^\nu,Q_\rho^\nu(x))\,,\label{eq:g-hat}
\end{align}
for every $x\in \R^n$, $\xi\in\R^{m\times n}$, $\zeta\in \R_0^m$, and $\nu\in \Sph^{n-1}$,
where for $A\in\A$ and $\bar{u}\in SBV^p(A;\R^m)$
\begin{align*}\label{def:m}
\m(\bar{u},A)\defas\inf\{\F(u,1,A)\colon u\in SBV^p(A;\R^m)\, ,\ u=\bar{u}\ \text{near}\ \partial A\}\,.
\end{align*}
\end{theorem}
\begin{proof}
Let $\F',\F'' \colon L^0(\R^n;\R^m)\times L^0(\R^n)\times \A \longrightarrow [0,+\infty]$ be the functionals defined as
\begin{equation*}
\F'(\,\cdot\,,\,\cdot\,,A) \defas \Gamma\hbox{-}\liminf_{k \to +\infty}\F_{k}(\,\cdot\,,\,\cdot\,,A)\quad \text{and}\quad
\F''(\,\cdot\,,\,\cdot\,,A) \defas \Gamma\hbox{-}\limsup_{k \to +\infty}\F_{k}(\,\cdot\,,\,\cdot\,,A)\,.
\end{equation*}
In view of Remark \ref{rem:bounds-Fe} we can invoke \cite[Theorem 3.1]{Foc01} to deduce the existence of a constant $C>0$ such that
\begin{equation}\label{est:bounds-F'-F''}
\begin{split}
\frac{1}{C}\Big(\int_A|\nabla u|^p\dx+\H^{n-1}(S_u\cap A)\Big) &\leq\F'(u,1,A)\\
&\leq\F''(u,1,A)\leq C\Big(\int_A|\nabla u|^p\dx+\H^{n-1}(S_u\cap A)\Big)\,,
\end{split}
\end{equation}
for every $A\in\A$ and every $u\in GSBV^p(A;\R^m)$; moreover 
\begin{equation}\label{eq:out-domain}
\F'(u,v,A)=\F''(u,v,A)= +\infty \quad \text{if either } \; u\notin GSBV^p(A;\R^m) \; \text{ or } \; v\neq 1\,. 
\end{equation}
By the general properties of $\Gamma$-convergence we know that for every $A\in \A$ fixed the functionals $\F'(\cdot,\cdot,A)$ and $\F''(\cdot,\cdot,A)$ are $\LtL$ lower semicontinuous \cite[Proposition 6.8]{DM} and local \cite[Proposition 16.15]{DM}. 
%
Further, the set functions $\F'(u,v,\cdot)$ and $\F''(u,v,\cdot)$ are increasing \cite[Proposition 6.7]{DM} and $\F'(u,v,\cdot)$ is superadditive \cite[Proposition 16.12]{DM}. 

Invoking \cite[Theorem 16.9]{DM} we can deduce the existence of a subsequence $(k_j)$, with $k_j \to +\infty$ as $j \to +\infty$, such that the corresponding $\F'$ and $\F''$ also satisfy 
\begin{equation}\label{eq:Gamma-bar}
\sup_{A'\subset \subset A,\, A'\in \A}\F'(u,v,A')\quad =\sup_{A'\subset \subset A,\, A'\in \A}\F''(u,v,A')=:\F(u,v,A)\,,
\end{equation}
for every $(u,v)\in L^0(\R^n;\R^m)\times L^0(\R^n)$ and for every $A\in \A$. We notice that the set function $\F(u,v,\cdot)$ given by \eqref{eq:Gamma-bar} is inner regular by definition. Moreover $\F$ satisfies the following properties: the functional $\F(\cdot,\cdot,A)$ is $\LtL$ lower semicontinuous \cite[Remark 15.10]{DM} and local \cite[Remark 15.25]{DM}, while the set function $\F(u,v,\cdot)$ is increasing and superadditive \cite[Remark 15.10]{DM}.

Thanks to the fundamental estimate Proposition \ref{prop:fund-est} we can appeal to \cite[Proposition 18.4]{DM} to deduce that $\F(u,v,\cdot)$ is also a subadditive set function. Here the only difference with a standard situation is that the reminder in \eqref{fund-est} is infinitesimal with respect to the $L^p(\R^n;\R^m)$ convergence in $u$ while we are considering the $\Gamma$-convergence of $\F_{k_j}$ in $L^0(\R^n;\R^m)\times L^0(\R^n)$. However, this issue can be easily overcome by resorting to a truncation argument together with a sequential characterisation of $\F$ (see \textit{e.g.}, \cite[Proposition 16.4 and Remark 16.5]{DM}), which holds true on $SBV^p(A;\R^m)\cap L^\infty(A;\R^m)$. 
Hence, we can now invoke the measure-property criterion of De Giorgi and Letta (see \textit{e.g.}, \cite[Theorem 14.23]{DM}) to deduce that for every $(u,v)\in L^0(\R^n;\R^m) \times L^0(\R^n)$ the set function $\F(u,v,\cdot)$ is the restriction to $\A$ of a Borel measure.   

Furthermore, \eqref{est:bounds-F'-F''} together with \cite[Proposition 18.6]{DM} and Proposition~\ref{prop:fund-est} imply that 
\begin{equation*}
\F(u,1,A)=\F'(u,1,A)=\F''(u,1,A) \quad \text{if}\quad u\in GSBV^p(A;\R^m)\,,
\end{equation*}
while, gathering \eqref{est:bounds-F'-F''} and \eqref{eq:out-domain} we may also deduce that 
\begin{equation*}
\F(u,v,A)=\F'(u,v,A)=\F''(u,v,A)=+\infty \quad \text{if either } \; u\notin GSBV^p(A;\R^m) \; \text{ or } \; v\neq 1\,. 
\end{equation*}
As a consequence, $\F(\cdot,\cdot, A)$ coincides with the $\Gamma$-limit of $\F_{k_j}(\cdot,\cdot,A)$ on $L^0(\R^n;\R^m)\times L^0(\R^n)$, for every $A\in \A$.

By \cite[Theorem 1]{BFLM02} and a standard perturbation and truncation argument (see \textit{e.g.}, \cite[Theorem 4.3]{CDMSZ19}), for every $A\in \A$ and $u\in GSBV^p(A;\R^m)$ we can represent the $\Gamma$-limit $\F$ in an integral form as  
\begin{equation*}
\F(u,1,A)=
\int_A \hat f(x,\nabla u)\dx+\int_{S_u}\hat g(x,[u],\nu_u)\dHn\,,
\end{equation*}
for some Borel functions $\hat f$ and $\hat g$.
Eventually, thanks to \eqref{est:bounds-F'-F''}, it can be easily shown that $\hat f$ and $\hat g$ are given by the same derivation formulas provided by \cite[Theorem 1]{BFLM02}, that is, they coincide with~\eqref{eq:f-hat} and \eqref{eq:g-hat}, respectively.
\end{proof}

%
%
\section{Identification of the volume integrand}\label{sect:bulk}
\noindent In this section we identify the volume integrand $\hat f$. Namely, we prove that $\hat f$ coincides with both $f'$ and $f''$, given by \eqref{f'} and \eqref{f''}, respectively. This shows, in particular, that the limit volume integrand $\hat f$ depends only on $f_k$, and hence only on $\F^b_k$.  

\begin{proposition}\label{p:volume-term}
Let $(f_k) \subset \mathcal F$ and  $(g_k)\subset \mathcal G$. Let $(k_j)$ and  $\hat f$ be as in Theorem~\ref{thm:int-rep}. Then it holds 
\[
\hat f(x,\xi)=f'(x,\xi) = f''(x,\xi),
\] 
for a.e.\ $x\in \R^n$ and for every $\xi \in \R^{m\x n}$,  where $f'$ and $f''$ are, respectively, as in \eqref{f'} and \eqref{f''} with $k$ replaced by $k_j$. 
\end{proposition}

\begin{proof}

For notational simplicity, in what follows we still let $k$ denote the index of the sequence provided by Theorem~\ref{thm:int-rep}. 

\medskip

By definition $f'\leq f''$, hence to prove the claim it suffices to show that
\begin{equation}\label{f-hat-claim}
f''(x,\xi)\leq \hat f(x,\xi) \leq f' (x,\xi)\,,
\end{equation}
for a.e.\ $x\in \R^n$ and for every $\xi \in \R^{m\x n}$. 
We divide the proof of \eqref{f-hat-claim} into two steps.

\medskip

\emph{Step 1:} In this step we show that $\hat f(x,\xi) \geq f''(x,\xi)$ for a.e.\ $x\in \R^n$ and for every $\xi \in \R^{m\x n}$.

\smallskip

By Theorem~\ref{thm:int-rep} we have that 
\begin{equation}\label{eq:G-lim-vol}
\int_A\hat f(x,\xi)\dx= \F(u_\xi,1,A)\,,
\end{equation}
for every $A\in \mathcal A$ and for every $\xi \in \R^n$. 

Now let $x \in \R^n$ be arbitrary, let $A\in \A$ be such that $x\in A$, and let $\rho>0$ be so small that $Q_\rho(x)\subset A$. 
By $\Gamma$-convergence we can find $(u_k,v_k) \subset L^0(\R^n;\R^m) \times L^0(\R^n)$ which is a recovery sequence for $\F(u_\xi,1,A)$\ie $(u_k, v_k) \subset W^{1,p}(A;\R^m) \times W^{1,p}(A)$, $0\leq v\leq 1$, $(u_k,v_k)$ converges to $(u_\xi,1)$ in measure on bounded sets and  
\begin{equation}\label{eq:hat-f-rs}
\lim_{k \to +\infty}\F_{k}(u_k,v_k,A)=\F(u_\xi,1,A)\,.
\end{equation}
Moreover, by \ref{hyp:lb-g} we also have $v_k \to 1$ in $L^p(A)$. 

We notice that $(u_k,v_k)$ also satisfies
\begin{equation}\label{eq:rs-cubes}
\lim_{k \to +\infty}\F_{k}(u_k,v_k,Q_{\rho}(x))=\F(u_\xi,1,Q_{\rho}(x))\, . 
\end{equation}
Indeed, thanks to~\eqref{eq:G-lim-vol} we have
\begin{equation*}
\F(u_\xi,1,A)= \F(u_\xi,1,Q_{\rho}(x))+\F(u_\xi,1,A \setminus \overline Q_{\rho}(x))\, .
\end{equation*}
Therefore, again by $\Gamma$-convergence we get
\begin{equation*}
\liminf_{k\to +\infty} \F_{k}(u_k,v_k,Q_{\rho}(x))\geq \F(u_\xi,1,Q_{\rho}(x))
\end{equation*}
and
\begin{equation*}
\liminf_{k\to +\infty} \F_{k}(u_k,v_k,A \setminus \overline Q_{\rho}(x)) \geq \F(u_\xi,1,A \setminus \overline Q_{\rho}(x))\, .
\end{equation*}
Hence \eqref{eq:rs-cubes} follows by \eqref{eq:hat-f-rs}.

We now estimate separately the surface and bulk term in $\F_{k}$.  We notice that by Young's Inequality we have
\[
\frac{p-1}{p}\frac{(1-v_k)^p}{\e_k}+\frac{1}{p}\e_k^{p-1}|\nabla v_k|^p \geq (1-v_k)^{p-1}|\nabla v_k|\,. 
\]
Hence, thanks to~\ref{hyp:lb-g}, by using the co-area formula we get that
\begin{align*}
\F_{k}^s(v_k,Q_\rho(x))
&\geq c_3 \int_{Q_\rho(x)}(1-v_k)^{p-1}|\nabla v_k|\dy = c_3 \int_0^{1} (1-t)^{p-1}\, \mathcal H^{n-1}(\partial^*E^t_{k,\rho})\dt\,, 
\end{align*}
where $E_{k,\rho}^t\defas\{y\in Q_\rho(x)\colon v_k(y)<t\}$.
Now let $\eta \in (0,1)$ be fixed, then by the mean-value Theorem we deduce the existence of $\bar t=\bar t(k,\rho,\eta)$,  $\bar t \in (\eta,1)$ such that   
\begin{align*}
\F_{k}^s(v_k,Q_\rho(x)) \geq c_3\int_\eta^{1} (1-t)^{p-1} \dt\, \mathcal H^{n-1}(\partial^*E^{\bar t}_{k,\rho})\,. 
\end{align*}
Set $w_k:=u_k \chi_{\R^n \setminus E^{\bar t}_{k,\rho}}$; since $u_k\in W^{1,p}(A;\R^m)$ and $E^{\bar t}_{k,\rho}$ is a set of finite perimeter, we have that $w_k \in GSBV^p(A;\R^m)$ and $\mathcal H^{n-1}(S_{w_k})\leq \mathcal H^{n-1}(\partial^*E^{\bar t}_{k,\rho})$. Moreover, since $v_k \to 1$ in $L^p(A)$ we have that $\mathcal L^n(E^{\bar t}_{k,\rho}) \to 0$ as $k \to +\infty$ so that $w_k \to u_\xi$ in measure on bounded sets. Eventually, since $\F_{k}$ is increasing as set function, using the fact that $\psi$ is increasing we obtain
\begin{align*}
\F_{k}(u_k,v_k,Q_{\rho}(x)) &= \F_{k}^b(u_k,v_k,Q_\rho(x))+\F_{k}^s(v_k,Q_\rho(x))\\[4pt] 
& \geq \psi(\eta)\int_{Q_\rho(x)\setminus E^{\bar t}_{k,\rho}}\hspace*{-1em} f_{k}(y,\nabla u_k)\dy+ c_3\int_\eta^{1} (1-t)^{p-1}\dt\, \mathcal H^{n-1}(\partial^*E^{\bar t}_{k,\rho})\,\\
& \geq \psi(\eta)\int_{Q_\rho(x)}\hspace*{-1em} f_{k}(y,\nabla w_k)\dy+ c_3\int_\eta^{1} (1-t)^{p-1}\dt\, \mathcal H^{n-1}(S_{w_k})\, ,
\end{align*}
where the last inequality follows from the definition of $w_k$. 

In view of~\ref{hyp:meas-f}--\ref{hyp:cont-xi-f} the right-hand side belongs to the class of functionals considered in~\cite{CDMSZ19}. 
Then, thanks to \cite[Theorem 3.5 and Theorem 5.2 (b)]{CDMSZ19} we have 
\begin{align*}
\lim_{k \to +\infty}\F_{k}(u_{k},v_{k},Q_{\rho}(x)) \geq \psi(\eta)\int_{Q_\rho(x)}f''(y,\xi)\dy\,, 
\end{align*}
where $f''$ is given by \eqref{f''}. 

Hence appealing to \eqref{eq:rs-cubes} gives
\begin{equation*}
\F(u_\xi,1,Q_{\rho}(x))\geq \psi(\eta)\int_{Q_\rho(x)}f''(y,\xi)\dy\,.
\end{equation*}
Then, by dividing both terms in the above inequality by $\rho^n$ and using \eqref{eq:G-lim-vol}, we obtain

\begin{align*}
\frac{1}{\rho^n}\int_{Q_\rho(x)} \hat f(y,\xi)\dy \geq \psi(\eta)\, \frac{1}{\rho^n}\int_{Q_\rho(x)} f''(y,\xi)\dy\,. 
\end{align*}
Thus invoking the Lebesgue differentiation Theorem together with the continuity in $\xi$ of $\hat f$ and $f''$  (see \cite{BFLM02} and Proposition \ref{prop:f'-f''}) \EEE we deduce that 
\begin{equation*}
\hat f(x,\xi)\geq \psi(\eta)\, f''(x,\xi),
\end{equation*}
for a.e.\ $x\in \R^{n}$, for every $\xi \in \R^{m\times n}$, and for every $\eta \in (0,1)$. Eventually, since $\psi$ is continuous and $\psi(1)=1$, the claim follows by letting $\eta \to 1$. 

\medskip

\emph{Step 2:} In this step we show that $\hat f(x,\xi) \leq f'(x,\xi)$ for every $x\in \R^n$ and for every $\xi \in \R^{m\x n}$. The proof is similar to that of~\cite[Theorem 5.2]{CDMSZ19}. However, we repeat it here for the readers' convenience.

Let $x\in\R^n$, $\xi\in\R^{m\x n}$, $\rho>0$, and $\eta>0$ be fixed. By \eqref{eq:mb} for every $k\in\N$ fixed we can find $u_k\in W^{1,p}(Q_\rho(x);\R^m)$ with $u_k=u_\xi$ near $\partial Q_\rho(x)$ such that
\begin{equation}\label{eq:am-mb}
\F_{k}^b(u_k,1, Q_\rho(x))\leq\m_{k}^b(u_\xi,Q_\rho(x))+\eta\rho^n\, .
\end{equation}
Combining~\eqref{eq:am-mb} with~\ref{hyp:lb-f} and~\ref{hyp:ub-f} yields
\begin{equation}\label{ub:grad-uk}
c_1\|\nabla u_k\|_{L^{p}(Q_\rho(x);\R^{m\x n})}^p\leq\F_{k}^b(u_k,1,Q_\rho(x))\leq \rho^n(c_2|\xi|^p+\eta)\,,
\end{equation}
where the second inequality follows by taking $u_\xi$ as a test in the definition of $\m_{k}^b(u_\xi,Q_\rho(x))$.
Let now $(k_j)$ be a diverging sequence such that
\begin{align*}
\lim_{j\to+\infty}\F_{{k_j}}^b(u_{k_j},1,Q_\rho(x))=\liminf_{k\to +\infty}\F_{k}^b(u_k,1,Q_\rho(x))\, .
\end{align*}
Since $u_{k_j}-u_\xi\in W_0^{1,p}(Q_\rho(x);\R^m)$, the uniform bound~\eqref{ub:grad-uk} together with the Poincar\'{e} Inequality provide us with a further subsequence (not relabelled) and a function $u\in W^{1,p}(Q_\rho(x);\R^m)$ such that $u_{k_j}\wto u$ weakly in $W^{1,p}(Q_\rho(x);\R^m)$. Then, by the Rellich Theorem $u_{k_j}\to u$ in $L^p(Q_\rho(x);\R^m)$.
We now extend $u$ and $u_k$ to functions $w,w_k\in W^{1,p}_{\loc}(\R^n;\R^m)$ by setting 
\begin{align*}
w\defas
\begin{cases}
u &\text{in}\ Q_\rho(x)\,,\\
u_\xi &\text{in}\ \R^n\sm Q_\rho(x)\,,
\end{cases}
\quad
w_k\defas
\begin{cases}
u_k &\text{in}\ Q_\rho(x)\,,\\
u_\xi &\text{in}\ \R^n\sm Q_\rho(x)\,,
\end{cases}
\end{align*}
respectively; clearly, $w=u_\xi$ in a neighbourhood of $\partial Q_{(1+\eta)\rho}(x)$ and $w_{k_j}\to w$ in $L^p_{\loc}(\R^n;\R^m)$. Hence by $\Gamma$-convergence, by \eqref{g-value-at-0} and~\eqref{eq:am-mb} we get
\begin{equation*}
\begin{split}
\m(u_\xi,Q_{(1+\eta)\rho}(x))&\leq\F(w,1,Q_{(1+\eta)\rho}(x))\leq\lim_{j\to+\infty}\F_{{k_j}}(w_{k_j},1,Q_{(1+\eta)\rho}(x))\\
&\leq\lim_{j\to+\infty}\F_{{k_j}}^b(u_{k_j},Q_\rho(x))+c_2|\xi|^p\big((1+\eta)^n-1\big)\rho^n\\
&\leq\liminf_{k\to+\infty}\m_{k}^b(u_\xi,Q_\rho(x))+\eta\rho^n+c_2|\xi|^p\big((1+\eta)^n-1\big)\rho^n\,.
\end{split}
\end{equation*} 
Eventually, dividing by $\rho^n$, passing to the limsup as $\rho \to 0$, and recalling the definition of $\hat f$ and $f'$ we get
\begin{align*}
(1+\eta)^n\hat{f}(x,\xi)\leq f'(x,\xi)+\eta+c_2|\xi|^p\big((1+\eta)^n-1\big)\, ,
\end{align*}
and hence the claim follows by the arbitrariness of $\eta>0$.
\end{proof}

 \section{Identification of the surface integrand}\label{sect:surface} \EEE
\noindent In this section we identify the surface integrand $\hat g$. Namely, we show that $\hat g$ coincides with both $g'$ and $g''$, given by \eqref{g'} and \eqref{g''}, respectively. This shows, in particular, that the limit surface integrand $\hat g$ is obtained by minimising only the surface term $\F_k^s$. We notice, however, that in this case the presence the bulk term $\F_k^b$ affects the class of test functions over which the minimisation is performed (cf. \eqref{eq:ms}-\eqref{c:adm-e}).  

\medskip

We start by proving some preliminary lemmas. The first lemma concerns the approximation of a minimisation problem involving the $\Gamma$-limit $\F$. 
\begin{lemma}[Approximation of minimum values]\label{lem:approx-min}
Let $(f_k) \subset \mathcal F$ and $(g_k)\subset \mathcal G$.
Let $\rho >0$; for $x\in\R^n$, $\zeta\in\R_0^m$, $\nu\in\Sph^{n-1}$, and $k\in\N$ such that $2\e_k<\rho$ set
\begin{multline*}
\m_k(\bar u_{x,\zeta,\e_k}^\nu,Q_\rho^\nu(x)) \defas\inf\{\F_k(u,v,Q_\rho^\nu(x))\colon (u,v)\in W^{1,p}(Q_\rho^\nu(x);\R^m)\times W^{1,p}(Q_\rho^\nu(x)),\; \\[4pt]
0\leq v\leq 1,
(u,v)=(\bar u_{x,\zeta,\e_k}^{\nu} ,\bar v_{x,\e_k}^{\nu})\ \text{near}\ \partial Q_\rho^\nu(x)\}\, .
\end{multline*}
Let $(k_j)$ be as in Theorem~\ref{thm:int-rep} and $\hat{g}$ be as in~\eqref{eq:g-hat}.
Then for every $x\in\R^n$, $\zeta\in\R_0^m$, and $\nu\in\Sph^{n-1}$ it holds
\begin{align}\nonumber
\hat{g}(x,\zeta,\nu)&=\limsup_{\rho\to0}\frac{1}{\rho^{n-1}}\liminf_{j\to+\infty}\m_{k_j}(\bar u_{x,\zeta,\e_{k_j}}^\nu,Q_\rho^\nu(x))\\\label{eq:limit-m-me}
&=\limsup_{\rho\to0}\frac{1}{\rho^{n-1}}\limsup_{j\to+\infty}\m_{k_j}(\bar u_{x,\zeta,\e_{k_j}}^\nu,Q_\rho^\nu(x)).
\end{align}
\end{lemma}
\begin{proof}

For notational simplicity, in what follows we still denote with $k$ the index of the (sub)sequence provided by Theorem~\ref{thm:int-rep}.

\medskip

We divide the proof into two steps.

\medskip

\emph{Step 1:} In this step we show that
\begin{equation}\label{lb:m-me}
\hat{g}(x,\zeta,\nu)\leq\limsup_{\rho\to 0}\frac{1}{\rho^{n-1}}\liminf_{k\to+\infty}\m_k(\bar u_{x,\zeta,\e_k}^\nu,Q_\rho^\nu(x)),
\end{equation}
for every $x\in\R^n$, $\zeta\in\R_0^m$, and $\nu\in\Sph^{n-1}$.\EEE

Let $\rho>0$ and $\eta>0$ be fixed; by definition of $\m_k(\bar u_{x,\zeta,\e_k}^\nu,Q_\rho^\nu(x))$ there exists $(u_k,v_k) \subset W^{1,p}(Q_\rho^\nu(x);\R^m)\times W^{1,p}(Q_\rho^\nu(x))$ such that $(u_k,v_k)=(\bar u_{x,\zeta,\e_k}^\nu\, ,\bar v_{x,\e_k}^\nu)$ in a neighbourhood of $\partial Q_\rho^\nu(x)$, and
\begin{equation}\label{eq:q-min-e}
\F_{k}(u_k,v_k,Q_\rho^\nu(x)) \leq \m_k(\bar u_{x,\zeta,\e_k}^\nu,Q_\rho^\nu(x)) +\eta\rho^{n-1}.
\end{equation}
Since the pair $(\bar u_{x,\zeta,\e_k}^\nu\,,\bar v_{x,\e_k}^\nu)$ is admissible for $\m_k(\bar u_{x,\zeta,\e_k}^\nu,Q_\rho^\nu(x))$, then~\eqref{prop:barue-barve} and~\eqref{1dim-energy-bis} readily give  
\begin{equation}\label{uniformbound:me}
\m_k(\bar u_{x,\zeta,\e_k}^\nu,Q_\rho^\nu(x))\leq\F_k(\bar{u}_{x,\zeta,\e_k}^\nu,\bar{v}_{x,\e_k}^\nu,Q_\rho^\nu(x))=\F_k^s(\bar{v}_{x,\e_k}^\nu,Q_\rho^\nu(x))\leq c_4C_\vv\rho^{n-1}.
\end{equation}
By a truncation argument (see \textit{e.g.}, \cite[Lemma 3.5]{BDf} or \cite[Lemma 4.1]{CDMSZ19}) it is not restrictive to assume that $\sup_k \|u_k\|_{L^\infty(Q^\nu_{\rho}(x);\R^m)}<+\infty$. 
We now extend $u_k$ to a $W^{1,p}_{\loc}(\R^n;\R^m)$-function by setting
\begin{align*}
w_k\defas
\begin{cases}
u_k &\text{in } Q_\rho^\nu(x)\,,\\
\bar u_{x,\zeta,\e_k}^\nu &\text{in } \R^n\sm Q_\rho^\nu(x)\,,
\end{cases}
\end{align*}
then, clearly, $\sup_k \|w_k\|_{L^\infty(\R^n;\R^m)}<+\infty$.
Now let $({k_j})$ be such that 
\[
\lim_{j\to +\infty} \F_{{k_j}}(u_{k_j},v_{k_j},Q_\rho^\nu(x)) = \liminf_{k \to +\infty}\F_{k}(u_k,v_k,Q_\rho^\nu(x))\,.
\]
In view of~\eqref{uniformbound:me}, \eqref{est:bounds-Fe}, and the uniform $L^\infty(\R^n;\R^m)$-bound on $w_{k_j}$ we can invoke \cite[Lemma 4.1]{Foc01} to deduce the existence of a subsequence (not relabelled) such that
\[(w_{k_j}, v_{k_j}) \to (u, 1)\quad  \text{in}\quad  L^p_{\loc}(\R^n;\R^m)\times L^p_{\loc}(\R^n)\,,\]
for some $u\in L^p_{\loc}(\R^n;\R^m)$ also belonging to $SBV^p(Q^\nu_{(1+\eta)\rho}(x);\R^m)$. Moreover, we also have $u=u_{x,\zeta}^\nu$ in a neighbourhood of $\partial Q^\nu_{(1+\eta)\rho}(x)$, so that
\begin{equation}\label{eq:trivial-ineq}
\m(u^\nu_{x,\zeta},Q^\nu_{(1+\eta)\rho}(x)) \leq \F(u,1,Q^\nu_{(1+\eta)\rho}(x))\,. 
\end{equation}
Eventually, by $\Gamma$-convergence together with~\eqref{eq:q-min-e} we obtain
\begin{align*}
\F(u,1,Q^\nu_{(1+\eta)\rho}(x)) &\leq \liminf_{j\to +\infty} \F_{{k_j}}(w_{k_j},v_{k_j},Q_{(1+\eta)\rho}^\nu(x))\\
&\leq\lim_{j \to +\infty} \F_{k_j}(u_{k_j},v_{k_j},Q_{\rho}^\nu(x))+c_4C_\vv\big((1+\eta)^{n-1}-1)\rho^{n-1}\\
&\leq \liminf_{k\to+\infty}\m_k(\bar u_{x,\zeta,\e_k}^\nu,Q_\rho^\nu(x))+\eta\rho^{n-1}+c_4C_\vv\big((1+\eta)^{n-1}-1\big)\rho^{n-1}\,.
\end{align*}
Thus, using~\eqref{eq:trivial-ineq}, dividing the above inequality by $\rho^{n-1}$ and passing to the limsup as $\rho\to0$ we obtain
\begin{equation*}
(1+\eta)^{n-1}\hat{g}(x,\zeta,\nu)\leq\limsup_{\rho\to0}\frac{1}{\rho^{n-1}}\liminf_{k\to+\infty}\m_k(\bar u_{x,\zeta,\e_k}^\nu,Q_\rho^\nu(x))+\eta+c_4C_\vv\big((1+\eta)^{n-1}-1\big)\,,
\end{equation*}
hence~\eqref{lb:m-me} follows by the arbitrariness of $\eta>0$.

\medskip

\emph{Step 2:} In this step we show that
\begin{equation}\label{ub:m-me}
\limsup_{\rho\to 0}\frac{1}{\rho^{n-1}}\limsup_{k\to+\infty}\m_k(\bar u_{x,\zeta,\e_k}^\nu,Q_\rho^\nu(x)) \leq \hat{g}(x,\zeta,\nu),
\end{equation}
for every $x\in\R^n$, $\zeta\in\R_0^m$, $\nu\in\Sph^{n-1}$, and $\rho>0$.
To this end, we fix $\eta>0$ and we choose $u\in SBV^p(Q_\rho^\nu(x);\R^m)$ with $u=u_{x,\zeta}^\nu$ near $\partial Q_\rho^\nu(x)$ and
\begin{equation}\label{eq:q-min-m}
\F(u,1,Q_\rho^\nu(x))\leq\m(u_{x,\zeta}^\nu,Q_\rho^\nu(x))+\eta\,.
\end{equation}
We extend $u$ to the whole $\R^n$ by setting $u=u_{x,\zeta}^\nu$ in $\R^n\sm Q_\rho^\nu(x)$. 
Then, by $\Gamma$-convergence there exists a sequence $(u_k,v_k)$ converging to $(u,1)$ in measure on bounded sets such that 
\begin{equation}\label{approx:min-rec}
\lim_{k \to +\infty}\F_{k}(u_k,v_k,Q_\rho^\nu(x))= \F(u,1,Q_\rho^\nu(x))\,.
\end{equation}
We notice, moreover, that thanks to a truncation argument (both on $u$ and $u_k$) and to the bound \ref{hyp:lb-g}, it is not restrictive to assume that  
$(u_k,v_k)$ converges to $(u,1)$ in $L^p_{\rm loc}(\R^n;\R^m) \times L^p_{\rm loc}(\R^n)$.
We now modify the sequence $(u_k,v_k)$ in such a way that it satisfies the boundary conditions required in the definition of $\m_k(\bar u_{x,\zeta,\e_k}^\nu,Q_\rho^\nu(x))$. This will be done by resorting to the fundamental estimate Proposition~\ref{prop:fund-est}.
Namely, we choose $0<\rho''<\rho'<\rho$ such that $u=u_{x,\zeta}^\nu$ on $Q_\rho^\nu(x)\sm\overline{Q}_{\rho''}^\nu(x)$ and we apply Proposition~\ref{prop:fund-est} with $A=Q_{\rho''}^\nu(x)$, $A'=Q_{\rho'}^\nu(x)$, $B=Q_\rho^\nu(x)\sm\overline{Q}_{\rho''}^\nu(x)$. In this way, we obtain a sequence $(\hat{u}_k,\hat{v}_k)\subset W^{1,p}_\loc(\R^n;\R^m)\times W^{1,p}_\loc(\R^n)$ converging to $(u,v)$ in $L^p(Q_\rho^\nu(x);\R^m)\times L^p(Q_\rho^\nu(x))$ such that $(\hat{u}_k,\hat{v}_k)=(u_k,v_k)$ in $Q_{\rho''}^\nu(x)$, $(\hat{u}_k,\hat{v}_k)=(\bar u_{x,\zeta,\e_k}^\nu,\bar v_{x,\e_k}^\nu)$ in $Q_\rho^\nu(x)\sm\overline{Q}_{\rho'}^\nu(x)$, and
\begin{equation}\label{approx:min-bc}
\begin{split}
\limsup_{k\to+\infty}\F_{k}(\hat{u}_k,\hat{v}_k,Q_\rho^\nu(x))&\leq (1+\eta)\limsup_{k\to+\infty}\Big(\F_{k}(u_k,v_k,Q_{\rho'}^\nu(x))+\F_{k}^s(\bar v_{x,\e_k}^\nu, Q_{\rho}^\nu(x)\sm\overline{Q}_{\rho''}^\nu(x))\Big)\\
&\leq (1+\eta)\F(u,1,Q_\rho^\nu(x)) +c_4C_\vv\L^{n-1}\big(Q_{\rho}'\sm\overline{Q'}_{\!\!\rho''}\big)\,,
\end{split}
\end{equation}
where the second inequality follows  from~\eqref{approx:min-rec} and~\eqref{1dim-energy-bis} together with~\eqref{g-value-at-0}, respectively.
 
Eventually, since $(\hat u_k,\hat v_k)$ is admissible in the definition of $\m_k(\bar u_{x,\zeta,\e_k}^\nu,Q_\rho^\nu(x))$, gathering~\eqref{eq:q-min-m} and~\eqref{approx:min-bc} we deduce that
\begin{align*}
\limsup_{k\to+\infty}\m_k(\bar u_{x,\zeta,\e_k}^\nu,Q_\rho^\nu(x))\EEE &\leq (1+\eta)\big(\m(u_{x,\zeta}^\nu,Q_\rho^\nu(x))+\eta\big)+C  (\rho^{n-1}-(\rho'')^{n-1})\,.
\end{align*}
Then letting $\rho''\to\rho$, by the arbitrariness of $\eta>0$ we get
\begin{equation}\label{ub:m-me-1}
\limsup_{k\to+\infty}\m_k(\bar u_{x,\zeta,\e_k}^\nu,Q_\rho^\nu(x))\leq\m(u_{x,\zeta}^\nu,Q_\rho^\nu(x))\,,
\end{equation}
therefore \eqref{ub:m-me} follows by dividing both sides of \eqref{ub:m-me-1} by $\rho^{n-1}$ and passing to the limsup as $\rho \to 0$. 
\end{proof}

\begin{remark}\label{rem:smooth-pair}
Thanks to the $(W^{1,p}(A;\R^m)\times W^{1,p}(A))$-continuity of $\F_k(\,\cdot\,,\,\cdot\,,A)$ (cf. Remark~\ref{rem:continuity}), a standard convolution argument shows that~\eqref{eq:limit-m-me} holds also true if the minimisation in $\m_k$ is carried over $C^1(Q_\rho^\nu(x);\R^m)\times C^1(Q_\rho^\nu(x))$.
\end{remark}

The following lemma shows that if $v$ is ``small'' in some region, then it can be replaced by a function which is equal to zero in that region, without essentially increasing $\F_k^s$.

\begin{lemma}\label{lem:auxiliary-function}
Let  $(g_k)\subset \mathcal G$, \EEE $A\in\A$, $v\in W^{1,p}(A)$, and $\eta\in(0,1)$ be fixed. Let $v^\eta\in W^{1,p}(A)$ be defined as
\begin{align}\label{def:a-f}
v^\eta\defas\min\{((1+\sqrt{\eta})v-\eta)^+\, ,\ v\}.
\end{align}
Then for a.e.\ $x\in A$ we have
\begin{equation}\label{property:a-f}
v^\eta(x)=0\;\text{ iff }\; v(x)\leq\frac{\eta}{1+\sqrt{\eta}}\qquad\text{and}\qquad v^\eta(x)=v(x)\;\text{ iff }\; v(x)\geq\sqrt{\eta}\,.
\end{equation}
Moreover, $v^\eta$ satisfies
\begin{equation}\label{estimate:a-f}
\F_k^s(v^\eta,A)\leq (1+c_\eta)\F_k^s(v,A)\, ,
\end{equation}
where $c_\eta>0$ is independent of $v$ and $A$ and such that $c_\eta\to 0$ as $\eta\to 0$.
\end{lemma}
\begin{proof}
A direct computation shows that $((1+\sqrt{\eta})v-\eta)^+\le v$ if and only if $v\le\sqrt{\eta}$ and that $((1+\sqrt{\eta})v-\eta)^+= 0$ if and only if $v\leq\frac{\eta}{1+\sqrt{\eta}}$\ie \eqref{property:a-f} is satisfied. Thus it remains to show that~\eqref{estimate:a-f} holds true. To this end we introduce the sets
\begin{align*}
A^\eta\defas\Big\{x\in A\colon v(x)\leq\frac{\eta}{1+\sqrt{\eta}}\Big\}\quad\text{and}\quad
B^\eta\defas\Big\{x\in A\colon \frac{\eta}{1+\sqrt{\eta}}< v(x) < \sqrt{\eta}\Big\}\, ,
\end{align*}
so that $v^\eta=v$ in $A\setminus(A^\eta\cup B^\eta)$; we have
\begin{equation}\label{eq:energy a-f}
\F_k^s(v^\eta,A)=\frac{1}{\e_k}\int_{A^\eta}g_k(x,0,0)\dx+\frac{1}{\e_k}\int_{B^\eta}g_k(x,v^\eta,\e_k\nabla v^\eta)\dx+\frac{1}{\e_k}\int_{A\setminus(A^\eta\cup B^\eta)}\hspace*{-2em} g_k(x,v,\e_k\nabla v)\dx\,.
\end{equation}
We start by estimating the first term on the right-hand side of~\eqref{eq:energy a-f}. Since $v\leq\frac{\eta}{1+\sqrt{\eta}}<\eta<1$ in $A^\eta$, using~\ref{hyp:cont-g} and~\ref{hyp:min-g} we get
\begin{equation*}
\begin{split}
g_k(x,0,0) &\leq g_k(x,v,0)+L_2(1+v^{p-1})v\\
&\leq g_k(x,v,\e_k\nabla v)+2L_2\eta
\leq g_k(x,v,\e_k\nabla v)+\frac{2L_2\eta}{(1-\eta)^p}(1-v)^p\,,
\end{split}
\end{equation*}
in $A^\eta$, which together with~\ref{hyp:lb-g} yields
\begin{align}\label{est:a-f1}
\frac{1}{\e_k}\int_{A^\eta}g_k(x,0,0)\dx\leq\Big(1+\frac{2 L_2\eta}{c_3(1-\eta)^p}\Big)\frac{1}{\e_k}\int_{A^\eta}g_k(x,v,\e_k\nabla v)\dx\, .
\end{align}
We now come to estimate the second term on the right-hand side of~\eqref{eq:energy a-f}. By definition, we have $v^\eta=(1+\sqrt{\eta})v-\eta <v<\sqrt{\eta}$ in $B^\eta$, thus from~\ref{hyp:cont-g} we deduce that 
\begin{equation}\label{est:B-eta1}
\begin{split}
g_k(x,v^\eta,\e_k\nabla v^\eta) \leq g_k(x,v,\e_k\nabla v) &+ L_2\big(1+2\eta^{\frac{p-1}{2}}\big)(\eta-\sqrt{\eta}v)\\
&+ L_2\big(1+\e_k^{p-1}|\nabla v|^{p-1}+\e_k^{p-1}|\nabla v^\eta|^{p-1}\big)\e_k|\nabla v-\nabla v^\eta|\,,
\end{split}
\end{equation}
in $B^\eta$. 
Furthermore, since $v<\sqrt{\eta}<1$ in $B^\eta$, it also holds 
\begin{equation}\label{est:B-eta2}
\big(1+2\eta^{\frac{p-1}{2}}\big)(\eta-\sqrt{\eta}v)\leq 3\eta\leq\frac{3\eta}{(1-\sqrt{\eta})^p}(1-v)^p\, ,
\end{equation}
and similarly
\begin{equation}\label{est:B-eta3}
|\nabla v-\nabla v^\eta|=\sqrt{\eta}|\nabla v|\leq \frac{\sqrt{\eta}}{(1-\sqrt{\eta})^{p-1}}(1-v)^{p-1}|\nabla v|\leq\frac{\sqrt{\eta}}{(1-\sqrt{\eta})^{p-1}}\Big(\frac{(1-v)^p}{\e_k}+\e_k^{p-1}|\nabla v|^p\Big)\,,
\end{equation}
where the last estimate follows by Young's inequality. Eventually, we also have
\begin{equation}\label{est:B-eta4}
\e_k^{p}\big(|\nabla v|^{p-1}+|\nabla v^\eta|^{p-1}\big)|\nabla v-\nabla v^\eta|=\e_k^p\big(1+(1+\sqrt{\eta})^{p-1}\big)\sqrt{\eta}|\nabla v|^p\leq\e_k^p\big(1+2^{p-1}\big)\sqrt{\eta}|\nabla v|^p\,.
\end{equation}
Gathering~\eqref{est:B-eta1}--\eqref{est:B-eta4}, dividing by $\e_k$ and using~\ref{hyp:lb-g} give
\EEE
\begin{equation}\label{est:a-f3}
\begin{split}
\frac{1}{\e_k}\int_{B^\eta}g_k(x,v^\eta,\e_k\nabla v^\eta)\dx 
&\leq\frac{1}{\e_k}\int_{B^\eta}g_k(x,v,\e_k\nabla v)\dx+c_\eta\frac{1}{\e_k}\int_{B^\eta}g_k(x,v,\e_k\nabla v)\dx\, ,
\end{split}
\end{equation}
where $c_\eta\defas\frac{L_2}{c_3}\max\{\frac{3\eta}{(1-\sqrt{\eta})^p}+(1+2^{p-1})\sqrt{\eta}\, ,\frac{\sqrt{\eta}}{(1-\sqrt{\eta})^{p-1}}\}$.
Then, since $(1-\eta)^p>(1-\sqrt{\eta})^p$ for $\eta\in(0,1)$, gathering~\eqref{eq:energy a-f}, \eqref{est:a-f1}, and~\eqref{est:a-f3} we obtain
\begin{equation*}
\F_k^s(v^\eta,A)\leq (1+c_\eta)\F_k^s(v,A)\, .
\end{equation*}
Since $c_\eta\to 0$ as $\eta\to 0$ this concludes the proof.
\end{proof}
With the help of Lemma \ref{lem:approx-min} and Lemma \ref{lem:auxiliary-function} we are now in a position to identify the surface integrand $\hat g$.
\begin{proposition}\label{p:surface-term}
Let  $(f_k) \subset \mathcal F$ and $(g_k) \subset \mathcal G$; \EEE let $(k_j)$ and $\hat g$ be as in Theorem~\ref{thm:int-rep}. Then, it holds 
\begin{align*}
\hat g(x,\zeta,\nu)=g'(x,\nu) = g''(x,\nu)\,,
\end{align*}
for every $x\in\R^n$, $\zeta\in\R^m_0$, and $\nu\in\Sph^{n-1}$, where $g'$ and $g''$ are, respectively, as in \eqref{g'} and \eqref{g''}, with $k$ replaced by $k_j$.
\end{proposition}
\begin{proof}

For notational simplicity, in what follows we still denote with $k$ the index of the sequence provided by Theorem~\ref{thm:int-rep}. 
\EEE

\medskip

By definition we have $g'\leq g''$; hence to prove the claim it suffices to show that 
\begin{equation}\label{eq:claim-g'-g''}
g''(x,\nu)\leq \hat{g}(x,\zeta,\nu)\leq g'(x,\nu), 
\end{equation} 
for every $x\in \R^n$, $\zeta\in \R^m_0$, and $\nu\in \Sph^{n-1}$.
We divide the proof of \eqref{eq:claim-g'-g''} into two steps.

\medskip
\emph{Step 1:} In this step we show that $\hat{g}(x,\zeta,\nu)\ge g''(x,\nu)$, for every $x\in\R^n$, $\zeta\in\R^m_0$, and $\nu\in\Sph^{n-1}$.

In view of Lemma~\ref{lem:approx-min} we have
\begin{equation*}
\hat{g}(x,\zeta,\nu)=\limsup_{\rho\to 0}\frac{1}{\rho^{n-1}}\limsup_{k\to+\infty}\m_k(\bar u_{x,\zeta,\e_k}^\nu,Q_\rho^\nu(x))\,.
\end{equation*}
Thanks to Remark \ref{rem:smooth-pair} the minimisation in the definition of $\m_k(\bar u_{x,\zeta,\e_k}^\nu,Q_\rho^\nu(x))$ can be carried over $C^1$-pairs $(u_k,v_k)$. 
Now let $\rho>0$ and $\eta \in (0,1)$ be fixed and for every $k$ such that $2\e_k<\rho$ let $(u_k,v_k) \subset C^1(Q_\rho^\nu(x);\R^m)\times C^1(Q_\rho^\nu(x))$ satisfy  
\begin{equation}\label{eq:bc0}
	(u_k,v_k)=(\bar u_{x,\zeta,\e_k}^\nu\, ,\bar v_{x,\e_k}^\nu)\;\text{ in }\; U_k\,,
\end{equation}
where $U_k$ is a neighbourhood of $\partial Q_\rho^\nu(x)$ and
\begin{equation}\label{bounded_energy}
\F_{k}(u_k,v_k,Q_\rho^\nu(x))\leq  \m_k(\bar u_{x,\zeta,\e_k}^\nu,Q_\rho^\nu(x)) + \eta \rho^{n-1},
\end{equation}
then, \eqref{uniformbound:me} readily gives
\begin{equation}\label{seq_bounded_energy}
\F_{k}(u_k,v_k,Q_\rho^\nu(x))\leq C \rho^{n-1}. 
\end{equation}
We now modify $v_k$ in order to obtain a new function $\tilde{v}_k$ for which there exists a corresponding $\tilde{u}_k$ such that the pair $(\tilde{u}_k,\tilde{v}_k)$ satisfies both
\begin{equation}\label{eq:bc}
(\tilde u_k,\tilde v_k)=( u_{x}^\nu\, ,1) \; \text{ in }\; \tilde U_k\cap\{|(y-x)\cdot\nu|>\e_k\} \,,
\end{equation}
where $\tilde U_k$ is a a neighbourhood of $\partial Q_\rho^\nu(x)$, and the constraint
\begin{equation}\label{eq:constraint}
\tilde{v}_k\,\nabla\tilde{u}_k=0\;\text{ a.e.\ in }\; Q_\rho^\nu(x)\,.
\end{equation}
In this way we have $\F_{k}(\tilde u_k,\tilde v_k, Q^\nu_{\rho}(x))=\F^s_{k}(\tilde v_k, Q^\nu_{\rho}(x))$ with  $\tilde{v}_k \in \Adm_{\e_k,\rho}(x,\nu)$.
\EEE 
The modification as above shall be performed without essentially increasing the energy $\F_{k}$.  

To this end, set $u_k:=(u_k^1,\ldots,u_k^m)$ and $\zeta:=(\zeta^1,\ldots,\zeta^m)$. Since $\zeta\in \R^m_0$ we can find $i\in \{1,\ldots,m\}$ so that $\zeta^i\neq 0$; without loss of generality we assume that $\zeta^i>0$.
We now consider the open set
\begin{equation*}
	S^\rho_k:=\left\{y\in Q^\nu_\rho(x)\colon 0<u_k^i(y)<\zeta^i\right\} \,.
\end{equation*}
Moreover, let $\sigma \in (0,1)$ be fixed and consider the following partition of $S^\rho_k$:
\[
S^\rho_k=\bigcup_{\ell =0}^{h_k^\sigma-1}S_{k,\ell}^\rho
\]
with
\begin{align*}
S_{k,\ell}^\rho &\defas\big\{y\in Q_\rho^\nu(x)\colon \ell\tfrac{\zeta^i}{h_k^\sigma}<u_k^i(y)\leq (\ell+1)\tfrac{\zeta^i}{h_k^\sigma}\big\}\,\quad \ell=0,\ldots,h_k^\sigma-2\,,\\ 
S_{k,h_{k}^\sigma-1}^\rho &\defas\big\{y\in Q_\rho^\nu(x)\colon (h_k^\sigma-1)\tfrac{\zeta^i}{h_k^\sigma}<u_k^i(y)<\zeta^i\big\}\,
\end{align*}
and $h_k^\sigma \in \N$ to be chosen later.

Then, there exists $\bar{\ell}=\bar{\ell}(k,\rho,\sigma)\in\{0,\ldots,h_k^\sigma-1\}$ such that
\begin{equation}\label{important}
	\int_{\tilde S^\rho_k}\left(\sigma\psi(v_k)|\nabla u_k^i|^p+(1-\sigma)\right)\,dy\le\frac{1}{h_k^\sigma}\int_{S^\rho_k}\left(\sigma\psi(v_k)|\nabla u_k^i|^p+(1-\sigma)\right)\dy\,,
\end{equation}
where $\tilde S^\rho_k:=S_{k,\bar{\ell}}^\rho$.

Therefore, gathering \ref{hyp:lb-f}, \eqref{seq_bounded_energy}, and \eqref{important} yields
\begin{equation}\label{c:stima}
\int_{\tilde S^\rho_k}\sigma\psi(v_k)|\nabla u_k^i|^p\,dy + (1-\sigma)\mathcal L^n (\tilde S^\rho_k) \leq \frac{\rho^{n-1}}{h_k^\sigma} C (\sigma +(1-\sigma)\rho)\,. \EEE
\end{equation}
In view of \eqref{eq:bc0} we have that 
\begin{equation}\label{inclusion:Skrho-tilde}
\tilde S^\rho_k\cap U_k\subset \{|(y-x)\cdot\nu|\le \e_k\}\cap U_k\,.
\end{equation}
In this way any modification to $v_k$ performed in the set $\tilde S^\rho_k$ will not affect the boundary conditions. In order to modify $v_k$ in $\tilde{S}_k^\rho$, we introduce an auxiliary function $\hat{v}_k$ which interpolates in a suitable way between the values $0$ and $1$ in $\tilde{S}_k^\rho$. To this end, we set $\gamma_k\defas(\bar{\ell}+\frac{1}{2})\frac{\zeta^i}{h_k^\sigma}$ and $\tau_k\defas\frac{\zeta^i}{4h_k^\sigma}$ and we define $\hat{v}_k\in W^{1,p}(Q_\rho^\nu(x))$ as follows:
\begin{equation*}
	 \hat v_k:=\begin{cases}
\min\biggl\{\dfrac{u^i_k-\gamma_k-\tau_k}{\tau_k}\,, 1\biggr\}& \text{in}\ \{ u^i_k \geq \gamma_k+\tau_k\}\,,\cr
	0&\text{in}\ \{\gamma_k-\tau_k < u^i_k < \gamma_k+\tau_k\}\,,\cr
\min	\biggl\{\dfrac{\gamma_k-\tau_k-u^i_k}{\tau_k}\,,1\biggr\}& \text{in}\ \{ u^i_k \leq \gamma_k-\tau_k\}\,,
	\end{cases}
\end{equation*}
\EEE
Eventually, we let $v_k^\eta\in W^{1,p}(Q_\rho^\nu(x))$ be the function defined in~\eqref{def:a-f} in Lemma \ref{lem:auxiliary-function}, with $v_k$ in place of $v$ and we define $\tilde v_k\in W^{1,p}(Q_\rho^\nu(x))$ as
\begin{equation*}
 \tilde v_k:=\min\{v_k^\eta\,, \hat v_k\}\,.
\end{equation*}
We notice that by definition
\[
\hat v_k\equiv 1 \quad \text{in $Q_\rho^\nu(x)\setminus\tilde S^\rho_k$}\,,
\] 
so that in particular $\tilde v_k\equiv v_k^\eta$ in $Q_\rho^\nu(x)\setminus\tilde S^\rho_k$; moreover, 
\[
\tilde v_k\equiv0 \quad \text{in $\{\gamma_k-\tau_k < u^i_k <\gamma_k+\tau_k\}$}\,.
\]
By the regularity of $u_k^i$ we can find $\xi_k>0$ (possibly small) such that 
\[
\{y \in Q^\nu_{\rho}(x) \colon\dist(y,\{u_k^i> \gamma_k\})<\xi_k \}\subset \{y \in Q^\nu_{\rho}(x) \colon u^i_k(y)>\gamma_k-\tau_k\}\,.
\] 
Then, we define $\tilde u_k\in W^{1,p}(Q_\rho^\nu(x);\R^m)$ as
\begin{align*}
\tilde{u}_k(y)\defas
\begin{cases}
\biggl(1-\dfrac{\dist(y,\{u^i_k> \gamma_k\})}{\xi_k}\biggr)e_1 &\text{if}\ \dist(y,\{u^i_k> \gamma_k\})<\xi_k\,,\\
0 &\text{otherwise}\,.
\end{cases}
\end{align*}
Thus the pair $(\tilde u_k,\tilde v_k)$ belongs to $W^{1,p}(Q_\rho^\nu(x);\R^m)\times W^{1,p}(Q_\rho^\nu(x))$ and by construction satisfies \eqref{eq:constraint}. Moreover,~\eqref{inclusion:Skrho-tilde} together with~\eqref{eq:bc0} ensures that~\eqref{eq:bc} is satisfied. 
Eventually, we have
\begin{align}\label{est:me-s}
\frac{1}{\rho^{n-1}}\m_{k,\rm N}^s(u_x^\nu,Q_\rho^\nu(x))\leq\frac{1}{\rho^{n-1}}\F_{k}^s(\tilde{v}_k,Q_\rho^\nu(x))\,.
\end{align}
To conclude the proof it only remains to show that, up to a small error, $\F_{k}^s(\tilde{v}_k,Q_\rho^\nu(x))$ is a lower bound for $\F_{k}(u_k,v_k,Q_\rho^\nu(x))$.
To this end we consider the following partition of $Q^\nu_\rho(x)$:
\begin{equation*}
	S^1_k:=\left\{y\in Q^\nu_\rho(x)\colon v^\eta_k(y)\le \hat v_k(y)
	\right\}\,,\quad 	S^2_k:=\left\{y\in Q^\nu_\rho(x)\colon v_k^\eta(y)> \hat v_k(y)
	\right\}\,;
\end{equation*}
then, appealing to \eqref{estimate:a-f} in Lemma \ref{lem:auxiliary-function}, we deduce
\begin{equation}\label{term1}
\begin{split}
\F^s_{k}(\tilde{v}_k,Q^\nu_\rho(x)) &=\frac{1}{\e_k}\int_{S^1_k}g_k(y,v_k^\eta,\e_k\nabla v_k^\eta)\dy+\frac{1}{\e_k}\int_{S_k^2}g_k(y,\hat v_k,\e_k\nabla \hat{v}_k)\dy\\
&\leq(1+c_\eta)\F_{k}^s(v_k,Q_\rho^\nu(x))+\frac{1}{\e_k}\int_{S_k^2}g_k(y,\hat v_k,\e_k\nabla \hat{v}_k)\dy\,,
\end{split}
\end{equation}
where $c_\eta \to 0$ as $\eta \to 0$. Hence, it remains to estimate the second term on the right-hand side of~\eqref{term1}. 

To this end, we start noticing that
\begin{equation}\label{important2}
	S^2_k\subset \left\{v_k>\frac{\eta}{1+\sqrt{\eta}}\right\}.
\end{equation}
Indeed, since $v_k^\eta >0$ a.e.\ in $S^2_k$, by definition of $v_k^\eta$ we readily get \eqref{important2}.  
\EEE

Therefore, by \ref{hyp:up-g}, using that $\hat v_k\equiv1$ in $Q_\rho^\nu(x)\setminus\tilde S^\rho_k$ we obtain
\begin{equation}\label{important3}
\begin{split}
	\frac{1}{\e_k}\int_{S_k^2}g_{k}(y,\hat v_k,\e_k\nabla \hat{v}_k)\dy&\le c_4\int_{S_k^2}\left( \frac{(1- \hat v_k)^p}{\e_k} +\e_k^{p-1}|\nabla \hat v_k|^p\right)\dy\\
	&=c_4\int_{S_k^2\cap\tilde S^\rho_k}\left( \frac{(1-\hat v_k)^p}{\e_k} +\e_k^{p-1}|\nabla \hat v_k|^p\right)\dy\\
	&\le c_4 \left(\frac{\L^n(\tilde S_k^\rho)}{\e_k}+\frac{\e_k^{p-1}}{\psi(\frac{\eta}{1+\sqrt{\eta}})} \left(\frac{4h_k^\sigma}{\zeta^i}\right)^{p}\int_{S_k^2\cap\tilde S^\rho_k}\psi(v_k)|\nabla u_k|^p\dy\right)\,,
\end{split}
\end{equation}
where the last inequality follows by \eqref{important2}, the definition of $\hat v_k$, and the monotonicity of $\psi$.

From \eqref{c:stima} we deduce both that
\begin{equation}\label{important4}
	\L^n(\tilde S_k^\rho) \le C\frac{\rho^{n-1}}{ h_k^\sigma}\left(\frac{\sigma }{1-\sigma}+\rho\right)
\end{equation}
and
\begin{equation}\label{important5}
	\int_{S_k^2\cap\tilde S_k}\psi(v_k)|\nabla u_k|^p\dx
	\leq C \frac{\rho^{n-1}}{h_k^\sigma}\Big(1+\frac{1-\sigma}{\sigma}\rho\Big)\,.
\end{equation}
Hence, gathering \eqref{important3}, \eqref{important4}, and \eqref{important5} we obtain
\begin{equation}\label{key-est}
		\frac{1}{\e_k}\int_{S_k^2}g_{k}(y,\hat v_k,\e_k\nabla \hat{v}_k)\dy\le C\,c_4\,\rho^{n-1}\bigg(\frac{1}{ \e_kh_k^\sigma}\Big(\frac{\sigma }{1-\sigma}+\rho\Big) +K_\eta(\e_kh_k^\sigma)^{p-1}\Big(1+\frac{1-\sigma}{\sigma}\rho\Big)\bigg)\,,
\end{equation}
with $K_\eta:= 4^p (\psi(\frac{\eta}{1+\sqrt{\eta}})(\zeta^i)^p)^{-1}$.
Now set $h_k^\sigma:=\lfloor \tilde h_k^\sigma\rfloor$ where
\begin{equation}\label{c:h-tilde}
\tilde h_k^\sigma:=\frac{1}{\e_k}\Big(\frac{\sigma}{1-\sigma}\Big)^{\frac{1}{p}}.
\end{equation}
Using that $h_k^\sigma\le \tilde h_k^\sigma\le h_k^\sigma+1$ we infer
\begin{equation*}
		\frac{1}{\e_k}\int_{S_k^2}g_{k}(y,\hat v_k,\e_k\nabla \hat{v}_k)\dy\le C\,c_4\rho^{n-1} \biggl(\frac{\tilde  h_k^\sigma}{\tilde h_k^\sigma-1}+K_\eta\biggr)\bigg(\Big(\frac{\sigma}{1-\sigma}\Big)^{\frac{p-1}{p}}+\rho\Big(\frac{\sigma}{1-\sigma}\Big)^{-\frac{1}{p}}\bigg)\,.
\end{equation*}
Plugging \eqref{key-est} into \eqref{term1} gives

\begin{equation}\label{c:almost-done}
\frac{1}{\rho^{n-1}}\F^s_{k}(\tilde{v}_k,Q^\nu_\rho(x))\leq(1+c_\eta)\frac{1}{\rho^{n-1}}\F_{k}(u_k,v_k,Q^\nu_\rho(x))+ \widehat K_\eta \bigg(\Big(\frac{\sigma}{1-\sigma}\Big)^{\frac{p-1}{p}}+\rho\Big(\frac{\sigma}{1-\sigma}\Big)^{-\frac{1}{p}}\bigg)\,,
\end{equation}
where $\widehat K_\eta:=C\,c_4(1+ K_\eta)$. We notice that $c_\eta \to 0$ as $\eta \to 0$, while $\widehat K_\eta \to +\infty$ as $\eta \to 0$.
Finally, by combining \eqref{bounded_energy}, \eqref{est:me-s}, and \eqref{c:almost-done} we get 
\begin{equation}\label{c:last}
\frac{1}{\rho^{n-1}}\m_{k,\rm N}^s(u_x^\nu,Q_\rho^\nu(x))\leq  (1+c_\eta)\bigg(\frac{1}{\rho^{n-1}}\m_{k}(\bar u_{x,\zeta,\e_k}^\nu,Q_\rho^\nu(x))+ \eta\bigg) + \widehat K_\eta \bigg(\Big(\frac{\sigma}{1-\sigma}\Big)^{\frac{p-1}{p}}+\rho\Big(\frac{\sigma}{1-\sigma}\Big)^{-\frac{1}{p}}\bigg)\,.
\end{equation}
Eventually, the claim follows by passing to the limit in \eqref{c:last} in the following order: first as $k\to+\infty$, then as $\rho\to0$, $\sigma \to 0$, and finally as $\eta \to0$. 

\medskip

\emph{Step 2:} In this step we show that 
\[\hat{g}(x,\zeta,\nu)\leq g'(x,\nu)\,,\]
for every $x\in\R^n$, $\zeta\in\R^m_0$ and $\nu\in\Sph^{n-1}$. Thanks to~\eqref{eq:g-hat}, in view of Lemma~\ref{lem:approx-min} and Proposition \ref{p:equivalence-N-D} \EEE it suffices to show that
\begin{align}\label{eq:surface-lb}
\m_{k}(\bar u_{x,\zeta,\e_k}^\nu,Q_\rho^\nu(x))\leq\m_{k}^s(\bar u_{x,\e_k}^\nu,Q_\rho^\nu(x))\,,
\end{align}
for every $\e>0$, where $\m_{k}^s(\bar u_{x,\e_k}^\nu,Q_\rho^\nu(x))$ is defined in \eqref{eq:ms-D}.
To prove~\eqref{eq:surface-lb} let  $v \in \Adm(\bar u^\nu_{x,\e_k},Q^\nu_\rho(x))$ with corresponding $u\in W^{1,p}(Q_\rho^\nu(x);\R^m)$ and notice that the pair $(\tilde{u},v)$ with $\tilde{u}\defas (u\cdot e_1)\zeta$ is an admissible competitor for $\m_{k}(\bar u_{x,\zeta,\e_k}^\nu,Q_\rho^\nu(x))$. Therefore, we obtain
\begin{equation*}
\m_{k}(\bar u_{x,\zeta,\e_k}^\nu,Q_\rho^\nu(x))\leq\F_{k}(\tilde{u},v,Q_\rho^\nu(x))=\F_{k}^s(v,Q_\rho^\nu(x))\,,
\end{equation*}
from which we deduce~\eqref{eq:surface-lb} by passing to the infimum in $v \in \Adm(\bar u^\nu_{x,\e_k},Q^\nu_\rho(x))$.
\end{proof}
\begin{remark}\label{rem:p-q} We observe that the second term in the right hand side of \eqref{c:last} is infinitesimal as $\rho, \sigma\to 0$ thanks to the fact that $p>1$. We notice, moreover, that in \eqref{c:last} the presence of the exponent $p$ in the reminder is a consequence of the $p$-growth of the volume integrand $f_k$. Indeed if $g_k$ satisfied conditions \ref{hyp:lb-g}-\ref{hyp:cont-g} with $p$ replaced by some exponent $q\in (1,p]$, then arguing as in the proof of Proposition \ref{p:surface-term}, in place of \eqref{key-est} we would get  
\begin{multline*}
		\frac{1}{\e_k}\int_{S_k^2}g_{k}(y,\hat v_k,\e_k\nabla \hat{v}_k)\dy
		\\
		\le C c_4\rho^{n-1}\bigg(\frac{1}{ \e_kh_k^\sigma}\Big(\frac{\sigma }{1-\sigma}+\rho\Big) +K_\eta(\e_kh_k^\sigma)^{q-1}\Big(1+\frac{1-\sigma}{\sigma}\rho\Big)^\frac{q}{p}\Big(\frac{\sigma }{1-\sigma}+\rho\Big)^\frac{p-q}{p}\bigg)\,,
\end{multline*}
with $K_\eta:= 4^q (\psi(\frac{\eta}{1+\sqrt{\eta}})(\zeta^i)^q)^{-1}$. Then, an easy computation shows that choosing $h_k^\sigma=\lfloor \tilde h_k^\sigma\rfloor$, with $\tilde h_k^\sigma$ given by \eqref{c:h-tilde}, exactly yields \eqref{c:last}.  
\end{remark}

We are now in a position to prove the main result of this paper, namely, Theorem \ref{thm:main-result}.

\begin{proof}[Proof of Theorem \ref{thm:main-result}] The proof follows by combining Proposition \ref{prop:f'-f''}, Proposition \ref{prop:g'-g''}, Theorem \ref{thm:int-rep}, Proposition \ref{p:volume-term}, and Proposition \ref{p:surface-term}.
\end{proof}

%
%
\section{Stochastic homogenisation}\label{sect:stochastic-homogenisation}

\noindent In this section we study the $\Gamma$-convergence of the functionals $\F_k$ when $f_k$ and $g_k$ are random integrands of type 
\[
f_k(\omega,x,\xi)=f\Big(\omega,\frac{x}{\e_k},\xi\Big), \quad g_k(\om, x,v,w)=g\Big(\omega,\frac{x}{\e_k},v,w\Big),
\]
where $\om$ belongs to the sample space $\Omega$ of a probability space $(\Omega,\T,P)$. 

\medskip

Before stating the main result of this section, we need to recall some useful definitions.

\begin{definition}[Group of $P$-preserving transformations] Let $d\in \N$, $d\geq 2$. A group of $P$-preserving transformations on $(\Omega,\T,P)$  is a family $(\tau_z)_{z\in\Z^d}$ of mappings $\tau_z\colon\Omega\to\Omega$ satisfying the following properties:
	\begin{enumerate}[label=(\arabic*)]
		\item (measurability) $\tau_z$ is $\T$-measurable for every $z\in\Z^d$;
		\item (invariance) $P(\tau_z(E))=P(E)$, for every $E\in\T$ and every $z\in\Z^d$;
		\item (group property) $\tau_0={\rm id}_\Omega$ and $\tau_{z+z'}=\tau_z\circ\tau_{z'}$ for every $z,z'\in\Z^d$.
	\end{enumerate}
	If, in addition, every $(\tau_z)_{z\in\Z^d}$-invariant set (\textit{i.e.}, every $E\in\T$ with  
	$\tau_z(E)=E$ for every $z\in\Z^d$) has probability 0 or 1, then $(\tau_z)_{z\in\Z^d}$ is called ergodic.
\end{definition}
Let $a\defas(a_1,\dots,a_d),\,b\defas(b_1,\dots,b_d)\in\Z^d$ with $a_i<b_i$ for all $i\in\{1,\ldots,d\}$; we define the $d$-dimensional interval 
\begin{equation*}
[a,b):=\{x\in\Z^d\colon a_i\le x_i<b_i\,\,\text{for}\,\,i=1,\ldots,d\}
\end{equation*}
and we set
\begin{equation*}
\I_d:=\{[a,b)\colon a,b\in\Z^d\,,\, a_i<b_i\,\,\text{for}\,\,i=1,\ldots,d\}\,.
\end{equation*}
\begin{definition}[Subadditive process]\label{sub_proc}
	A discrete subadditive process with respect to a group $(\tau_z)_{z\in\Z^d}$ of $P$-preserving transformations on $(\Omega,\T,P)$ is a function $\mu\colon\Omega \x \I_d\to\R$ satisfying the following properties:
	\begin{enumerate}[label=(\arabic*)]
		\item\label{subad:meas} (measurability) for every $A\in\I_d$ the function $\omega\mapsto\mu(\omega,A)$ is $\T$-measurable;
		\item\label{subad:cov} (covariance) for every $\omega\in\Omega$, $A\in\I_d$, and $z\in\Z^d$ we have $\mu(\omega,A+z)=\mu(\tau_z(\omega),A)$;
		\item\label{subad:sub} (subadditivity) for every $A\in\I_d$ and for every finite family $(A_i)_{i\in I}\subset\I_d$ of pairwise disjoint sets such that $A=\cup_{i\in I}A_i$, we have
		\begin{equation*}
		\mu(\omega,A)\le\sum_{i\in I}\mu(\omega,A_i)\quad\text{for every}\,\,\omega\in\Omega\,;
		\end{equation*}
		\item\label{subad:bound} (boundedness) there exists $c>0$ such that $0\le\mu(\omega,A)\le c\L^d(A)$ for every $\omega\in\Omega$ and $A\in\I_d$.
	\end{enumerate}
\end{definition}

\begin{definition}[Stationarity] 
Let $(\tau_z)_{z\in\Z^n}$ be a group of $P$-preserving transformations on $(\Omega,\T,P)$. 
We say that $f\colon\Omega\x\R^n\x\R^{m\x n}\to[0,+\infty)$ is stationary with respect to $(\tau_z)_{z\in\Z^n}$ if
\begin{equation*}
f(\omega,x+z,\xi)=f(\tau_z(\omega),x,\xi)
	\end{equation*}
	for every $\omega\in\Omega$, $x\in\R^n$, $z\in\Z^n$ and $\xi\in\R^{m\x n}$.
	
	Analogously, we say that $g\colon\Omega\x\R^n\x\R \x\R^n\to[0,+\infty)$ is stationary with respect to $(\tau_z)_{z\in\Z^n}$ if
	\begin{equation*}
g(\omega,x+z,v,w)=g(\tau_z(\omega),x,v,w)
	\end{equation*}
	for every $\omega\in\Omega$, $x\in\R^n$, $z\in\Z^n$, $v\in[0,1]$ and $w\in\R^n$. 	
\end{definition}

In all that follows we consider random integrands $f\colon\Omega\x\R^n\x\R^{m\x n}\to[0,+\infty)$  satisfying  

\smallskip

\begin{enumerate}[label= ($F\arabic*$)]
	\item\label{meas_f} $f$ is $(\T\otimes\mathcal{B}^n\otimes\mathcal{B}^{m\x n})$-measurable;
	\item\label{finF} $f(\omega,\,\cdot\,,\,\cdot\,)\in\mathcal F$ for every $\omega\in\Omega$;
\end{enumerate}

\smallskip

and random integrands $g\colon\Omega\x\R^n\x\R\x\R^n\to[0,+\infty)$ satisfying 

\smallskip

\begin{enumerate}[label= ($G\arabic*$)]
	\item\label{meas_g} $g$ is $(\T\otimes\mathcal{B}^n\otimes\mathcal{B}\otimes\mathcal{B}^n)$-measurable;
	\item\label{ginG} $g(\omega,\,\cdot\,,\,\cdot\,,\,\cdot\,)\in\mathcal G$ for every $\omega\in\Omega$.
\end{enumerate}

\medskip

\noindent Let $f$ and $g$ be random integrands satisfying \ref{meas_f}-\ref{finF} and \ref{meas_g}-\ref{ginG}, respectively. We consider the sequence of random elliptic functionals $\F_k(\om) \colon L^0(\R^n;\R^m) \times L^0(\R^n) \times \A \longrightarrow [0,+\infty]$ given by
\begin{equation}\label{F_e_omega}
\F_k(\om)(u,v, A)\defas \int_A \psi(v)\,f\left(\omega,\frac{x}{\e_k},\nabla u\right)\dx+\frac{1}{\e_k}\int_A g\left(\omega,\frac{x}{\e_k},v,\e_k\nabla v\right)\dx\,,
\end{equation}
if $(u,v)\in W^{1,p}(A;\R^m)\x W^{1,p}(A)$, $0\leq v\leq 1$ and extended to $+\infty$ otherwise. We also let $\F^b(\omega)$ be as in~\eqref{eq:Fb} and $\F^s(\omega)$ as in~\eqref{eq:Fs}, with $f(\,\cdot\,,\,\cdot\,)$ and $g(\,\cdot\,,\,\,\cdot\,,\,\cdot\,)$ replaced by $f(\omega,\,\cdot\,,\,\cdot\,)$ and $g(\omega,\,\cdot\,,\,\cdot\,,\,\cdot\,)$, respectively.  Moreover, for $\omega\in\Omega$ and $A\in\A$ we set
\begin{equation}\label{mb-omega}
\m^b_\om(u_\xi,A)\defas\inf\{\F^b(\om)(u,A)\colon u\in W^{1,p}(A;\R^m)\, ,\ u=u_\xi\ \text{near}\ \partial A\}
\end{equation}
and 
\begin{equation}\label{ms-omega}
\m^s_{\om}(\bar u_{x}^\nu,A)\defas\inf\{\F^s(\om)(v,A)\colon v\in\Adm(\bar{u}_x^\nu,A)\}\,,
\end{equation}
where $\Adm(\bar{u}_x^\nu,A)$ is as in~\eqref{c:adm-e-bar} with $\bar{u}_x^\nu$ in place of $\bar{u}_{x,\e_k}^\nu$\ie with $\e_k=1$.

Eventually, we extend the definition of $\m^b_{\om}(u_\xi,\cdot)$, $\m^s_{\om}(\bar{u}_x^\nu,\cdot)$, and $\Adm(\bar{u}_0^\nu,\cdot)$ to any $A\subset\R^n$ with $\rm int\, A\in\A$ by setting $\m^b_{\om}(u_\xi,A)\defas\m^b_{\om}(u_\xi,\rm int\, A)$, $\m^s_{\om}(\bar{u}_x^\nu,A)\defas\m^s_{\om}(\bar{u}_x^\nu,\rm int\, A)$, and $\Adm(\bar{u}_0^\nu,A)\defas \Adm(\bar{u}_0^\nu,\rm int\, A)$. 

\medskip

We are now ready to state the main result of this section.

\begin{theorem}[Stochastic homogenisation]\label{thm:stoch_hom_2}
Let $f$ and $g$ be random integrands satisfying \ref{meas_f}-\ref{finF}  and \ref{meas_g}-\ref{ginG}, respectively. Assume moreover that $f$ and $g$ are stationary with respect to a group $(\tau_z)_{z\in\Z^n}$ of $P$-preserving transformations on $(\Omega,\T,P)$. For every $\omega\in\Omega$ let $\F_k(\omega)$ be as in \eqref{F_e_omega} and $\m^b_\omega$, $\m^s_\omega$ be as in~\eqref{mb-omega} and~\eqref{ms-omega}, respectively. Then there exists $\Omega'\in\T$, with $P(\Omega')=1$ such that for every $\omega\in\Omega'$, $x\in\R^n$, $\xi\in\R^{m\x n}$, $\nu\in \Sph^{n-1}$ the limits
	\begin{align}\label{eq:f-stoch-hom}
	\lim_{r\to +\infty} \frac{\m^b_\omega(u_\xi,Q_r(rx))}{r^n}=
		\lim_{r\to +\infty} \frac{\m^b_\omega(u_\xi,Q_r(0))}{r^n}=:f_\hom(\om,\xi)\,,\\
	\lim_{r\to +\infty} \frac{\m^s_\omega(\bar u^\nu_{rx},Q^\nu_r(rx))}{r^{n-1}}=
	\lim_{r\to +\infty} \frac{\m^s_\omega(\bar u^\nu_{0},Q^\nu_r(0))}{r^{n-1}}=:g_\hom(\om,\nu)\label{eq:g-stoch-hom}
	\end{align}
exist and are independent of $x$. The function $f_\hom \colon \Omega \times \R^{m\times n} \to [0,+\infty)$ is $(\T\otimes \mathcal B^{m\x n})$-measurable and $g_\hom \colon \Omega \times \Sph^{n-1} \to [0,+\infty)$ is $(\T\otimes \mathcal B(\Sph^{n-1}))$-measurable.

Moreover, for every $\omega\in\Omega'$ and for every $A\in \A$ the functionals $\F_k(\omega)(\cdot,\cdot,A)$ $\Gamma$-converge in $L^0(\R^n;\R^m)\times L^0(\R^n)$ to the functional $\F_{\hom}(\omega)(\cdot,\cdot,A)$ with $\F_{\hom}(\omega)\colon L^0(\R^n;\R^m)\times L^0(\R^n)\times\A\longrightarrow[0,+\infty]$ given by
\begin{equation*}
\F_{\rm hom}(\omega)(u,v,A)\defas
\begin{cases}
\displaystyle\int_A f_{\rm hom}(\omega,\nabla u)\dx+\int_{S_u \cap A} g_{\rm hom}(\omega,\nu_u)\dHn &\text{if}\ u \in GSBV^p(A;\R^m)\, ,\\
& v= 1\, \text{a.e.\ in } A\, ,\\[4pt]
+\infty &\text{otherwise}\,.
\end{cases}
\end{equation*}
If, in addition, $(\tau_z)_{z\in\Z^n}$ is ergodic, then $f_{\hom}$ and $g_{\hom}$ are independent of $\omega$ and 
		\begin{align*}
	f_{\rm hom}(\xi) &=
	\lim_{r\to +\infty} \frac{1}{r^n}\int_\Omega \m^b_\omega(u_\xi,Q_r(0))\dP(\omega)\,,\\
g_{\rm hom}(\nu) &=
	\lim_{r\to +\infty} \frac{1}{r^{n-1}}\int_\Omega \m^s_\omega(\bar u^\nu_{0},Q^\nu_r(0))\dP(\omega)\,,
	\end{align*}
	thus, $\F_\hom$ is deterministic.
\end{theorem}

The almost sure $\Gamma$-convergence result in Theorem~\ref{thm:stoch_hom_2} is an immediate consequence of Theorem~\ref{thm:hom} once we show the existence of a $\T$-measurable set $\Omega' \subset \Omega$, with $P(\Omega')=1$, such that for every $\om \in \Omega'$ the limits in~\eqref{eq:f-stoch-hom}-\eqref{eq:g-stoch-hom} exist and are independent of $x$. Therefore, the rest of this section is devoted to prove the existence of such a set. 

\subsection{Homogenisation formulas}\label{sect:cell-formulas}
In this subsection we prove that conditions \ref{meas_f}, \ref{finF}, \ref{meas_g}, and \ref{ginG} together with the stationarity of the random integrands $f$ and $g$ ensure that the assumptions of Theorem~\ref{thm:hom} are satisfied almost surely. 

\medskip

The proof of the following result is based on the pointwise Subbaditive Ergodic Theorem~\cite[Theorem 2.4]{AK81} applied 
to the function $(\om,I) \mapsto \m^b_\om(u_\xi,I)$, which defines a subadditive process on $\Omega \times \mathcal I_n$, as shown in \cite[Proposition 3.2]{MM}.

\begin{proposition}[Homogenised volume integrand]\label{prop:cell-volume}
Let $f$ satisfy \ref{meas_f}-\ref{finF} and assume that it is stationary with respect to a group $(\tau_z)_{z\in\Z^n}$ of $P$-preserving transformations on $(\Omega,\T,P)$. For $\omega\in\Omega$ let $\m^b_\omega$ be as in~\eqref{mb-omega}.
Then there exists $\Omega'\in\T$, with $P(\Omega')=1$ and a $(\T\otimes \mathcal B^{m\x n})$-measurable function
$f_{\hom}\colon\Omega\x\R^{m\x n}\to[0,+\infty)$ such that 
\begin{equation*}
\lim_{r\to +\infty} \frac{\m^b_\omega(u_\xi,Q_r(rx))}{r^n}=
\lim_{r\to +\infty} \frac{\m^b_\omega(u_\xi,Q_r(0))}{r^n}=
f_{\rm hom}(\omega,\xi).
\end{equation*}
for every $\omega\in \Omega'$, $x\in\R^n$ and every $\xi\in\R^{m\x n}$. If, in addition, $(\tau_z)_{z\in\Z^n}$ is ergodic, then $f_{\hom}$ is independent of $\omega$ and given by
\begin{equation*}
f_{\rm hom}(\xi)=
\lim_{r\to +\infty} \frac{1}{r^n}\int_\Omega \m^b_\omega(u_\xi,Q_r(0))\dP(\omega).
\end{equation*}
\end{proposition}
\begin{proof}
The proof follows by \cite[Proposition 3.2]{MM} arguing as in \cite[Theorem 1]{DMM} (see also \cite[Corollary 3.3]{MM}).
\end{proof}

We now deal with the existence of the homogenised surface integrand $g_\hom$. Unlike the case of $f_\hom$, the   
existence and $x$-homogeneity of the limit in~\eqref{eq:g-stoch-hom} cannot be deduced by a direct application of the Subadditive Ergodic Theorem ~\cite[Theorem 2.4]{AK81}. In fact, due to the $x$-dependent boundary datum appearing in $\m^s_\om(\bar u_{rx}^\nu,Q_r^\nu(rx))$ (cf.~\eqref{ms-omega}), the proof of the $x$-homogeneity of $g_\hom$ is rather delicate and follows by \textit{ad hoc} arguments, which are typical of surface functionals \cite{ACR11, CDMSZ19a}. 

The proof of the existence of $g_\hom$ will be carried out in several step. In a first step we prove that when $x=0$ the minimisation problem \eqref{ms-omega} defines a subadditive process on $\O \times \mathcal I_{n-1}$.
To do so we follow the same procedure as in~\cite[Section 5]{CDMSZ19a} (see also \cite{ACR11, BCR18}). Namely, given $\nu\in \Sph^{n-1}$ we let $R_\nu$ be an orthogonal matrix as in~\ref{Rn}. Then $\{R_\nu e_i\colon i=1\,,\ldots\,,n-1\}$ is an orthonormal basis for $\Pi^\nu$, further $R_\nu\in\Q^{n\x n}$, if $\nu\in\Sph^{n-1}\cap\Q^n$. Now let $M_\nu>2$ be an integer such that $M_\nu R_\nu\in\Z^{n\x n}$; therefore $M_\nu R_\nu (z',0)\in\Pi^\nu\cap\Z^n$ for every $z'\in\Z^{n-1}$. 

Let $I\in\mathcal{I}_{n-1}$\ie $I=[a,b)$ with $a,b\in\Z^{n-1}$. Starting from $I$ we define the $n$-dimensional interval $I_\nu$ as
\begin{equation}\label{def:interval-n-dim}
I_\nu\defas M_\nu R_\nu\big(I\x[-c,c)\big)\quad\text{where}\quad c\defas\frac{1}{2}\max_{i=1,\ldots,n-1}(b_i-a_i)\,.
\end{equation}
Correspondingly, for fixed $\nu \in \Sph^{n-1}\cap \Q^n$ we define the function $\mu_\nu \colon \O \times \mathcal I_{n-1} \mapsto \R$ as  
\begin{equation}\label{def:subad-process-g}
\mu_\nu(\omega,I)\defas\frac{1}{M_\nu^{n-1}}\m^s_\om(\bar{u}_{0}^\nu, I_\nu)\,,
\end{equation} 
where $\m^s_\om (\bar{u}_{0}^\nu, I_\nu)$ is as in \eqref{ms-omega} with $x=0$ and $A=I_\nu$.

The following result asserts that $\mu_\nu$ defines a subadditive process on $\O \times \mathcal I_{n-1}$.

\begin{proposition}\label{prop:subadditive}
Let $g$ satisfy \ref{meas_g}-\ref{ginG} and assume that it is stationary with respect to a group $(\tau_z)_{z\in\Z^n}$ of $P$-preserving transformations on $(\Omega,\T,P)$. Let $\nu\in\Sph^{n-1}\cap \Q^n$ and let $\mu_\nu\colon \O \times \mathcal I_{n-1} \mapsto \R$ be as in \eqref{def:subad-process-g}.
Then there exists a group of $P$-preserving transformations $(\tau_{z'}^\nu)_{z'\in\Z^{n-1}}$ on $(\Omega,\T,P)$ such that $\mu_\nu$ is a subadditive process on $(\Omega,\T,P)$ with respect to $(\tau_{z'}^\nu)_{z'\in\Z^{n-1}}$. Moreover, it holds
\begin{equation}\label{c:bd-mu}
0\le\mu_\nu(\omega, I)\le c_4 C_\vv \L^{n-1}(I),
\end{equation}
for $P$-a.e.\ $\om \in \Omega$ and for every $I\in \mathcal I_{n-1}$.  
\end{proposition}
\EEE
\begin{proof}

Let $\nu\in\S^{n-1}\cap\Q^n$ be fixed;  below we show that $\mu_\nu$ satisfies conditions~\ref{subad:meas}--\ref{subad:bound} in Definition \ref{sub_proc}, for some group of $P$-preserving transformations $(\tau_{z'}^\nu)_{z'\in\Z^{n-1}}$.

We divide the proof into four step, each of them corresponding to one of the four conditions in Definition \ref{sub_proc}.  

\medskip

\textit{Step 1: measurability.}
Let $I\in \mathcal I_{n-1}$ and let $I_\nu \subset \R^n$ be as in \eqref{def:interval-n-dim}. Let $v\in W^{1,p}(I_\nu)$ be fixed. In view of \ref{meas_g} the Fubini Theorem ensures that the map $\omega\mapsto\F^s(\omega)(v,I_\nu)$ is $\T$-measurable. 
On the other hand the space $W^{1,p}(I_\nu)$ is separable, hence the set of functions  
\begin{multline*}
\Adm(\bar{u}_0^\nu,I_\nu)= \{v\in W^{1,p}(I_\nu)\, ,\; 0\leq v \leq 1\,, \ v= \bar v^\nu_0\ \text{ near }\ \partial I_\nu \ \text{and}\ \exists\, u\in W^{1,p}(I_\nu;\R^m),\\
u=\bar u^\nu_0\ \text{ near }\, \partial I_\nu,\; \text{such that}\;  v\,\nabla u=0\ \text{a.e.\ in}\  I_\nu\}
\end{multline*}
defines a separable metric space, when endowed with the distance induced by the $W^{1,p}(I_\nu)$-norm. Then, by the continuity of $\F^s(\omega)(\cdot, I_\nu)$ with respect to the strong $W^{1,p}(I_\nu)$-topology, the infimum in the definition of $\m^s_\om(\bar u^\nu_0,I_\nu)$ can be equivalently expressed as an infimum on a countable subset of $\Adm(\bar{u}_0^\nu,I_\nu)$, thus ensuring the $\T$-measurability of $\om \mapsto \m^s_\om(\bar u^\nu_0,I_\nu)$ and consequently that of $\om \mapsto \mu_\nu(\omega,I)$, as desired.  

\medskip

\textit{Step 2: covariance.} Let $z'\in\Z^{n-1}$ be fixed. Let $I\in\I_{n-1}$; by \eqref{def:interval-n-dim} we have
\begin{equation*}
(I+z')_\nu=I_\nu +M_\nu R_\nu(z',0)=I_\nu+z'_\nu\,,
\end{equation*}
where $z'_{\nu}:=M_\nu R_\nu(z',0) \in \Z^n\cap\Pi^\nu$.
Then, by \eqref{def:subad-process-g} we get
\begin{equation}\label{cov_1}
	\mu_\nu(\omega,I+z')=\frac{1}{M_\nu^{n-1}}\m^s_\omega(\bar u_0^\nu,I_\nu+z'_\nu)\,.
\end{equation}
Now let $v\in \Adm(\bar{u}_0^\nu,I_\nu+z_\nu')$ with its corresponding $u\in W^{1,p}(I'_\nu+z'_\nu; \R^m)$. Setting $\tilde v(x)\defas v(x+z'_\nu)$ for $x\in I_\nu$, a change of variables together with the stationarity of $g$ yield
\begin{eqnarray*}
\F^s(\omega)(v,{\rm int}\, (I_\nu+z'_\nu)) &=&\int_{I_\nu+z'_\nu}g(\omega,x,v,\nabla v)\dx
=\int_{I_\nu}g(\omega,x+z'_\nu,\tilde v,\nabla \tilde v)\dx\\&=&\int_{I_\nu}g(\tau_{z'_\nu}(\omega),x,\tilde v,\nabla\tilde v)\dx=\F^s(\tau_{z'_\nu}(\omega))(\tilde v, {\rm int}\, I_\nu)\,.
\end{eqnarray*}
 Set $(\tau^\nu_{z'})_{z'\in\Z^{n-1}}:=(\tau_{z'_\nu})_{z'\in\Z^{n-1}}$; we notice that $(\tau^\nu_{z'})_{z'\in\Z^{n-1}}$ is well defined since $z'_\nu\in\Z^n$ and it defines a group of $P$-preserving transformations on $(\Omega, P,\T)$. Then, the equality above can be rewritten as
\begin{equation}\label{cov_2}
	\F^s(\omega)(v, {\rm int}\,(I_\nu+z'_\nu))=\F^s(\tau_{z'}^\nu(\omega))(\tilde v,{\rm int}\, I_\nu)\,.
\end{equation}
Moreover, if we set $\tilde{u}(x)\defas u(x+z'_\nu)$ for $x\in I_\nu$, then $\tilde{v}\, \nabla \tilde{u}=0$ a.e.\ in $I_\nu$, while since $z_\nu'\in\Pi^\nu$, we also have 
\[
\big(\tilde{u},\tilde{v}\big)=\big(\bar{u}_0^\nu(\cdot+z'_\nu),\bar{v}_0^\nu(\cdot+z'_\nu)\big)=\big(\bar{u}_0^\nu,\bar{v}_0^\nu\big)
\]
near $\partial I_\nu$\ie $\tilde{v}\in\Adm(\bar{u}_0^\nu, I_\nu)$.
Thus, gathering \eqref{cov_1} and \eqref{cov_2}, by the arbitrariness of $v$ we infer 
\begin{equation*}
\mu_\nu(\omega,I+z')=
\mu_\nu(\tau^\nu_{z'}(\omega),I)\,,
\end{equation*}
and hence the covariance of $\mu_\nu$ with respect to $(\tau^\nu_{z'})_{z'\in\Z^{n-1}}$.

\medskip

\textit{Step 3: subadditivity.} Let  $\omega\in\Omega$, $I\in\I_{n-1}$, and let $\{I_1,\dots,I_N\}\subset\I_{n-1}$ be a finite family of pairwise disjoint sets such that $I=\bigcup_i I_i$.
Let $\eta>0$ be fixed; for every $i=1,\dots,N$ let $v_i\in W^{1,p}((I_i)_\nu)$ be admissible for $\m^s_\omega(\bar u_0^\nu,(I_i)_\nu)$ and such that 
\begin{equation}\label{sub_1}
	\F^s(\omega)(v_i,{\rm int}\,(I_i)_\nu)\le\m^s_\omega(\bar u_0^\nu,(I_i)_\nu)+\eta\,. 
\end{equation}
Therefore for every $i=1,\dots,N$ there exists a corresponding $u_i\in W^{1,p}((I_i)_\nu;\R^m)$ such that 
\begin{equation*}
v_i \nabla u_i =0\;\text{ a.e.\ in }\;(I_i)_\nu\;\text{ with }\; (u_i,v_i)=(\bar u^\nu_0,\bar v^\nu_0)\quad\text{near}\,\,\partial(I_i)_\nu\,.
\end{equation*}
We define
\begin{equation*}
v:=\begin{cases}
v_i&\text{in }\;(I_i)_\nu\,,\ i=1\,,\ldots,N\,,\\
\bar v_0^\nu&\text{in }\; I_\nu\setminus\bigcup_i (I_i)_\nu\,,
\end{cases}
\quad
u:=
\begin{cases}
u_i&\text{in }\;(I_i)_\nu\,,\ i=1\,,\ldots,N\,,\\
\bar u_0^\nu&\text{in }\; I_\nu\setminus\bigcup_i(I_i)_\nu\,,
\end{cases}
\end{equation*}
so that $(u,v)\in W^{1,p}(I_\nu)\times W^{1,p}(I_\nu;\R^m)$ and 
\begin{equation*}
v \nabla u=0\;\text{ a.e.\ in }\; I_\nu\;\text{ with }\; (u,v)=(\bar u^\nu_0,\bar v^\nu_0)\;\text{ near }\;\partial I_\nu\,.
\end{equation*}
Therefore $v\in\Adm(\bar{u}_0^\nu,I_\nu)$.
Furthermore, we have
\begin{equation}\label{est:subad}
\F^s(\omega)(v,{\rm int}\,I_\nu)=\sum_{i=1}^N\F^s(\omega)(v_i,{\rm int}\,(I_i)_\nu)+\F^s(\omega)\big(\bar v^\nu_0,{\rm int}\,(I_\nu\setminus\bigcup_{i=1}^N(I_i)_\nu)\big)\,.
\end{equation}
Since $M_\nu>2$ and $c\geq\frac{1}{2}$ in~\eqref{def:interval-n-dim}, it follows that
$\{y\in I_\nu\colon|y\cdot\nu|\le 1\}\subset\bigcup_{i}(I_i)_\nu$.
Thus, $\bar v_0^\nu\equiv1$ in $I_\nu\setminus\bigcup_i(I_i)_\nu$ which, thanks to~\eqref{g-value-at-0}, gives
\begin{equation*}
\F^s(\omega)\big(\bar v^\nu_0,{\rm int}\,(I_\nu\setminus\bigcup_{i=1}^N(I_i)_\nu)\big)=0\,.
\end{equation*}
Eventually, gathering~\eqref{sub_1}--\eqref{est:subad} we obtain
\begin{equation*}
\m^s_\omega(\bar u_0^\nu,I_\nu)\le \F^s(\omega)(v,{\rm int}\,I_\nu)=\sum_{i=1}^N\F^s(\omega)(v_i,{\rm int}\,(I_i)_\nu)\le \sum_{i=1}^N\m^s_\omega(\bar u_0^\nu,(I_i)_\nu)+N\eta\,,
\end{equation*}
thus the subadditivity of $\mu_\nu$ follows from~\eqref{def:subad-process-g} and the arbitrariness of $\eta>0$.

\medskip

\textit{Step 4: boundedness.} Let $\omega\in \Omega$ and $I\in\I_{n-1}$. Then \eqref{def:interval-n-dim} and \eqref{1dim-energy-bis} readily imply
\begin{equation*}
0\le\mu_\nu(\omega, I)=\frac{1}{M_\nu^{n-1}}\m^s_\omega(\bar u_0^\nu,I_\nu)\le c_4 C_\vv\L^{n-1}(I).
\end{equation*}
\end{proof}
Having at hand Proposition~\ref{prop:subadditive}, with the help of Lemma~\ref{lem:cubes-shift} and Lemma~\ref{lem:cubes-rotation} we now prove the following result, which establishes the almost sure existence of the limit defining $g_\hom$ when $x=0$.
\begin{proposition}[Homogenised surface integrand for $x=0$]\label{prop:ex-limit-zero}
Let $g$ satisfy \ref{meas_g}-\ref{ginG} and assume that it is stationary with respect to a group $(\tau_z)_{z\in\Z^n}$ of $P$-preserving transformations on $(\Omega,\T,P)$. For $\omega\in\Omega$ let $\m^s_\omega$ be as in~\eqref{ms-omega}. Then there exist $\tilde{\Omega}\in\mathcal T$ with $P(\tilde{\Omega})=1$ and a $(\T\otimes \mathcal B(\Sph^{n-1}))$-measurable function $g_{\hom}:\Omega\x\Sph^{n-1}\to[0,+\infty)$ such that
\begin{equation}\label{eq:ex-limit-zero}
\lim_{r\to+\infty}\frac{\m^s_\om(\bar{u}_{0}^\nu,Q_{r}^\nu(0))}{r^{n-1}}=g_\hom(\om,\nu)
\end{equation}
for every $\omega\in\tilde{\Omega}$ and every $\nu\in\Sph^{n-1}$. Moreover, $\tilde \Omega$ and $g_\hom$ are $(\tau_z)_{z\in \Z^n}$-translation invariant\ie $\tau_{z}(\tilde{\Omega})=\tilde{\Omega}$ for every $z\in \Z^n$ and
\begin{equation}\label{eq:shift-invariance}
g_\hom(\tau_z(\omega),\nu)=g_\hom(\omega,\nu),
\end{equation}
for every $z\in \Z^n$, for every $\omega\in\tilde{\Omega}$, and every $\nu\in\Sph^{n-1}$. Therefore, if $(\tau_z)_{z\in\Z^n}$ is ergodic then $g_{\hom}$ is independent of $\omega$ and given by
\begin{equation}\label{c:g-hom-er}
g_{\rm hom}(\nu)=
\lim_{r\to +\infty} \frac{1}{r^{n-1}}\int_\Omega \m^s_\omega(\bar u^\nu_0,Q^\nu_r(0))\dP(\omega).
\end{equation}
\end{proposition}
\begin{proof}
We divide the proof into three steps.

\medskip

\textit{Step 1: existence of the limit for $\nu\in\Sph^{n-1}\cap\Q^n$.}
	Let $\nu\in\Sph^{n-1}\cap\Q^n$ be fixed. Thanks to Proposition~\ref{prop:subadditive} we can apply the Subadditive Ergodic  Theorem \cite[Theorem 2.4]{AK81} to the subadditive process $\mu_\nu$ defined on $(\Omega,\T,P)$ by \eqref{def:subad-process-g}.
Then choosing $I=[-1,1)^{n-1}$ we get the existence of a set $\Omega_\nu \in \mathcal T$, with $P(\Omega_\nu)=1$, and of a $\mathcal T$-measurable function $g_\nu \colon \Omega \to [0,+\infty)$ such that
\[
\lim_{j\to +\infty}\frac{{\mu_\nu}(\om, j I)}{(2j)^{n-1}} = g_\nu (\om),   
\]
for every $\om \in \Omega_\nu$. By \eqref{def:subad-process-g}, since $I_\nu=2M_\nu Q^\nu(0)$ this yields 
\begin{equation}\label{c:lim-integer}
g_\nu (\om)=\lim_{j\to +\infty}\frac{\m^s_\om(\bar u_0^\nu, j2M_\nu Q^\nu(0))}{(j2M_\nu)^{n-1}}\,,
\end{equation}
for every $\om \in \Omega_\nu$. 
Let $(r_j)$ be a sequence of strictly positive real numbers with $r_j \to +\infty$, as $j \to +\infty$, and consider the two sequences of integers defined as follows: 
\begin{equation*}
r_j^-\defas 2M_\nu\Big(\Big\lfloor \frac{r_j}{2M_\nu}\Big\rfloor-1\Big)\quad\text{and}\quad r_j^+\defas 2M_\nu\Big(\Big\lfloor \frac{r_j}{2M_\nu}\Big\rfloor+2\Big)\,.
\end{equation*}
Let $j \in N$ be such that $r_j>4(1+M_\nu)$, and thus $r_j^->4$. We clearly have 
\[  
Q^\nu_{r_j^-+2}(0)\wcont Q_{r_j}^\nu(0)\wcont Q_{r_j+2}^\nu(0)\wcont
Q_{r_j^+}^\nu(0)\,,
\]
hence we can apply Lemma~\ref{lem:cubes-shift} with $x=\tilde x=0$ first choosing $r=r_j^-$ and $\tilde r=r_j$ to get
\begin{equation}\label{up}
\frac{\m^s_\omega(\bar{u}_0^\nu,Q_{r_j}^\nu(0))}{r_j^{n-1}}\le 
\frac{\m^s_\omega(\bar{u}_{0}^\nu,Q_{r_j^-}^\nu(0))}{(r_j^-)^{n-1}}
+
\frac{L(r_j-r_j^-+1)}{r_j}
\,,
\end{equation}
and then $r=r_j$ and $\tilde r= r_j^+$ to obtain
\begin{equation}\label{low}
\frac{\m^s_\omega(\bar{u}_{0}^\nu,Q_{r_j}^\nu(0))}{r_j^{n-1}}\ge
\frac{\m^s_\omega(\bar{u}_0^\nu,Q_{r_j^+}^\nu(0))}{(r_j^+)^{n-1}}-
\frac{L(r_j^+-r_j+1)}{r_j}\,.
\end{equation} 
Clearly, $r_j^+-r_j\leq 4M_\nu$ and $r_j-r_j^-\leq 4M_\nu$,
thus, if $\om \in \Omega_\nu$, thanks to \eqref{c:lim-integer}, passing to the limsup in as $j \to +\infty$ in \eqref{up} gives 
\begin{equation}\label{c:ls-real}
\limsup_{j \to +\infty} \frac{\m^s_\omega(\bar{u}_{0}^\nu,Q_{r_j}^\nu(0))}{r_j^{n-1}} \leq g_\nu(\om),
\end{equation}
while passing to the liminf in as $j \to +\infty$ in \eqref{low} yields
\begin{equation}\label{c:li-real}
\liminf_{j \to +\infty} \frac{\m^s_\omega(\bar{u}_{0}^\nu,Q_{r_j}^\nu(0))}{r_j^{n-1}} \geq g_\nu(\om). 
\end{equation}
Hence gathering \eqref{c:ls-real} and \eqref{c:li-real} gives for every $\om \in \Omega_\nu$ the existence of the limit along $(r_j)$ together with the equality
\begin{equation*}
\lim_{j \to +\infty} \frac{\m^s_\omega(\bar{u}_{0}^\nu,Q_{r_j}^\nu(0))}{r_j^{n-1}}=g_\nu(\om).
\end{equation*}
Therefore setting 
\[
\tilde\Omega:=\hspace{-0.3cm}\bigcap_{\nu\in\S^{n-1}\cap\Q^n}\hspace{-0.3cm}\Omega_\nu,
\]  
clearly gives $\tilde \Omega \in \mathcal T$ as well as $P(\tilde\Omega)=1$; moreover we have
\begin{equation*}
g_\nu(\omega)=\lim_{r\to+\infty}\frac{\m^s_\omega(\bar u_0^\nu,Q_{r}^\nu(0))}{r^{n-1}}
\end{equation*}
for every $\omega\in\tilde\Omega$ and every $\nu\in\S^{n-1}\cap\Q^n$.

\medskip

\textit{Step 2: existence of the limit for $\nu\in\Sph^{n-1}\sm\Q^n$.}
Consider the two functions $\underline{g}, \overline{g} \colon \tilde \Omega \times \Sph^{n-1} \to [0,+\infty]$ defined as
\begin{equation*}
\underline{g}(\omega,\nu)\defas\liminf_{r\to+\infty}\frac{\m^s_\omega(\bar{u}_0^\nu,Q_r^\nu(0))}{r^{n-1}}\,,\quad \overline{g}(\omega,\nu)\defas\limsup_{r\to+\infty}\frac{\m^s_\omega(\bar{u}_0^\nu,Q_r^\nu(0))}{r^{n-1}}.
\end{equation*}
By the previous step we have that $\underline{g}(\omega,\nu)=\overline{g}(\omega,\nu)=g_\nu(\om)$, for every $\om \in \tilde \Omega$ and for every $\nu \in \Sph^{n-1}\cap\Q^n$. Therefore, if we show that the restrictions of the functions $\nu\mapsto\underline g(\omega,\nu)$ and $\nu\mapsto\overline g(\omega,\nu)$ to the sets $\widehat{\S}_\pm^{n-1}$ are continuous, by the density of 
 $\widehat{\Sph}_\pm^{n-1}\cap\Q^n$ in $\widehat{\Sph}_\pm^{n-1}$ we can readily deduce that
$\underline g(\omega,\nu)=\overline g(\omega,\nu)=g_\nu(\om)$ for every $\omega\in\tilde{\Omega}$ and every $\nu \in \Sph^{n-1}$, and thus the claim. 

Then it remains to show that $\underline g(\omega,\cdot)$ and $\overline g(\omega,\cdot)$ are continuous in $\widehat\S^{n-1}_\pm$. 
We only prove that $\overline g(\omega,\cdot)$ is continuous in $\widehat\S^{n-1}_+$, the other proofs being analogous. 
To this end, let $\nu\in\widehat{\S}_+^{n-1}$, $(\nu_j)\subset\widehat{\S}_+^{n-1}$ be such that $\nu_j\to\nu$, as $j \to +\infty$. Then, for every $\alpha\in (0,\frac{1}{2})$ there exists $j_\alpha \in \N$ such that the assumptions of Lemma~\ref{lem:cubes-rotation} are satisfied for every $j\geq j_\alpha$. 
Thus, choosing $x=0$ and $\tilde \nu =\nu_j$ in \eqref{est:cubes-rotation} we obtain
\begin{equation*}
\m^s_\om(\bar u^{\nu_j}_0, Q^{\nu_j}_{(1+\alpha)r}(0)) -c_\alpha r^{n-1} \leq \m^s_\om(\bar u^{\nu}_0, Q^{\nu}_{r}(0)) \leq \m^s_\om(\bar u^{\nu_j}_0, Q^{\nu_j}_{(1-\alpha)r}(0)) +c_\alpha r^{n-1},
\end{equation*}
where $c_\alpha \to 0$, as $\alpha \to 0$. 
Therefore, by definition of $\overline g$ passing to the limsup as $r \to +\infty$ we get the two following inequalities 
 \begin{align}\label{c:eins}
(1+\alpha)^{n-1}\,\overline{g}(\omega,\nu_j) 
&\le \overline g(\omega,\nu)+c_\alpha\,,\\\label{c:zwei}
(1-\alpha)^{n-1}\,\overline{g}(\omega,\nu_j)&\geq \overline g(\omega,\nu)-c_\alpha\,.
 \end{align}
Passing to the limsup as $j\to+\infty$ in \eqref{c:eins} and to the liminf as $j\to+\infty$ in \eqref{c:zwei} and eventually letting $\alpha\to0$ we infer
\begin{equation*}
	\limsup_{j\to+\infty}\overline{g}(\omega,\nu_j)\le\overline g(\omega,\nu)\leq\liminf_{j\to+\infty}\overline{g}(\omega,\nu_j)\,,
\end{equation*}
and hence the claim. 

For every $\om \in \Omega$ and $\nu \in \Sph^{n-1}$ set
\begin{equation}\label{c:def-g-hom}
g_{\hom}(\om,\nu):=\begin{cases} \overline g(\om,\nu) & \text{if }\; \om \in \tilde \Omega, 
\cr
c_2 c_p & \text{if }\; \om \in \Omega \setminus \tilde \Omega, 
\end{cases}
\end{equation}
then, for every $\om \in \tilde \Omega$ and every $\nu \in \Sph^{n-1}$ we have
\begin{equation}\label{c:eventually}
g_\hom(\om, \nu)=\lim_{r\to+\infty}\frac{\m^s_\omega(\bar u_0^\nu,Q_{r}^\nu(0))}{r^{n-1}}.
\end{equation}
Moreover 
\[
\om \mapsto \overline g (\om,\nu) \; \text{ is $\mathcal T$-measurable in $\tilde \Omega$, for every $\nu \in \Sph^{n-1}$} 
\]
and 
\[
\nu \mapsto \overline g (\om,\nu) \; \text{ is continuous in $\widehat \Sph^{n-1}_\pm$, for every $\om \in \tilde \Omega$}.  
\]
Therefore the restriction of $\overline g$ to $\tilde \Omega \times \widehat \Sph^{n-1}_\pm$ is measurable with respect to the $\sigma$-algebra induced in $\tilde \Omega \times \widehat \Sph^{n-1}_\pm$ by $\mathcal T \otimes \mathcal B(\Sph^{n-1})$ thus, finally, 
$g_\hom$ is $(\mathcal T \otimes \mathcal B(\Sph^{n-1}))$-measurable on $\Omega \times \Sph^{n-1}$. 
\EEE

\medskip

\textit{Step 3: $(\tau_z)_{z\in \Z^n}$-translation invariance.}
Let $z\in\Z^n$, $\omega\in\tilde\Omega$, and $\nu\in\S^{n-1}$ be fixed.  Let $r>4$ \EEE
and $ v\in \Adm(\bar u^\nu_0,Q^\nu_r(0))$ with
\begin{equation}\label{c:quasi-min-s}
\F^s(\omega)(v,Q_r^\nu(0))\le\m^s_\omega(\bar u^\nu_0,Q^\nu_r(0))+1\,. 
\end{equation}
Set $\tilde v(y):=v(y+z)$;  by the stationarity of $g$ \EEE we have
\begin{equation*}
\F^s(\omega)(v,Q_r^\nu(0))=\F^s(\tau_z(\omega))(\tilde v,Q_r^\nu(-z)),
\end{equation*}
further, since $\tilde v\in\Adm(\bar u^\nu_{-z},Q^\nu_r(-z))$, by \eqref{c:quasi-min-s}
we immediately get
\begin{equation}\label{shift_prob1}
	\m^s_{\tau_z(\omega)}(\bar u^\nu_{-z},Q^\nu_r(-z))\le\m^s_\omega(\bar u^\nu_0,Q^\nu_r(0))+1\,. 
\end{equation}
Now let $r, \tilde r$ be such that $\tilde r> r$ and
\begin{equation*}
	Q^\nu_{r+2}(-z)\wcont Q^\nu_{\tilde r}(0)\; \text{ and }\;  
\dist(0,\Pi^\nu(-z))\le\frac r4,
\end{equation*}
therefore  applying Lemma \ref{lem:cubes-shift} with $x=-z$ and $\tilde x=0$ we may deduce the existence of a constant $L>0$ such that \EEE
\begin{equation}\label{shift_prob2}
\m^s_{\tau_z(\omega)}(\bar{u}_{0}^\nu,Q_{\tilde{r}}^\nu(0)) \leq\m^s_{\tau_z(\omega)}(\bar{u}_{-z}^\nu,Q_{r}^\nu(-z))+L\big(|z|+|r-\tilde{r}|+1\big)(\tilde{r})^{n-2}\,.
\end{equation}
Hence, gathering \eqref{shift_prob1} and \eqref{shift_prob2} we obtain
\begin{equation}\label{c:uno-ti}
	\frac{\m^s_{\tau_z(\omega)}(\bar{u}_{0}^\nu,Q_{\tilde{r}}^\nu(0))}{\tilde r^{n-1}}\le
	\frac{\m^s_{\omega}(\bar u^\nu_0,Q^\nu_r(0))+1}{r^{n-1}}+\frac{L\big(|z|+|r-\tilde{r}|+1\big)}{\tilde r}\,.
\end{equation}
An analogous argument, now replacing $\omega$ with $\tau_z(\omega)$ and $z$ with $-z$, yields 
\begin{equation}\label{c:due-ti}
\frac{\m^s_\omega(\bar{u}_{0}^\nu,Q_{\tilde{r}}^\nu(0))}{\tilde r^{n-1}}\le
\frac{\m^s_{\tau_z(\omega)}(\bar u^\nu_0,Q^\nu_r(0))+1}{r^{n-1}}+\frac{L\big(|z|+|r-\tilde{r}|+1\big)}{\tilde r}\,.
\end{equation}

Taking in \eqref{c:uno-ti} the limsup as $\tilde r\to+\infty$ and the limit as $r\to+\infty$ gives
\begin{equation}\label{c:ti-1}
\limsup_{\tilde r\to+\infty}	\frac{\m^s_{\tau_z(\omega)}(\bar{u}_{0}^\nu,Q_{\tilde{r}}^\nu(0))}{\tilde r^{n-1}}\le g_\hom(\omega,\nu),
\end{equation}
while taking in \eqref{c:due-ti} the limit as $\tilde r\to+\infty$ and the liminf as $r\to+\infty$ entails
\begin{equation}\label{c:ti-2}
g_\hom(\omega,\nu)\le \liminf_{r\to+\infty}\frac{\m^s_{\tau_z(\omega)}(\bar{u}_{0}^\nu,Q_{r}^\nu(0))}{r^{n-1}}.
\end{equation}
Eventually, by combining \eqref{c:ti-1} and \eqref{c:ti-2} we both deduce that
$\tau_z(\omega)\in\tilde\Omega$ and that
\[
g_\hom(\tau_z(\omega),\nu)=g_\hom(\omega,\nu)\,,
\]
for every $z\in \Z$, $\omega\in\tilde{\Omega}$, and $\nu\in\Sph^{n-1}$.
Then, we observe that thanks to the group properties of $(\tau_z)_{z\in\Z^n}$ we also have that $\om \in \tau_z(\tilde \Omega)$, for every $z\in \Z^n$. Indeed we have $\om = \tau_{z}(\tau_{-z}(\om))$ and $\tau_{-z}(\om) \in \tilde \Omega$. 

If $(\tau_z)_{z\in\Z^n}$ is ergodic, then the independence of $\omega$ of the function of $g_\hom$ is a direct consequence of~\eqref{eq:shift-invariance} (cf.~\cite[Corollary 6.3]{CDMSZ19a}). Furthermore, \eqref{c:g-hom-er} can be obtained by integrating \eqref{c:eventually} on $\Omega$ and using the Dominated Convergence Theorem, thanks to \eqref{c:bd-mu} (see also \eqref{def:subad-process-g}).
\end{proof}
The following result is of crucial importance in our analysis as it extends Proposition \ref{prop:ex-limit-zero} to the case of an arbitrary $x\in \R^n$. More precisely, Proposition \ref{prop:ex-limit-x} below establishes the existence of the limit in \eqref{eq:ex-limit-zero} when $x=0$ is replaced by any $x\in \R^n$; moreover it shows that this limit is $x$-independent, and hence it coincides with \eqref{eq:ex-limit-zero}. 

The proof of the following proposition can be obtained arguing exactly as in \cite[Theorem 6.1]{CDMSZ19a} (see also~\cite[Theorem 5.5]{ACR11}), now appealing to Proposition \ref{prop:ex-limit-zero}, Lemma~\ref{lem:cubes-shift}, and Lemma~\ref{lem:cubes-rotation}. For this reason we omit its proof here. 

\begin{proposition}[Homogenised surface integrand]\label{prop:ex-limit-x}
Let $g$ satisfy \ref{meas_g}-\ref{ginG} and assume that it is stationary with respect to a group $(\tau_z)_{z\in\Z^n}$ of $P$-preserving transformations on $(\Omega,\T,P)$. For $\omega\in\Omega$ let $\m^s_\omega$ be as in~\eqref{ms-omega}. Then there exists $\Omega'\in \T$ with $P(\Omega')=1$ such that 
\begin{equation}\label{eq:ex-limit-x}
\lim_{r\to+\infty}\frac{\m^s_\omega(\bar{u}_{rx}^\nu,Q_{r}^\nu(rx))}{r^{n-1}}=g_\hom(\om,\nu)
\end{equation}
for every $\omega\in\Omega'$, every $x\in \R^n$, and every $\nu\in\Sph^{n-1}$, where $g_\hom$ is given by~\eqref{eq:ex-limit-zero}. In particular, the limit in~\eqref{eq:ex-limit-x} is independent of $x$. Moreover, if $(\tau_z)_{z\in\Z^n}$ is ergodic, then $g_\hom$ is independent of $\omega$ and given by~\eqref{c:g-hom-er}.
\end{proposition}
We conclude this section with the proof of Theorem \ref{thm:stoch_hom_2}.
\begin{proof}[Proof of Theorem \ref{thm:stoch_hom_2}]
The proof follows by Theorem~\ref{thm:hom} now invoking Proposition \ref{prop:cell-volume} and Proposition \ref{prop:ex-limit-x}.
\end{proof}

\section*{Acknowledgments}
\noindent The authors wish to thank Manuel Friedrich and Matthias Ruf for fruitful discussions. 
The work of R. Marziani and C. I. Zeppieri was supported by the Deutsche Forschungsgemeinschaft (DFG, German Research Foundation) project 3160408400 and under the Germany Excellence Strategy EXC 2044-390685587, Mathematics M\"unster: Dynamics--Geometry--Structure. The work of A.\ Bach was supported by the DFG Collaborative Research Center TRR 109, ``Discretization in Geometry and Dynamics''.

\appendix
\section*{Appendix}\label{appendix}
\setcounter{theorem}{0}
\setcounter{equation}{0}
\renewcommand{\theequation}{A.\arabic{equation}}
\renewcommand{\thetheorem}{A.\arabic{theorem}}

\noindent In this last section we state and prove two technical lemmas which are used in Subsection \ref{sect:cell-formulas}.

\medskip

For $A \in \mathcal A$, $x\in \R^n$, and $\nu \in \Sph^{n-1}$, in what follows $\m^s(\bar u^\nu_x, A)$ denotes the infimum value given by \eqref{eq:m-s-bis}. 

\begin{lemma}\label{lem:cubes-shift}
Let $g \in \mathcal G$; let $\nu\in\Sph^{n-1}$, $x,\tilde{x}\in\R^n$, and $\tilde{r}> r>4$ be such that
\begin{equation*}
{\rm (i)}\ Q_{r+2}^\nu(x)\wcont Q_{\tilde{r}}^\nu(\tilde x)\,,\qquad {\rm (ii)}\ \dist(\tilde x,\Pi^\nu(x))\leq\frac{r}{4}\,.
\end{equation*}
Then there exists a constant $L>0$ (independent of $\nu,x,\tilde{x},r,\tilde{r}$) such that
\begin{equation}\label{eq:cubes-shift}
\m^s(\bar{u}_{\tilde{x}}^\nu,Q_{\tilde{r}}^\nu(\tilde x))\leq\m^s(\bar{u}_{x}^\nu,Q_{r}^\nu(x))+L\big(|x-\tilde x|+|r-\tilde{r}|+1\big)\tilde{r}^{n-2}\,.
\end{equation}
\end{lemma}
\begin{proof}

To prove \eqref{eq:cubes-shift} we are going to perform a construction which is similar to that used in Proposition~\ref{p:equivalence-N-D}. 

Let $\nu\in S^{n-1}$, $\eta>0$ be fixed and let $ v\in\Adm(\bar{u}^\nu_x,Q^\nu_r(x))$  with
\begin{equation}\label{est:energy-shift}
\F^s(v,Q^\nu_{r}(x))\leq\m^s(\bar{u}^\nu_{x},Q^\nu_{r}(x))+\eta\,.
\end{equation}
Let $u\in W^{1,p}(Q^\nu_{r}(x);\R^m)$ correspond to the $v$ as above\ie $v \nabla u=0$ a.e.\ in $Q^\nu_{r}(x)$ and $(u,v)=(\bar{u}^\nu_{x},\bar{v}^\nu_{x})$ a.e.\ in $U$, where $U$ is a neighbourhood of $\partial Q^\nu_{r}(x)$. Let moreover $\beta\in(0,1)$ be such that $Q^\nu_{r}(x)\sm\overline{Q}^\nu_{r-\beta}(x)\subset U$. Let $\tilde x \in \R^n$ and $\tilde r>r>4$ satisfy (i)-(ii). We set
\begin{equation*}
R\defas R_\nu\biggl( \big(Q'_{r}\sm\overline{Q'}_{r-\beta}\big)\x\Big(-1-\frac{|(x- \tilde{x})\cdot\nu|}{2}\,,1+\frac{|(x-\tilde{x})\cdot\nu|}{2}\Big)\biggr)+x+\frac{(\tilde x-x)\cdot\nu}{2} \nu\,,
\end{equation*}
where $R_\nu$ be as in \ref{Rn} (see Figure~\ref{fig:ShiftedCubes}). 
Then let $\varphi \in C_c^\infty(Q_{r}')$ be a cut-off function between $Q'_{r-\beta}$ and $Q'_{r}$\ie $0\leq \varphi \leq 1$, and 
  $\varphi\equiv 1$ on $Q_{r-\beta}'$. Eventually, for $y=(y',y_n)\in Q_{\tilde r}^\nu(\tilde x)$ we define the pair
  $(\tilde u,\tilde v)$ by setting
\begin{equation*}
\tilde u(y):=\varphi((R_\nu^T(y- x))')u(y)+(1-\varphi((R_\nu^T(y- x))'))\bar u_{\tilde x}^\nu(y)\,,
\end{equation*}
\begin{equation*}
\tilde v(y):=\begin{cases}	\min\{v(y),d(y)\}&\text{in}\ Q^\nu_{r}(x)\,,\\
\min\{\bar v^\nu_{\tilde x}(y),d(y)\}&\text{in}\ Q^\nu_{\tilde{r}}(\tilde x)\sm\overline{Q}^\nu_{r}(x)\,,
\end{cases}
\end{equation*}
where $d(y)\defas\dist(y,R)$. Clearly, $\tilde u\in W^{1,p}(Q^\nu_{\tilde{r}}(\tilde x);\R^m)$, moreover, by construction we have
\begin{equation*}
|(y-x)\cdot\nu|>1\;\text{ and }\;|(y-\tilde x)\cdot\nu|>1\;\text{ for every }\; y\in (Q_r^\nu(x)\sm\overline{Q}_{r-\beta}^\nu(x))\sm\overline{R}\,.
\end{equation*}
Hence, the boundary conditions satisfied by $v$ imply that 
\[
v=\bar{v}^\nu_{x}=\bar{v}^\nu_{\tilde x}=1 \; \text{ in }\; (Q^\nu_r(x)\sm\overline{Q}^\nu_{r-\beta}(x))\sm \overline{R}\,,
\] 
thus $\tilde{v}\in W^{1,p}(Q^\nu_{\tilde{r}}(\tilde x))$. 

Thanks to (ii) and to the fact that $r>4$ we have $R\subset Q^\nu_{r}(x)$ so that $\{d<1\}\subset Q^\nu_{r+2}(x)$. Therefore, in view of (i) we get that $\tilde{v}=\bar{v}^\nu_{\tilde{x}}$ in a neighbourhood of $\partial Q^\nu_{\tilde{r}}(\tilde{x})$, further $\tilde{u}=\bar{u}^\nu_{\tilde x}$ in $Q^\nu_{\tilde{r}}(\tilde x)\sm\overline{Q}^\nu_r(x)$. Then, to show that $\tilde{v}$ is admissible for $\m^s(\bar{u}^\nu_{\tilde x},Q^\nu_{\tilde{r}}(\tilde x))$ it only remains to check that $\tilde{v}\,\nabla\tilde{u}=0$ a.e.\ in $Q^\nu_{\tilde{r}}(\tilde x)$. Clearly, $\tilde{v}\,\nabla\tilde{u}=0$ a.e.\ in $R$. On the other hand, in $Q^\nu_r(x)\sm\overline{R}$ we have $\tilde{v} |\nabla\tilde{u}| \leq v|\nabla u|=0$, while in $\big(Q^\nu_{\tilde{r}}(\tilde x)\sm Q^\nu_r(x)\big)\sm\overline{R}$ we get $\tilde{v}|\nabla\tilde{u}|\leq \bar{v}^\nu_{\tilde{x}}|\nabla\bar{u}^\nu_{\tilde{x}}|=0$, and hence the claim.
 
Eventually, arguing as in Proposition~\ref{p:equivalence-N-D} we obtain
\begin{equation}\label{est:modification-shift}
\begin{split}
\F^s\big(\tilde{v},Q^\nu_{\tilde{r}}(\tilde x)\big) &\leq\F^s(v,Q^\nu_r(x))+\F^s\big(\bar{v}^\nu_{\tilde x},Q^\nu_{\tilde{r}}(\tilde x)\sm\overline{Q}^\nu_r(x)\big)+2c_4\L^n(\{d<1\})\,.
\end{split}
\end{equation}
Thanks to~\eqref{g-value-at-0} and~\eqref{1dim-energy-bis} we have
\begin{equation}\label{c:c0}
\F^s\big(\bar{v}^\nu_{\tilde x},Q^\nu_{\tilde{r}}(\tilde x)\sm\overline{Q}^\nu_r(x)\big)\leq c_4C_{\vv}\mathcal{L}^{n-1}\big(Q'_{\tilde{r}}\sm\overline{Q}'_r\big)\leq C|r-\tilde{r}|\tilde{r}^{n-2},
\end{equation}
moreover 
\begin{equation}\label{c:c1}
\L^n(\{d<1\})\leq C(|x-\tilde x|+1)(\beta+1) r^{n-2}.
\end{equation} 
We observe that both in \eqref{c:c0} and \eqref{c:c1} the positive constant $C>0$ does not depend on any of the parameters and on $\nu$. Therefore, gathering~\eqref{est:energy-shift} and~\eqref{est:modification-shift} we finally obtain
\begin{equation*}
\m^s(\bar{u}^\nu_{\tilde x},Q^\nu_{\tilde{r}}(\tilde x))\leq\F^s\big(\tilde{v},Q^\nu_{\tilde{r}}(\tilde x)\big)\leq\m^s(\bar{u}^\nu_{x},Q^\nu_r(x))+L\big(|x-\tilde x|+|r-\tilde{r}|+1\big)\tilde{r}^{n-2}+\eta\,,
\end{equation*}
for some $L>0$ independent of $x,\tilde x, r, \tilde r, \nu$, thus \eqref{eq:cubes-shift} follows by the arbitrariness of $\eta>0$. 
\end{proof}
\begin{figure}[H]
\begin{subfigure}[t]{0.49\textwidth}
\centering\captionsetup{width=.9\linewidth}
\def\svgwidth{0.7\columnwidth}
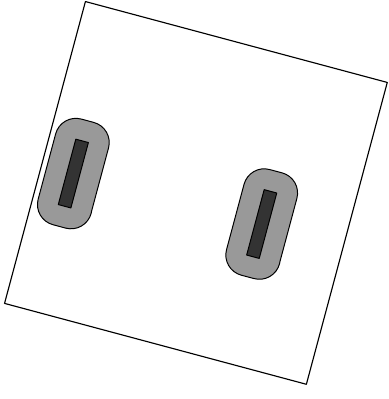
\caption{ The sets $Q_{r-\beta}^\nu(x)$, $Q_{r}^\nu(x)$, $Q_{\tilde{r}}^\nu(\tilde x)$ and in gray the sets $R$ (dark gray) and $\{d<1\}$ (light gray).}
\label{fig:ShiftedCubes}
\end{subfigure}
\begin{subfigure}[t]{0.49\textwidth}
\centering\captionsetup{width=.9\linewidth}
\def\svgwidth{0.85\columnwidth}
\input{RotatedCubes.pdf_tex}
\caption{The sets $Q_{r-\beta}^\nu(rx)$, $Q_{r}^\nu(rx)$, $Q_{(1+\alpha)r}^{\tilde{\nu}}(rx)$ and in gray the sets $R$ (dark gray) and $\{d<1\}$ (light gray).}
\label{fig:RotatedCubes}
\end{subfigure}
\caption{The sets used in the construction of $(\tilde{u},\tilde{v})$ in Lemma~\ref{lem:cubes-shift} and Lemma~\ref{lem:cubes-rotation}.}
\end{figure}
\begin{lemma}\label{lem:cubes-rotation}
Let $g \in \mathcal G$; let $\alpha\in (0,\frac{1}{2})$ and $\nu,\tilde{\nu}\in\Sph^{n-1}$ be such that 
\begin{equation}\label{c:30-condition}
\max_{1\leq i\leq n-1}|R_\nu e_i-R_{\tilde \nu} e_i| + |\nu -\tilde \nu| <\frac{\alpha}{\sqrt n},
\end{equation}
where $R_\nu$ and $R_{\tilde \nu}$ are orthogonal $(n\times n)$-matrices as in \ref{Rn}. 
Then there exists a constant $c_\alpha>0$ (independent of $\nu,\tilde{\nu}$), with $c_\alpha\to 0$ as $\alpha\to 0$, such that for every $x\in\R^n$ and every $r>2$ we have
\begin{equation}\label{est:cubes-rotation}
\begin{split}
\m^s\big(\bar{u}_{rx}^{\tilde{\nu}},Q_{(1+\alpha)r}^{\tilde{\nu}}(rx)\big)-c_\alpha r^{n-1} &\leq \m^s\big(\bar{u}_{rx}^\nu, Q_{r}^\nu(rx)\big)\\
&\leq \m^s\big(\bar{u}_{rx}^{\tilde{\nu}},Q_{(1-\alpha)r}^{\tilde{\nu}}(rx)\big)+c_\alpha r^{n-1}\,.
\end{split}
\end{equation}
\end{lemma}
\begin{proof}
We only prove that
\begin{equation}\label{c:30-claim}
\m^s\big(\bar{u}_{rx}^{\tilde{\nu}},Q_{(1+\alpha)r}^{\tilde{\nu}}(rx)\big)-c_\alpha r^{n-1} \leq \m^s\big(\bar{u}_{rx}^\nu, Q_{r}^\nu(rx)\big),
\end{equation}
for some $c_\alpha>0$, with $c_\alpha \to 0$, as $\alpha \to 0$; the proof of the other inequality is analogous.

Let $x\in\R^n$, $r>2$, and set $r_\alpha^\pm\defas (1\pm\alpha)r$. We notice that condition \eqref{c:30-condition} readily implies that
\begin{equation}\label{cond:cube-rotation-center-x} 
Q_{ r_\alpha^-}^{\tilde{\nu}}(rx) \wcont Q_r^\nu(rx)\wcont Q_{ r_\alpha^+}^{\tilde{\nu}}(rx).
\end{equation}
Moreover, we let $r$ be so large that $Q_{r+2}^\nu(rx)\wcont Q_{ r_\alpha^+}^{\tilde{\nu}}(rx)$. This allows us to use a similar argument as in the proofs of Proposition~\ref{p:equivalence-N-D}  and Lemma~\ref{lem:cubes-shift}, which we repeat here for the readers' convenience.  

Let $\eta>0$ be fixed and let $v\in W^{1,p}(Q^\nu_r(rx))$ be a test function for $\m^s(\bar u_{rx}^\nu,Q^\nu_r(rx))$ satisfying 
\begin{equation}\label{cont_v}
\F^s(v,Q^\nu_r(rx))\le \m^s(\bar u_{rx}^\nu,Q^\nu_r(rx))+\eta.
\end{equation}
Therefore we know that there exist $u\in W^{1,p}(Q_r^\nu(rx);\R^m)$ and a neighbourhood $U$ of $\partial Q_r^\nu(rx)$ such that 
\begin{equation}\label{bc:rotation-cubes}
(u,v)=(\bar u_{rx}^\nu,\bar v_{rx}^\nu)\quad\text{in}\quad U\;\text{ and }\; v \nabla u=0\;\text{ a.e.\ in }\; Q_r^\nu(rx)\,.
\end{equation}
Now let $\beta\in(0,1)$ be such that $Q_r^\nu(rx)\setminus\overline{Q}_{r-\beta}^\nu(rx)\subset U$ and set
\begin{equation*}
R:=R_\nu\Big(Q'_{r} \setminus\overline {Q'}_{r-\beta}\x(-1-\alpha r,1+\alpha r)\Big)+rx\,,
\end{equation*}
where $R_\nu$ is as in~\ref{Rn} (see Figure~\ref{fig:RotatedCubes}).
Let $\varphi \in C^\infty_c(Q'_r)$ be a cut-off function between $Q_{r-\beta}'$ and $Q_{r}'$; we define
\begin{align*}
\tilde u(y) &:=\varphi\big(\big(R_\nu^T(y-rx)\big)'\big)u(y)+\Big(1-\varphi\big(\big(R_\nu^T(y-rx)\big)'\big)\Big)\bar u_{ rx}^{\tilde{\nu}}(y)\,,\\[4pt]
	\tilde v(y) &:=\begin{cases}
	\min\{v(y),d(y)\}&\text{in }\; Q_r^\nu(rx)\,,\\[4pt]
	\min\{\bar v^{\tilde{\nu}}_{ rx}(y),d(y)\}&\text{in }\;Q^\nu_{ r_\alpha^+}(rx)\setminus\overline{Q}^\nu_r(rx)\,,
	\end{cases}
\end{align*}
where $d(y):=\dist(y,R)$. We now observe that in view of \eqref{c:30-condition} and~\eqref{bc:rotation-cubes}, by definition of $R$ we get
\[
v=\bar v_{rx}^\nu=\bar{v}_{ rx}^{\tilde{\nu}}=1\; \text{ in }\; (Q_{r}^\nu(rx)\sm\overline{Q}_{r-\beta}(rx))\sm \overline{R},
\]
so that $\tilde v\in W^{1,p}\big(Q_{ r_\alpha^+}^{\tilde{\nu}}(rx)\big)$. Since $Q_{r+2}^\nu(rx)\wcont Q_{ r_\alpha^+}^{\tilde{\nu}}(rx)$, arguing as in Lemma~\ref{lem:cubes-shift} it is easy to check that the pair $(\tilde u,\tilde v) \in W^{1,p}\big(Q_{ r_\alpha^+}^{\tilde{\nu}}(rx);\R^m\big)\times W^{1,p}\big(Q_{ r_\alpha^+}^{\tilde{\nu}}(rx)\big)$ also satisfies 
\begin{equation*}
(\tilde u,\tilde v)=(\bar u_{rx}^{\tilde{\nu}},\bar v_{rx}^{\tilde{\nu}})\;\text{ near }\; \partial Q^{\tilde{\nu}}_{ r_\alpha^+}(rx)\quad\text{and}\quad \tilde v\, \nabla \tilde u =0\; \text{ a.e.\ in }\; Q_{ r_\alpha^+}^{\tilde{\nu}}(rx)\,.
\end{equation*}
Therefore, $\tilde v$ is admissible for $\m^s(\bar u_{ rx}^{\tilde{\nu}},Q^{\tilde{\nu}}_{ r_\alpha^+}(rx))$. Then, using the same arguments  as in Proposition~\ref{p:equivalence-N-D} we deduce
\begin{equation}\label{est:rotation-cubes1}
\begin{split}
\F^s(\tilde v,Q^{\tilde{\nu}}_{ r_\alpha^+}(rx)) &\le \F^s(v,Q_r^\nu(rx))+\F^s(\bar{v}_{ rx}^{\tilde{\nu}},Q_{ r_\alpha^+}^{\tilde{\nu}}(rx)\sm\overline{Q}_{r}^\nu(rx))+2c_4\L^n(\{d<1\})\\
&\leq \F^s(v,Q_r^\nu(rx))+\F^s(\bar{v}_{ rx}^{\tilde{\nu}},Q_{ r_\alpha^+}^{\tilde{\nu}}(rx)\sm\overline{Q}_{ r_\alpha^-}^{\tilde{\nu}}(rx))+2c_4\L^n(\{d<1\})\,,
\end{split}
\end{equation}
where to obtain the second inequality we used the first inclusion in~\eqref{cond:cube-rotation-center-x}. Furthermore by~\eqref{g-value-at-0} and~\eqref{1dim-energy-bis} we have
\begin{align}\nonumber 
\F^s(\bar{v}_{ rx}^{\tilde{\nu}},Q_{ r_\alpha^+}^{\tilde{\nu}}(rx)\sm\overline{Q}_{ r_\alpha^-}^{\tilde{\nu}}(rx))& \leq c_4C_{\vv}\L^{n-1}\big(Q_{ r_\alpha^+}'(rx)\sm Q_{ r_\alpha^-}'(rx)\big)\\ \label{est:rotation-cubes2}
&= c_4C_{\vv} ((1+\alpha)^{n-1}-(1-\alpha)^{n-1})r^{n-1}
\end{align}
whereas  
\begin{equation}\label{est:rotation-cubes3}
\L^n(\{d<1\})\leq C (\beta+1) r^{n-2}\,\alpha r\leq C \alpha r^{n-1},
\end{equation}
for some $C>0$ independent of $x, \nu$, and $r$. 
Thus, gathering~\eqref{cont_v},~\eqref{est:rotation-cubes1},~\eqref{est:rotation-cubes2}, and~\eqref{est:rotation-cubes3}, thanks to~\eqref{bc:rotation-cubes} we infer
\begin{align*}
\m^s(\bar{u}_{ rx}^{\tilde{\nu}},Q_{ r_\alpha^+}^{\tilde{\nu}}(rx)) \leq\m^s(\bar{u}_{rx}^\nu,Q_r^\nu(rx))+c_\alpha r^{n-1} +\eta\,,
\end{align*} 
where $c_\alpha\defas c_4C_{\vv}\big((1+\alpha)^{n-1}-(1-\alpha)^{n-1}\big)+C\alpha$. Eventually, \eqref{c:30-claim} follows by the arbitrariness of $\eta>0$.
\end{proof}
\EEE

\end{document}